%% file: main.tex
\documentclass[5p, times]{elsarticle}

\usepackage{amsmath}
\usepackage{xcolor}
\usepackage{bm}
\usepackage{dcolumn,amsthm}
\usepackage{hyperref}
\newcommand\restr[2]{{
  \left.\kern-\nulldelimiterspace 
  #1 
  \littletaller 
  \right|_{#2} 
  }}

\usepackage{url}
\usepackage{upgreek}
\usepackage{chngcntr}
\usepackage{apptools}
\newcommand{\littletaller}{\mathchoice{\vphantom{\big|}}{}{}{}}
\renewenvironment{proof}[1][Proof]{\noindent\textit{#1. } }{\hfill$\square$}

\newtheorem{lemma}{\sc Lemma}[subsection]
\AtAppendix{\counterwithin{lemma}{subsection}}

\usepackage{amssymb}

\journal{-}

\usepackage{amsmath,amssymb}

\input{preamble.tex}

\begin{document}

\begin{frontmatter}




\title{Navier--Stokes Modelling of Non-Newtonian Blood Flow in Cerebral Arterial Circulation and its Dynamic Impact on Electrical Conductivity in a Realistic Multi-Compartment Head Model}


\author[inst1]{Maryam Samavaki\corref{cor1}}
\ead{maryamolsadat.samavaki@tuni.fi}
\cortext[cor1]{Corresponding author at: Sähkötalo building, Korkeakoulunkatu 3, Tampere, 33720, FI}

\affiliation[inst1]{organization={Mathematics, Computing Sciences, Tampere University},
            addressline={Korkeakoulunkatu 1}, 
            city={Tampere University},
            postcode={33014}, 
            country={Finland}}

            \affiliation[inst2]{organization={Faculty of  Mathematics, K. N. Toosi University of Technology},
            addressline={Mirdamad Blvd, No. 470}, 
            city={Tehran},
            postcode={1676-53381}, 
            country={Iran}}

\author[inst2]{Arash Zarrin nia}

\author[inst1]{Santtu Söderholm}
 
\author[inst1]{Sampsa Pursiainen}

\begin{abstract}
\textbf{Background and Objective:} This study aims to show the feasibility of evaluating the dynamic effect of non-Newtonian cerebral arterial circulation on the electrical conductivity distribution inside a realistic multi-compartment head model. This research question is both important and challenging considering various electrophysiological modalities, including transcranial electrical stimulation (tES), electro-/magnetoencephalography (EEG/MEG), and electrical impedance tomography (EIT). This is due to the significant impact of electrical conductivity on forward modelling accuracy, as well as the complex nature of tortuous vessel networks, limitations of data acquisition, particularly in Magnetic Resonance Imaging (MRI), and various non-linear blood flow phenomena, such as the relationship between shear rate and viscosity in non-Newtonian fluid.

\noindent \textbf{Methods:} As a governing flow model, we use the Navier--Stokes equations (NSEs) of the non-Newtonian flow. Our solver is subdivided into two stages. The first one solves the pressure field via a dynamical form of the pressure-Poisson equation, which follows from the original Navier--Stokes system. The second one performs a Leray-regularized update of the velocity field using the pressure distribution of the first stage. We establish the connection between blood velocity and viscosity using the Carreau-Yasuda model. The blood concentration in the microvessels of the brain tissues is approximated by Fick's law of diffusion, and the resulting conductivity mapping is obtained via Archie's law, i.e., an effective conductivity estimates for a mixture of fluid and porous tissue. The head model of the numerical experiments corresponds to an open 7 Tesla MRI dataset, which distinguishes the arterial vessels among other tissue structures.

\noindent \textbf{Results:} The results suggest that a dynamic model of CBF for both arterial and microcirculation can be established; we find blood pressure and conductivity distributions given a numerically simulated pulse sequence and approximate the blood concentration and conductivity inside the brain based on those.

\noindent \textbf{Conclusions:} Our model provides an approximation of the dynamical blood flow and the corresponding electrical conductivity distribution in the different parts of the brain. The advantage of our approach is that it is applicable with limited {\em a priori} information about the blood flow and with an arbitrary head model distinguishing the arteries.

\end{abstract}





\end{frontmatter}


\section{Introduction}
\label{sec:introduction}

This study concerns the dynamic effects of cerebral blood flow (CBF) and the influence of those on the electrical conductivity distribution of the brain. CBF is assumed to follow NSEs in the arterial circulation and a diffusion process in microcirculation, and the conductivity is found as a mixture \cite{j2001estimation, peters2005electrical, glover2000modified, cai2017electrical} between estimated volumetric blood concentration and porous tissues with constant conductivity values \cite{dannhauer2010}. The importance of the present focus is highlighted by the central role of the electrical conductivity in several electrophysiological modalities such as EEG \cite{niedermeyer2004},  tES \cite{herrmann2013transcranial}, and EIT \cite{cheney1999electrical,moura2021anatomical,lahtinen2023silico}.  Recently, solving NSEs for a simplified cylindrical vessel geometry has been proposed as a technique for forming an atlas of the electrical conductivity \cite{moura2021anatomical,lahtinen2023silico}. In this study, we aim to show that a geometrical restriction is not necessary {\em per se}, but that a dynamic atlas can be formed directly for a given realistic multi-compartment head model.

To this end, we present a dynamic solution for arbitrary non-Newtonian blood flow and implement it numerically using the finite element method \cite{svavcek2008approximation,pacheco2021continuous}. In non-Newtonian flow, viscosity is a non-constant parameter that depends on the shear rate tensor determined by the velocity field and, thereby, needs to be taken into account in the formulation of the NSEs to obtain an appropriate approximation of the blood flow. The importance of the non-Newtonian blood flow phenomenon is highlighted by the important role of viscosity in cardiovascular diseases, for example, in stenosis \cite{liu2021comparison} and blood clotting \cite{ouared2005lattice,wajihah2022effects}. In this study, we consider the Carreau-Yasuda viscosity model  \cite{johnston2004non, johnston2006non, cho1991effects} which assumes that the viscosity is a smooth function of the strain tensor with two asymptotic values for the cases where the velocity tends to zero or infinity.

We use the NSEs of the non-Newtonian flow as the governing flow model. Our solver is subdivided into two stages, the first of which solves the pressure field via a dynamical form of PPE, which follows from the original Navier--Stokes system via taking a sidewise divergence \cite{pacheco2021continuous}. In the second stage, this pressure field is substituted into Leray-regularized NSEs to obtain the velocity field \cite{leray1934mouvement, pietarila2008three, guermond2004definition}. The incoming pressure is assumed to be given {\em a priori}. We consider this approach advantageous in practical applications where measuring blood pressure is relatively easy compared to blood velocity. In addition, we assume a linear pressure-area relationship for vessel cross-sections \cite{absi2018revisiting, st1998computer} and that the boundary derivative of the pressure is locally proportional to the fluid flow rate to establish boundary conditions for the pressure field, which is in accordance with previous studies \cite{berg2020modelling, arciero2017mathematical, reichold2009vascular}. 

To estimate circulation in distinguishable arteries, we propose a combination of the PPE method and Fick's law of diffusion for the microcirculation \cite{berg2020modelling, arciero2017mathematical}. We determine the relationship between estimated blood concentration and tissue conductivity through Archie's law, which is a two-phase mixture model that finds the effective conductivity of a mixture between fluid and porous media \cite{j2001estimation, peters2005electrical, glover2000modified, cai2017electrical}.

Numerical experiments were performed using a human head model consisting of multiple different compartments obtained by segmenting an openly available 7 Tesla MRI dataset \cite{ds} and allowing one to distinguish the arterial vessels inside the brain \cite{fiederer2016}. The results suggest that a dynamic model of CBF for both arterial and microcirculation can be established, allowing us to approximate the dynamic properties of the conductivity inside the brain. These results cover simulated mean, peak, and standard deviation (STD) distributions for two different heart rates.


This article consists of four sections. The second section provides a brief review of various topics, including NSEs, dynamic PPE, Fick's law for microcirculation, Archie's model, non-Newtonian viscosity modelling, and our approach to simulating signal pulses. The third section presents the results of the study. Finally, in the last section, the results are discussed and potential directions for future work are suggested.

\section{Methods}
\label{sec: Theory} 

\subsection{Navier--Stokes equations for non-Newtonian blood flow}
\label{sec:nses} 

In order to determine the pressure and velocity in the blood vessels, NSEs are approximated as follows:
\begin{subequations}
\begin{align}
&\rho \, {\bf u}_{,t} \! + \! \rho \, \mathsf{div}({\bf u} \! \otimes \! {\bf u}) \! - \! \mathsf{div}(\mu   {\bf Su} ) \!+ \! \nabla p \! =  \! \rho \,{\bf {\bf f}}&\mathsf{in}\,\Omega \! \times \! [0, T]\,,
\label{eqn:line-1.1}
 \\
& \mathsf{div}({\bf u})=0&\mathsf{in}\,\Omega \! \times \!  [0, T]\,,
\label{eqn:line-1.2}
  \\
&{\bf u}({\bf x};0)=0 &\mathsf{on}\,\Omega \,, 
\\
&{p}({\bf x};t)=p^{\mathcal{(B)}}({\bf x};t) - \frac{1}{\zeta \lambda } \frac{\partial p }{\partial \vec{\bf n} }&\mathsf{on}\, \partial \Omega\,, \label{robin}
\end{align}
\label{main-NSE}
\end{subequations}
where 
\begin{equation} {\bf Su} = \nabla {\bf u} + (\nabla {\bf u})^T
\label{eq:strain_rate}
\end{equation} 
is a second-order strain tensor described in section \eqref{sec:Mechanical characteristics of blood}.
The physical bound of our problem is considered to $\Omega\cup\microdomain$, where \(\microdomain\) is the micro-circulation domain and $[0, T]\in\mathbb{R}^+$ is the time domain. The blood velocity and pressure are defined as ${\bf u}={\bf u}({\bf x};t)$ and $p = p({\bf x}; t)$, respectively, in which ${\bf x}\in \Omega \cup \microdomain$ and $t\in[0, T]$. In the system \eqref{main-NSE}, the symbols $\rho$ and $\mu$ denote blood density and viscosity, respectively. The vector ${\bf f}$ represents the force density stemming from a constant gravitational field, while the term $\mathsf{div}(\mu {\bf Su})$ accounts for friction resulting from viscous drag.

The distribution ${p}^{\mathcal{(B)}}$  can be conceptualized as an incoming pressure pulse that generates the velocity field.
The  Robin-type boundary condition (\ref{robin}) has been  derived from the Hagen--Poiseuille model \cite{caro2012mechanics} of blood flow between arteries and arterioles, approximating the boundary normal derivative, i.e.,  the force density pushing blood out of the arteries,  as \[   -{\bf g}(\nabla p, {\vec{{\bf n}}})=-\frac{\partial p}{\partial\vec{{\bf n}}} =  \zeta \lambda (p - p^{\mathcal{B}}).  \]
Here  ${\bf g}$ is a Riemannian metric tensor, which is incorporated for generality, ${\vec{{\bf n}}}$ defines the outward normal unit vector on the artery wall, parameter $\zeta$ affects the amount of blood flowing into the arterioles, and $\lambda = \xi/\overline{\xi}$ is the ratio between  the length density $\xi$ of microvessels per unit volume ($m^{-2}$), which represents the total number of cross-sections per unit area, and its integral mean $\overline{\xi}$, which  is calculated as
\begin{align*}
\overline{\xi} = \frac{1}{|\partial \Omega|} \int_{\partial \Omega} \xi \hbox{d} \omega_{\partial \Omega}\,.
\end{align*}

The following formula for $\zeta$ can be derived from the Hagen--Poiseuille model by assuming that the normal derivative corresponds to the pressure decay per unit length in the arterioles, which lead the blood from the arteries to the microcirculation system \cite{samavaki2023ppe}:
\[
  \zeta \! = \! \frac{8 \pi \mu \overline{Q}}{  |\partial \Omega|  A_a \overline{p}  }\,.
\]
Here $Q$ denotes the total volumetric blood flow rate through the arteries, $A_a$ the cross-sectional area of an individual arteriole, $\overline{p}$ a reference average pressure value, ${\mu}$  a reference value for dynamic viscosity, which is defined in section \eqref{sec:Mechanical characteristics of blood}, and $|\partial \Omega|=\int_{\partial \Omega}\mathrm{d}\omega_{\partial \Omega}$ represents the surface area of $\partial \Omega$ which is the common boundary of the domain $\Omega$ and its exterior $\hat{\Omega}$.

\subsection{Dynamic pressure-Poisson equation}
\label{sec:Pressure-Poisson equatio}

A desired pressure field has the form $p = p^{\mathcal{(D)}} + p^{\mathcal{(H)}}$, i.e., it consists of the sum of dynamic arterial pressure $p^{\mathcal{(D)}}$ and hydrostatic venous pressure $p^{\mathcal{(H)}}$. Assuming that the gravity field remains constant, the blood flow is incompressible, and the geometry is flat with zero Riemannian curvature throughout the domain, the pressure $p$ meets the following pressure--Poisson system \cite{samavaki2023ppe}:
\begin{subequations}
\begin{align}
 &\Delta_B\, p =\mathsf{div}\big(\mathsf{div}(\mu   {\bf Su} ) -\rho \, \mathsf{div}({\bf u} \! \otimes \! {\bf u}) \big)
&\mathsf{in}\,\Omega \! \times \! [0, T]\,,
\label{eqn:line-2.2}
\\
&{\bf g}(\nabla p, \vec{{\bf n}})=-\zeta \, \lambda \,  ( p -  {p}^{\mathcal{(B)}}) &\mathsf{on}
\,\,\partial \Omega\,.
\label{eqn:line-2.4}
\end{align}
\label{main-PPE}
\end{subequations}

\label{sec:dynamic_form}

In order to derive a dynamic version of  (\ref{main-PPE}), it is assumed that there is a linear relationship between the volumetric flow rate $Q$ and pressure $p$, given by: 
\begin{equation}
p = \frac{\overline{p}}{\overline{Q}} \,  Q\,, 
\end{equation}
where $\overline{Q}$ is a reference value for the flow rate.  Additionally, it is assumed that pressure and area are linearly related \cite{st1998computer}, expressed as:
\begin{equation}
p = \alpha \, ( A - \overline{A})\,, 
\end{equation}
where $A$ is the cross-sectional area of an artery and $\overline{A}$ is a reference area where the pressure is zero. 

We can express the total volume of the arteries as $V = A \ell$, where $\ell$ is a constant equivalent length. Then, we can derive the following relationships:
\begin{equation}
Q = \frac{\overline{Q}}{\overline{p}} p =  \frac{\partial V}{\partial t} = \frac{\ell}{\alpha} \frac{\partial p}{\partial t} \quad \hbox{or} \quad   p = \frac{ \overline{p} \,  \ell}{\alpha \,  \overline{Q} } \frac{\partial p}{\partial t}\,.
\end{equation}
Assuming that the vessel has increased its diameter by a factor of $\sqrt{1 + \beta^2}$ due to the pressure $\overline{p}$, which can be expressed as $\overline{p} = \alpha  \beta^{2} \overline{A}$, we can derive the following equation:
\begin{equation}
\frac{\partial p}{\partial t} = \frac{(1 + \beta^2)  \overline{Q} }{\beta^2  \overline{V}}  p\,,
\label{pt}
\end{equation}
where $\overline{V}$ denote the overall volume of changes, which can be expressed as $(1 + \beta^2) \overline{A} \ell$.


Therefore, the PPE represented by equation \eqref{main-PPE} can be expressed in a dynamic form, where solving it results in a pulse wave as shown in  \eqref{pulse_wave_equation} on the boundary of the domain $\partial \Omega$ as follows:
\begin{subequations}
\begin{align}
 &\Delta_B\, p=   \mathsf{div}\big(\mathsf{div}(\mu   {\bf Su} ) -\rho\,{\bf u}\cdot\nabla {\bf u}\big)&\mathsf{in}\,\Omega \! \times \! [0, T]\,,
\label{eqn:PWE1}
\\
 &   {\bf g}(\nabla p, {\vec{{\bf n}}}) = - \frac{\lambda}{ \zeta\,\overline{\nu}^2} \, \left( \frac{\partial^2 p}{\partial t^2} -  \frac{\partial^2 {p}^{\mathcal{(B)}}}{\partial t^2} \right)
&\mathsf{on}
\,\,\partial \Omega\, ,
\label{eqn:PWE2}
\end{align}
\label{main-PWE}
\end{subequations}
with 
\begin{equation}
    \overline{\nu} = \frac{(1+\beta^2) \, |\partial \Omega| \, A_a  \, \overline{p}}{8 \pi \,  \mu   \, \beta^2 \, \overline{V}}\, .
\end{equation}

\subsubsection{Variational form of the Navier--Stokes system}
\label{sec:variational_nses}

Our assumptions are as follows: 1) The blood pressure, denoted as $p$, is defined within the space $\mathbb{W}=\mathrm{L}^2(\Omega)$. Additionally, we have ${\bf u}\in{\mathbb{V}}_0$ and ${\bf f}\in{\mathbb{V}}$, where $\mathbb{V}_0$ and $\mathbb{V}$ represent spaces of vector-valued functions in 3D. Specifically, $\mathbb{V}=[\mathrm{H}^1(\Omega)]^3$ and ${\mathbb{V}}_0 = [{ \mathrm{H}}^1_0(\Omega)]^3$, where ${\mathbb{V}}_0\subset\mathbb{V}$. 2) The Ricci curvature of ${\bf u}$ is zero, i.e., ${\bf Ri}({\bf u})=0$, which means that the domain has significant implications for the geometry; 3) The vanishing of the second covariant derivative of blood viscosity indicates that the blood flow in different directions through the blood vessels of the domain is the same. This is because the viscosity of blood is assumed to be uniform in all directions; 4)  the gravitational field is assumed to be divergence-free, i.e.,  $\mathsf{div}(\rho {\bf g})=0$. By applying Lemma \ref{curl-au}, we can express the variational formulation of the system of NSEs for the dynamic PPE \eqref{main-PWE} as follows:

\begin{itemize}
    \item[I.] Find $p\in {\mathbb{W}}$ such that, for a smooth enough test function $ q\in  {\mathbb{W}}$

\begin{equation}
\begin{aligned}
 b( p, q)&=     2 \, \int_{\Omega}   {\bf g} \left( \nabla q, (\nabla {\bf u}) \nabla \mu \right)
 \, \hbox{d} \omega_{\Omega}  \\ & - \int_{\Omega} \rho \, {\bf g} \left( \nabla q, (\nabla{\bf u}) \, {\bf u} \right)  \hbox{d}\omega_{\Omega}\,,
\end{aligned}
\label{PWE_variational}
\end{equation}

where the continuous bilinear form $b:{\mathbb{W}}\times {\mathbb{W}}\rightarrow \mathbb{R}$ is defined as follows:
 \begin{align*}
b(p, q):= & \! \int_{\Omega } \! {\bf g}(\nabla p, \! \nabla q)\mathrm{d}\omega_{\Omega} - \int_{\partial \Omega}q\,{\bf g}(\nabla p, \vec{{\bf n}})\mathrm{d}\omega_{\partial \Omega} \\
= & \! \int_{\Omega } \! {\bf g}(\nabla p, \! \nabla q) \, \mathrm{d}\omega_{\Omega}  \! + \! \zeta \lambda \frac{\partial^2 }{\partial t^2} \! \int_{\partial \Omega} q\,( p \! - \! p^{\mathcal(B)}) \,  \mathrm{d}\omega_{\partial \Omega}\,. 
\end{align*}
 
The variational formulation for the system of NSEs \eqref{main-NSE} can be expressed through the application of Lemma \ref{lemma1} and \ref{curl-au} as follows:
 \item[II.] Find ${\bf u} \in {\mathbb{V}}_0$  such that, for a smooth enough 
 test function $ {\bf v} \in {\mathbb{V}}_0$
 \begin{equation}
       a({\bf u}, {\bf v}) = -  \int_\Omega  {\bf g}\left(\nabla p, {\bf v}\right)   \, \hbox{d} \omega_\Omega \, .
 \label{NSE_variational}
 \end{equation}
The continuous bilinear form $a_j:{\mathbb{V}}_0 \times {\mathbb{V}}_0 \rightarrow \mathbb{R}$ is defined as follows:
\begin{align*}
 a({\bf u}, {\bf v}) &:=  \frac{\partial}{\partial t} \int_{\Omega}  \rho \,{\bf g}( {\bf u},  {\bf v})  \, \mathrm{d}\omega_{\Omega} +  \int_{\Omega}\rho\,  {\bf g} \left({\bf u}\cdot\nabla {\bf u},  {\bf v} \right)  \,  \mathrm{d}\omega_{\Omega}
 \\
 & -2  \int_{\Omega} {\bf g}\big((\nabla{\bf u})^T\nabla\mu,  {\bf v} \big)  \,  \mathrm{d}\omega_{\Omega}
+\int_{\Omega} \mu\,({\bf Au}:\nabla{\bf v})  \,  \mathrm{d}\omega_{\Omega}
 \\
 &-\int_{\Omega}\rho\,{\bf g} \left({\bf f}, {\bf v}\right)\,  \mathrm{d}\omega_{\Omega}\,.
\end{align*}
\end{itemize}

\subsection{Fick's law and microcirculation}
\label{sec:ficks_law}

Microcirculation is modeled using Fick's first law. Our assumption is that the blood concentration and diffusive flow of the tissue are proportional to the relative density of microvessels in the tissue following
\begin{equation}
\label{fick}
  {\bf J}=-\varsigma \lambda \nabla{\bf c}  \quad \hbox{with} \quad \varsigma =  \frac{A_a \overline{p} }{8 \pi  \mu}\,.
\end{equation}
Diffusivity here is defined via the effective diffusion coefficient $\varsigma$. The following equation is derived without any convection because the model does not include tissue anisotropies or microvessel structures:
\begin{subequations}
\begin{align}
{\bf c}_{,t}-\varsigma\lambda \Delta_B {\bf c}=- \boldsymbol{ \varepsilon} {\bf c}+{\bf s} \quad \mathsf{in}\,\,\hat\Omega \! \times \! [0, T]\,.
\label{DE.vessel_1.1}
\end{align}
\label{DE.vessel_1}
\end{subequations}
We assume that the peak concentration occurs at the source ${\bf s}={\bf s}({\bf x}, t)$, which differs from zero only for $\partial \Omega$, representing a flux from $\Omega$ to $\hat{\Omega}$. In a venous circulation system, the blood flow from the microcirculation domain occurs at a rate proportional to the decay term $\boldsymbol{ \varepsilon} {\bf c}$. In the FE discretization, the absorption amplitude ${\varepsilon}$  multiplied by the volume $\mathcal{V}_{\hbox{\scriptsize max}} = (4/3) \pi R^3$ of the largest element is set to match the outward flux integrated over a surface of a sphere with an equal volume:
\[
\varepsilon \mathcal{V}_{\hbox{\scriptsize max}} = (4/3) \pi R^3 \varepsilon =  4 \pi R^2 \| {\bf J} \| \quad \hbox{or} \quad 
\varepsilon =   \varsigma \lambda \frac{ \vartheta }{ L  } \left( \frac{45 \pi}{\mathcal{V}_{\hbox{\scriptsize max}}} \right)^{1/3}\,.
\]

\subsubsection{Variational form of Fick's law}

Multiplying \eqref{DE.vessel_1.1}  by a smooth enough test function $h \in {\hat{\mathbb{W}}}$  and integrating by parts yields the following variational form of Fick's law:
\begin{itemize}
  \item[] Find $c\in {\mathbb{W}}$ such that, for a smooth enough  test function $ q\in  {\mathbb{W}}$
\begin{equation}
 \begin{aligned}
   d( c, h) = ( s, h) +  \int_{\partial \Omega}\varsigma \lambda h\, {\bf g}( \nabla c, {\vec{ \bf{n}}})\, \mathrm{d }\omega_{\partial \Omega}\,.
 \end{aligned}
\label{fick_variational}
\end{equation}
\end{itemize}

Here, $\vec{\bf n}$ represents the normal unit vector in the microcirculation domain, and the linear boundary term describing the incoming flow is on the left. Bilinear forms $ d( c, h)$ and $( s, h)$ are defined by $d:{\hat{\mathbb{W}}}\times {\hat{\mathbb{W}}}\rightarrow \mathbb{R}$ and $s:{\hat{\mathbb{W}}}\times {\hat{\mathbb{W}}}\rightarrow \mathbb{R}$, where

\begin{align*}
d(c, h) =  & \frac{\partial}{\partial t}  \int_{\Omega} c\, h\,\mathrm{d}\omega_{\hat{\Omega}}   +  \int_{\hat{\Omega}}  \varsigma \lambda{\bf g}(\nabla c, \nabla h)\,\mathrm{d}\omega_{\hat{\Omega}}
+  \int_{\hat{\Omega}} \varepsilon c\, h\,\mathrm{d}\omega_{\hat{\Omega}}\,,
\\
( s, h) = & \int_{\hat{\Omega}} s\,h\,\mathrm{d }\omega_{\hat{\Omega}}\,.
\end{align*}

\subsection{Discretization}

We utilize the Ritz-Galerkin method \cite{braess2007finite} to discretize equations \eqref{PWE_variational}, \eqref{NSE_variational}, and (\ref{fick_variational}). It is assumed that the discrete solutions are enclosed within the trial function spaces as stated below.
\begin{align*}
{\mathbb{V}_h}&=\mathsf{span}\big\{\psi^1,\dots,\psi^{n}\big\}
   \subset {\mathbb{V}} 
   \\
 {\mathbb{W}_h}&=\mathsf{span}\big\{\varphi^1,\dots,\varphi^{m}\big\}
   \subset {\mathbb{W}} 
   \\ 
   {\hat{\mathbb{W}}_h}&=\mathsf{span}\big\{\phi^1,\dots,\phi^{m}\big\}
   \subset {\hat{\mathbb{W}}}\,. 
  \end{align*}
In each case, the discretization error is assumed to be orthogonal to the solution. 
We utilize linear Lagrangian (nodal) basis functions, which consist of sets of piecewise linear functions supported in $\Omega$ and $\hat{\Omega}$. Specifically, we consider the sets $\{\psi^i\}_{i=1}^{n}$ and $\{\varphi^h\}_{h=1}^{m}$ supported in $\Omega$ and $\{\phi^h\}_{h=1}^{m}$ supported in $\hat{\Omega}$, such that $\psi^i(x_j)=\delta^i_j$ for $i, j=1, \cdots, n$, $\varphi^h(x_k)=\delta^h_k$ for $h, k=1, \cdots, m$ at the finite element mesh nodes of $\Omega$, and $\phi^h(x_k)=\delta^h_k$ for $h, k=1, \cdots, m$ at the nodes of $\hat{\Omega}$. Consequently, the velocity $\bf u\in {\mathbb{V}}$, the pressure $p \in \mathbb{W}$, and the concentration $c \in \hat{\mathbb{W}}$ are of the form   
\begin{align*}
 u^{\ell}({\bf x}; t)&=\sum_{i=1}^{n} \psi^i({\bf x}) u^{\ell}_i(t)\,, \quad
f^{\ell}({\bf x})=\sum_{i=1}^{n} \psi^i({\bf x}) f^{\ell}_i\,,
 \\
 p({\bf x}; t) &=  \sum_{h=1}^{m} \varphi^{h}({\bf x}) p_{h}(t) \,,\quad \,\,\,
c({\bf x}; t) \! = \! \sum_{h=1}^{m} \phi^{h}({\bf x}) c_{h}(t)\,,
\end{align*}
for $\ell= 1, 2, 3$. The coordinate vectors are denoted by {\setlength\arraycolsep{2pt} 
\begin{align*} 
{\bf u} & = (u^1, u^2, u^3)=(\{u^1_i\}_{i=1}^n, \{u^2_i\}_{i=1}^n, \{u^3_i\}_{i=1}^n)
\\ 
{\bf f} & =  (f^1, f^2, f^3)=(\{f^1_i\}_{i=1}^n,\{f^2_i\}_{i=1}^n, \{f^3_i\}_{i=1}^n) 
\\ 
 {\bf p } & =  (p^1, p^2, p^3)=(\{p_i\}_{i=1}^m,\{p_i\}_{i=1}^m, \{p_i\}_{i=1}^m)
\\
{\bf c }& =   (c_1, c_2, \ldots, c_m)
\\
\uppsi&=(\psi_1, \ldots, \psi_n )
\quad
\upvarphi =(\varphi_1, \ldots, \varphi_m )
\\
\Phi &=( \phi_1, \ldots, \phi_m)\,.
\end{align*}
In this context, we have $p^\ell=p$ for $\ell=1, 2, 3$.


\subsubsection{Discretized arterial blood flow}
\label{sec:discretized_blood_flow}

The discretized version of the variational form \eqref{PWE_variational} is given by 
\begin{equation}
{\bf K} p  + \frac{\partial^2}{\partial t^2} {\bf M} p = {\bf D}({\bf u})\,,
\label{PPE_diff}
\end{equation}
where
{\setlength\arraycolsep{2 pt}
\begin{align*}
&  {\bf K}_{ij}  =  \int_{\Omega} \! \varphi_{i, h}\,\hat{\varphi}_{j, h}=\int_{\Omega} \! \hbox{d}\upvarphi \,(\hbox{d}\hat{\upvarphi})^T \,  \hbox{d} \omega_{\Omega }
\\ 
& {\bf M}_{ij} =  \int_{\partial \Omega} \! \zeta \lambda  \, \varphi_i \hat{\varphi}_j  \, \hbox{d} \omega_{\partial \Omega}=\int_{\partial \Omega} \!\zeta \lambda  \, \upvarphi\hat{\upvarphi}^T\,  \hbox{d} \omega_{\partial \Omega} 
\\
& {\bf D}^\ell_{ij}({\bf u}) =2  \int_{\Omega}  \hat{\varphi}_{\ell, i}\,\psi_{j, h}\,\mu_{,h}\, \hbox{d} \omega_{\Omega}  
 - \int_{\Omega} \!\rho \, \hat{\varphi}_{\ell, i}\,\psi_{j, h}\,u^{h} \,\hbox{d}\omega_{\Omega}
 \\
 &\quad\quad\,\,\,=2  \int_{\Omega}  \hbox{d}\hat{\upvarphi}\,\big((\hbox{d}\uppsi)\,\hbox{d}\mu \big)\,\, \hbox{d} \omega_{\Omega} -\int_{\Omega}\rho\, \big(\hbox{d}\hat{\upvarphi}\,(\uppsi^T u^1)\uppsi^T_{,1}
 \\
 &\quad\quad\quad\quad\quad\quad+\hbox{d}\hat{\upvarphi}\,(\uppsi^T u^2)\uppsi^T_{,2}
+\hbox{d}\hat{\upvarphi}\,(\uppsi^T u^3)\uppsi^T_{,3}\big)\, \mathrm{d}\omega_{\Omega}\,,
\end{align*}}
and
\[
 {\bf D}({\bf u})= \begin{pmatrix}{\bf D}^1({\bf u})& {\bf D}^2({\bf u})&\ldots&{\bf D}^m({\bf u}) \end{pmatrix}^T\,.
\]
The discretized velocity can be obtained via the following regularized equation, which follows as the discretization of \eqref{NSE_variational}: 
\begin{equation}
\frac{\partial}{\partial t} {\bf \mathcal{ C}} {\bf u} + {\bf \mathcal{ H}}(\tilde{\bf u}) {\bf u} +{\bf \mathcal{L}} {\bf u}+{\bf \mathcal{ F}} = - {\bf \mathcal{ Q}}({\bf p})\,,
\label{SE_diff}
\end{equation}


where the smoothed velocity field, denoted by $\tilde{\bf u}$, is obtained from the original velocity field ${\bf u}$ using Leray regularization as described in \cite{guermond2004definition}; to obtain $\tilde{\bf u}$, we solve a discretized version of the equation $({\bf I }- \epsilon^2 {\bf A}) \tilde{\bf u} = {\bf u}$, where \[
{\bf A}^{\ell}_{ij} = \int_\Omega \,\psi_{i}\hat{\psi}_j
 \, \mathrm{d}\omega_{\Omega}=\int_\Omega \uppsi \hat{\uppsi}^T \, \mathrm{d}\omega_{\Omega}\,,
\]
for all $\ell=1, 2, 3$.  The smoothing parameter $\epsilon$ is a user-defined parameter that is used in the regularization process. Thus, the matrices referred to in equation \eqref{SE_diff} have a specific form as follows
\begin{align*}
{\bf \mathcal{ C}}   & = \begin{pmatrix} {\bf C} & {\bf 0} & {\bf 0} \\  {\bf 0} & {\bf C}& {\bf 0} \\ {\bf 0} & {\bf 0} & {\bf C}  \end{pmatrix}
\quad 
{\bf \mathcal{ H}}({\bf u})    = \begin{pmatrix} {\bf H}({\bf u}) & {\bf 0} & {\bf 0} \\  {\bf 0} & {\bf H} ({\bf u})& {\bf 0} \\ {\bf 0} & {\bf 0} & {\bf H} ({\bf u}) \end{pmatrix} 
\\
{\bf \mathcal{F}}  &= \begin{pmatrix} {\bf F}^1 \\ {\bf F}^2\\ {\bf F}^3\end{pmatrix} 
\quad 
{\bf \mathcal{ Q}}({\bf p})  = \begin{pmatrix} {\bf Q}({\bf p}) \\ {\bf Q}({\bf p}) 
\\ {\bf Q}({\bf p})\end{pmatrix}
\\
{\bf \mathcal{ L}}^1   &= \begin{pmatrix} {\bf L}^1_{22}+{\bf L}^1_{33} & -{\bf L}^1_{12} & -{\bf L}^1_{13} \\  -{\bf L}^1_{21} &  {\bf L}^1_{11}+{\bf L}^1_{33}& -{\bf L}^1_{23} \\ -{\bf L}^1_{31} & -{\bf L}^1_{32} &  {\bf L}^1_{11}+{\bf L}^1_{22}  
\end{pmatrix} 
\\
{\bf \mathcal{L}}^2   &= \begin{pmatrix} {\bf L}^2_{11} & {\bf L}^2_{12} & {\bf L}^2_{13} \\  {\bf L}^2_{12} & {\bf L}^2_{22}& {\bf L}^2_{23} \\ {\bf L}^2_{31} & {\bf L}^2_{32} & {\bf L^2_{33}} 
\end{pmatrix}
\\
{\bf \mathcal{ L}}&={\bf \mathcal{ L}}^1+{\bf \mathcal{L}}^2\,.
\end{align*}
with
\begin{align*}
& {\bf C}_{ij}  =   \int_{\Omega} \! \rho \, \psi_{i}  \hat{\psi}_{j}  \, \mathrm{d}\omega_{\Omega}= \int_{\Omega}  \! \rho \, \uppsi\,\hat{\uppsi}^T\mathrm{d}\omega_{\Omega}
\\
& {\bf H}_{ij}({\bf u})  = \int_{\Omega} \!\rho\, {\psi}_{i, k}  \, u^k\, \hat{\psi}_j \, \mathrm{d}\omega_{\Omega}= \int_{\Omega} \!\rho\, {\psi}_{i, k}  \, u_\ell^k\,\psi^{\ell} \,\hat{\psi}_j\,  \mathrm{d}\omega_{\Omega}
\\
 &\quad\quad\,\,\,=\int_{\Omega}\rho\, \big(\hat{\uppsi}(\uppsi^T u^1)\uppsi^T_{,1}+\hat{\uppsi}(\uppsi^T u^2)\uppsi^T_{,2}
\\
&\quad\quad\quad\quad\quad\quad\quad\quad\quad\quad\quad+\hat{\uppsi}(\uppsi^T u^3)\uppsi^T_{,3}\big)\, \mathrm{d}\omega_{\Omega}
 \end{align*}
 \begin{align*}
& {\bf L}^1_{ij}  = \int_{\Omega} \mu \,\psi_{h, i}\,\hat{\psi}_{h, j}\,\mathrm{d}\omega_{\Omega}
=\int_{\Omega} \mu \,  \mathrm{d}\uppsi(\mathrm{d}\hat{\uppsi})^T\,\mathrm{d}\omega_{\Omega}
  \\
&  {\bf L}^2_{ij}  = -2\int_{\Omega} \psi_{k, i}\, \mu_{,k}\,\hat{\psi}_{j}  \,\mathrm{d} \omega_\Omega=-2\int_{\Omega} \big(\mathrm{d} \Psi\,(\mathrm{d} \mu)^T\big)\,\hat{\uppsi}^T  \,\mathrm{d} \omega_\Omega
\\ 
& {\bf Q}_{ij}({\bf p})  =  -\int_\Omega  p_{h}\,\varphi_{h, i} \,\hat{\psi}_{j}  \, \hbox{d} \omega_\Omega = - \int_\Omega    \, \mathrm{d}p\,\hat{\uppsi}^T\,\mathrm{d} \omega_\Omega
 \\
& {\bf F}^{\ell}_{ij}  = \int_\Omega \rho\, f^\ell_i\,\hat{\psi}_j\,\mathrm{d} \omega_\Omega=\int_\Omega \rho\, f^\ell\, \hat{\uppsi}^T   \,\mathrm{d} \omega_\Omega\,,
\end{align*}
for all $\ell=1, 2, 3$.
To solve the equations \eqref{PPE_diff} and \eqref{SE_diff}, we can use an iterative approach that involves approximating the first and second-order time derivatives using the second-order forward difference formula.
\begin{align*} 
 \frac{\partial F}{\partial t} ({\bf x}; t)  = & \frac{F({\bf x}; t+h) - F({\bf x}; t)}{h}   - \frac{h^2}{2} \frac{\partial^2 }{\partial t^2} F({\bf x}; c), \\ &  \hbox{with} \quad t \leq c \leq t + h \\
 \frac{\partial^2 F}{\partial t^2} ({\bf x}; t)  = & \frac{F({\bf x}; t+2h) -2 F({\bf x}; t+h)  + F({\bf x};t)}{h^2} \\   & -\frac{h^2}{12} \frac{\partial^4 }{\partial t^4} F({\bf x}; c), \quad \hbox{with} \quad t \leq c \leq t + 2h\,, 
\end{align*}
which together imply the following recursion formulae:
\begin{align}
{\bf p}_{k}   &=  \,    \big( \Delta t^2 \,  {\bf K}  +  {\bf M} \big)^{-1} \big( \Delta t^2 \, {\bf D}({\bf u}_{k-1})   +  2 {\bf M}  {\bf p}_{k-1}   -  {\bf M}  {\bf p}_{k-2} \nonumber
\\
& +  {\bf M}  {\bf p}_{k}^\mathcal{(B)}  -  2 {\bf M} {\bf p}_{k-1}^\mathcal{(B)}  +  {\bf M}  {\bf p}_{k-2}^\mathcal{(B)} \big)\,, 
\\
{\bf u}_k  &=   {\bf u}_{k-1} + \Delta t \,  {\bf \mathcal{ C}}^{-1} \left( -{\bf \mathcal{ Q}}({\bf p}_k) - {\bf \mathcal{ H}}(\tilde{\bf u}_{k-1}){\bf u}_{k-1} - {\bf \mathcal{ L}} {\bf u}_{k-1}-{\bf \mathcal{ F}}_{k-1} \right)\, . 
\end{align}


\subsubsection{Discretized microcirculation model}

The variational discrete diffusion problem related to the system of \eqref{fick_variational} can be reformulated as follows:
\begin{itemize}
     \item[] Find $c_h\in  {\hat{\mathbb{W}}_h}$, such that, for all $ \varphi_h\in \hat{\mathbb{W}}_h$
\begin{equation}
\begin{aligned}
 d( c_h, \phi_h) =  ( s, \phi_h)\,.
\end{aligned}
\label{fick_variational_var_bar}
\end{equation}
\end{itemize}
A numerical solution ${\bf c}$ of \eqref{fick_variational_var_bar} can be obtained recursively via
\[
\label{d_fick}
\frac{\partial }{\partial t} {\bf U} {\bf c}+ {\bf R}{\bf c} + {{\bf T}}{\bf c} = {\bf w}\,,
\]
where
\begin{align*}
& {\bf U}_{ij}  =   \int_{\hat{\Omega}}  \phi_i  \hat{\phi}_j \,  \hbox{d} \omega_{\hat{\Omega}} =\int_{\hat{\Omega}} \Phi \,\Phi^T \,  \hbox{d} \omega_{\hat{\Omega}}
\\ 
& {\bf R}_{ij}  =   \int_{\hat{\Omega}}  \varsigma \lambda \,\phi_{h,i}\,  \hat{\phi}_{h,j} \,  \hbox{d} \omega_{\hat{\Omega}} =   \int_{\hat{\Omega}}  \varsigma \lambda \,\hbox{d}\Phi\, (\hbox{d} \hat{\Phi} )^T\,  \hbox{d} \omega_{\hat{\Omega}} 
\\ 
&{\bf T}_{ij}  =  \int_{\hat{\Omega}}   \varepsilon \,\phi_i \, \hat{\phi}_j \,  \hbox{d} \omega_{\hat{\Omega}}=  \int_{\hat{\Omega}}  \varepsilon\,  \Phi \, \hat{\Phi}^T \,  \hbox{d} \omega_{\hat{\Omega}}
\\
&{\bf w}_{i}  =  \int_{ \hat{\Omega}}  s \,\phi_i \,  \hbox{d} \omega_{\hat{\Omega}}=  \int_{ \hat{\Omega}}  s \,\Phi \,  \hbox{d} \omega_{\hat{\Omega}}\,.
\end{align*}

The time iteration of the matrix equation above is expressed as follows:
\begin{equation}
{\bf c}_k =  ({\bf U} +  \Delta t \,  {\bf R} + \Delta t \, {\bf T})^{-1}  ({\bf U} {\bf c}_{k-1} + \Delta t \, {\bf w } )\,.
\label{cv}
\end{equation}

  
\subsection{Simulated incoming pulse wave}
\label{sec:pulse}

To simulate the incoming boundary pulse distribution $p^{({\bf \mathcal{B}})}$, we apply a three-component wave model \cite{nagasawa2022blood}  of the form 
\begin{equation*}
    p^{(\mathcal{B})}({\bf x}; t) \! = \!  \mathcal{A}  \alpha_P  w \left(\frac{t -  t_P}{L_P/\mathsf{c}} \right) \! + \! \mathcal{A}  \alpha_T  w \left(\frac{t -  t_T}{L_T/\mathsf{c}} \right)  \! +  \! \mathcal{A}  \alpha_D  w \left(\frac{t -  t_D}{L_D/\mathsf{c}} \right) 
\end{equation*}
in ${\bf x} \in  \mathcal{S}$ and $p^{(\mathcal{B})}({\bf x}; t) = 0$, otherwise.
In the above equation, $w(t)$ is a smooth window function on the interval $0 < t < 1$ that is periodic, such that $w(t) = w(t + 1)$ for any $t$. $\mathcal{A}$ is a normalizing amplitude, while $\mathsf{c}$ represents the cycle length. $\alpha_P$, $\alpha_T$, and $\alpha_D$ denote the weights of the percussion, tidal, and dicrotic wave components, respectively, while $L_P$, $L_T$, and $L_D$ represent their durations, and $t_P$, $t_T$, and $t_D$ denote their start times. We use the Blackman-Harris window as the pulse form $w(t)$.
\\
The spatial support, denoted as $\mathcal{S}$, for $p^{(\mathcal{B})}({\bf x}; t)$ corresponds to two spheres with radii of 10 and 3 mm, feeding pressure in the internal carotid and basilar arteries  in the front of the brainstem. We chose the following parameter values: $\alpha_P = 0.50$, $w_T = 0.30$, $w_D = 0.25$, $L_P = 0.55$, $L_T = 0.55$, $L_D = 0.60$, $t_P = 0.05$, $t_T = 0.20$, and $t_D = 0.38$. We used two alternative cycle lengths, $\mathsf{c} = 1$ and $\mathsf{c} = 0.75$ s, i.e., heart rates of 60 and 80 bpm. The normalizing amplitude $\mathcal{A}$ was set to match the total pulse pressure of 50 mmHg. The resulting two boundary pulses are shown in Figure \ref{fig:boundary_pulse} and their spatial support  $\mathcal{S}$ in Figure \ref{fig:domain}.

\begin{figure}[h!]
    \centering
    \begin{scriptsize}
    \includegraphics[width=8cm]{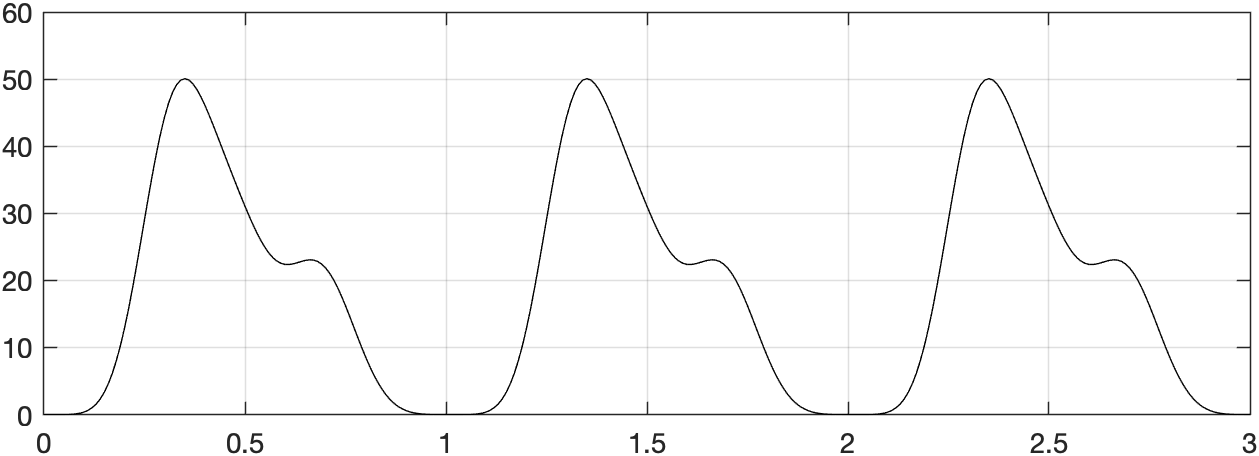} \\ \vskip0.2cm 60 bpm \\ \vskip0.2cm
    \includegraphics[width=8cm]{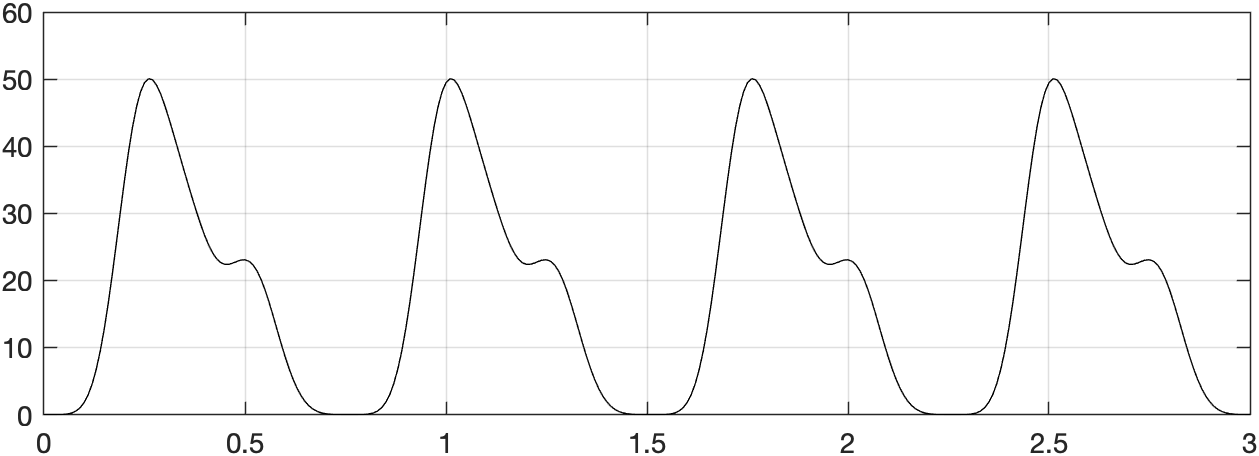} \\
    80 bpm
    \end{scriptsize}
    \caption{The boundary pulse function $p^{(\mathcal{B})}({\bf x}; t)$ for all ${\bf x} \in  \mathcal{S}$ obtained as a  sum of simulated percussion, tidal and dicrotic wave components. Two alternative heart rates 60 and 80 bpm are shown (top and bottom, respectively). }
    \label{fig:boundary_pulse}
\end{figure}
 
\subsection{Blood viscosity}
\label{sec:Mechanical characteristics of blood}

In accordance with Newton's law of viscosity, variations in blood flow geometry or velocity may lead to significant changes in blood viscosity. Blood is a non-Newtonian fluid, which does not follow Newton’s law, and, thus, its viscosity is not constant and can change according to a variety of variables, including the shear rate $\mu=\mu(\dot{\gamma})$.  Therefore, we define the kinematic viscosity $\mu$, [$\mu$] = Pa s (SI unit), as a variable of the blood flow. 


\subsubsection{Strain, shear stress and stress tensors}

Viscosity is closely related to stress, which is the physical quantity determining the magnitude of the forces causing deformation. The strain tensor describes the deformation of fluid in response to applied stress and is a mathematical construct for an order-two symmetric tensor that expresses the rate of deformation in 3D and has the form \eqref{eq:strain_rate}, $[{\bf S{\bf u}}] = s^{-1}$ (SI unit). The rate at which a fluid is being sheared or deformed while it flows is known as the shear rate $\dot{\gamma}$ and is a scalar quantity of the rate of deformation (strain) tensor, which is typically represented as follows:
\begin{align*}
\dot{\gamma} :=\sqrt{\frac{1}{2}\mathsf{tr}({\bf Su})}\,,
    \label{}
\end{align*}
where
\[
\mathsf{tr}({\bf Su})={\bf Su}:{\bf Su}=(\nabla {\bf u}+(\nabla {\bf u})^T)^{ij}(\nabla {\bf u}+(\nabla {\bf u})^T)_{ji}\,.
\]
The amount of force per area required for a fluid to flow is measured by shear stress. It represents the amount of stress being applied to the fluid and is also represented by an order-two tensor, ${\pmb\tau}{\bf u}=\mu {\bf S{\bf u}}$, $[{\pmb\tau}{\bf u}] =$ Pa (SI unit).

The distribution of stresses within a fluid is described by mathematical structures known as stress tensors, which are defined as ${\bf\sigma}({\bf u}, p)=-p{\bf I}+\mu {\bf S{\bf u}}$ an order-two tensors, $[{\bf\sigma}({\bf u}, p)] =$ Pa (SI unit). We can understand more about the fluid's behavior under different conditions and how it may react to external forces in its surroundings by analyzing the stress tensor.




One common approximation of the zero-shear-rate viscosity, where the shear force is absent, and Infinite-shear-rate viscosity which describes the viscosity of a fluid under an infinite shear force or stress are as follows:
\begin{align*}
   \mu_{0}&= \mathsf{lim}_{\dot{\gamma}\rightarrow 0} \,\mu(\dot{\gamma})= 56\times 10^{-3}\, \hbox{Pa s}\,,
\\   
   \mu_{\infty}&= \mathsf{lim}_{\dot{\gamma}\rightarrow \infty} \,\mu(\dot{\gamma})= 3.45\times 10^{-3}\,\hbox{Pa s}\,.
\end{align*}

\subsubsection{Carreau-Yasuda model} 
\label{sec:carreau_yasuda_model}

Given the shear rate $\dot{\gamma}$, Carreau-Yasuda model \cite{cho1991effects,weddell2015hemodynamic} can be applied to define the viscosity of non-Newtonian fluids as follows:
$$
 \eta(\gamma)=\mu_{\infty}+\left(\mu_0-\mu_{\infty}\right)\left(1+(\uplambda\, |\dot{\gamma}|)^\upalpha\right)^{(\mathsf{n}-1) / \upalpha}
$$
where $\uplambda$ is a relaxation time, $\mathsf{n}$ is a power law index, and $\alpha$ is a transition parameter. In this study, we apply the Carreau-Yasuda model by selecting these parameters as suggested in \cite{weddell2015hemodynamic}, i.e., $\uplambda = 1.902$, $0.22$, $1.25$. In the numerical discretization, the Carreau-Yasuda model was implemented using the linear Lagrangian finite element basis functions and the (Leray) smoothing scheme described in Section \ref{sec:discretized_blood_flow} to obtain a numerically stable solution.


\subsection{Archie's law}
\label{sec:archies_law}

To estimate the effect of the volumetric blood concentration on the conductivity distribution, we apply Archie's law,  which approximates the effective electrical conductivity $\sigma$ for a two-phase mixture of fluid and inhomogeneous medium as follows \cite{j2001estimation,peters2005electrical,glover2000modified,cai2017electrical}:
\begin{equation}
\label{archie}
 \sigma =\sigma_m ( 1 - c )^\tau+\sigma_f {c}^{\beta} \text { with } \tau=\frac{\log \left(1-c^\beta \right) }{ \log (1-c) }\,, 
\end{equation}
Here $\sigma_f$ and $\sigma_m$ are  fluid and medium conductivities, respectively, and $\beta$ is a cementation factor \cite{j2001estimation,glover2000modified}, which for the cerebral cortex has been suggested to satisfy $3/2 < \beta < 5/3$ \cite{j2001estimation}. The lower and upper limit for $\beta$  follow from spherical  and cylindrical inhomogeneities, which in the  cortex  are represented by the somas and dendrites of the pyramidal cells, respectively. Since we assume the uncertainty of the concentration model be greater than that of the Archie's law, we restrict the present investigation to the upper limit $\beta = 5/3$ which gives the lower conductivity out of the two limits.

\subsection{Head model segmentation}
\label{sec:head_model}

\begin{table}[!ht]
    \centering
    \caption{Compartments of the head model segmentation applied in the numerical experiments. The piecewise constant conductivity values of the compartments were based on  \cite{dannhauer2010}  associating the subcortical nuclei with the conductivity of the grey matter \cite{rezaei2021reconstructing} and the vessel conductivity with that of the blood \cite{gabriel1996compilation}. 
    }
    \begin{footnotesize}
    \begin{tabular}{llr}
    \hline 
        Compartment   &  Segmentation method & $\sigma_m$ (S m\textsuperscript{-1}) \\
        \hline
 Blood vessels & Vesselness &  0.70\\
    Grey matter & FreeSurfer &   0.33 \\
White matter  &   "& 0.14 \\
Cerebellum cortex & " & 0.33 \\
Cerebellum white matter & "  & 0.14 \\
Brainstem & " & 0.33 \\ 
Cingulate cortex & " & 0.14 \\ 
Ventral Diencephalon &" & 0.33 \\
Amygdala & " & 0.33 \\ 
Thalamus &  "& 0.33 \\
       Caudate &" & 0.33 \\  
              Accumbens & " & 0.33 \\ 
    Putamen & " & 0.33 \\
     Hippocampus & "& 0.33 \\ 
 Pallidum & " & 0.33 \\ 
  Ventricles & "& 0.33 \\ 
       Cerebrospinal fluid (CSF) & FieldTrip-SPM12 & 1.79 \\    
                         \hline                     
    \end{tabular}
    \end{footnotesize}
    \label{tab:segmentation}
\end{table}

The head model of this study was created using the open sub-millimeter precision Magneto Resonance Imaging (MRI) dataset\footnote{doi:10.18112/openneuro.ds003642.v1.1.0} of CEREBRUB-7T \cite{svanera2021cerebrum}.  Due to the dataset being acquired using 7 Tesla fields, arterial vessels can be separated as individual compartments, as shown in \cite{fiederer2016}. The segmentation of the head was found
applying the FreeSurfer Software Suite \cite{fischl2012freesurfer}, FieldTrip's \cite{oostenveld2011} interface for the SPM12 surface extractor \cite{ashburner2014spm12}, and the Vesselness algorithm \cite{van2014scikit,frangi1998multiscale,fiederer2016}, the last one of which found the vessel segmentation. Altogether, 16 brain compartments enclosed by the skull were included in the microcirculation domain (Table \ref{tab:segmentation}). The segmentation is shown in Figure \ref{fig:domain}. 

\subsubsection{Vessel extraction}
\label{sec:vessels}

\begin{figure*}[h!]
\begin{footnotesize}
    \centering
    \begin{minipage}{3.0cm}
    \centering
    \includegraphics[width=3cm]{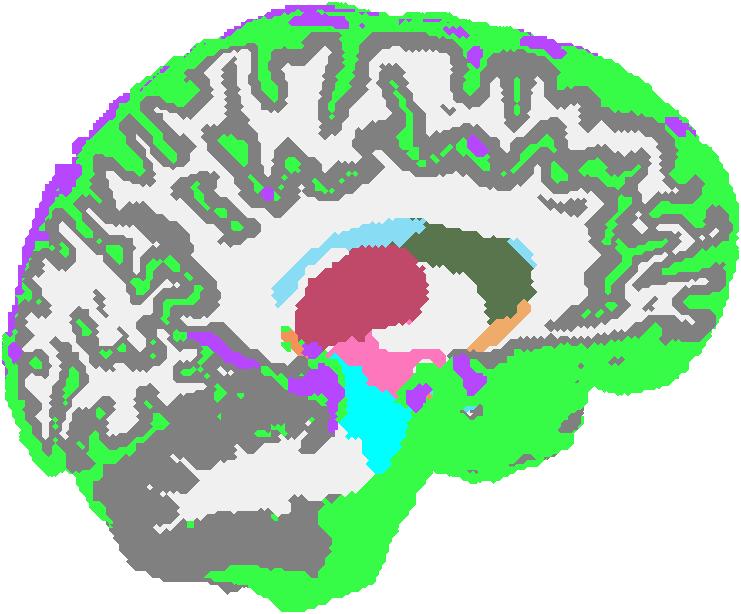} \\
   Domain, sagittal
\end{minipage}
 \begin{minipage}{3.0cm}
    \centering
        \includegraphics[width=2.6cm]{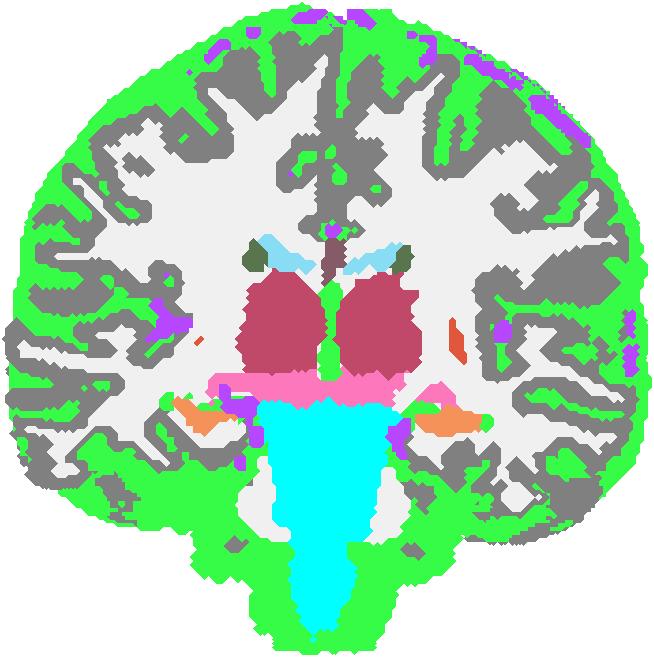} \\
        Domain, coronal
        \end{minipage}
        \begin{minipage}{3.0cm}
    \centering
        \includegraphics[width=2.1cm]{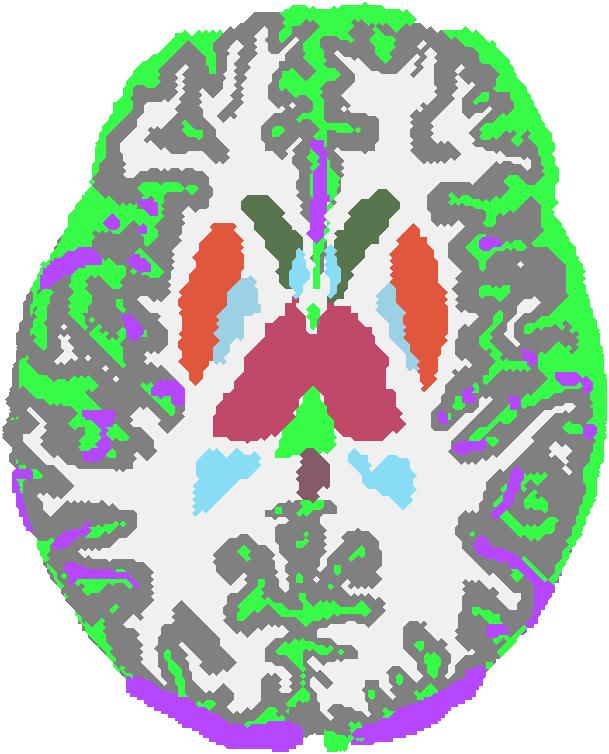} \\
        Domain, axial
        \end{minipage}  
                \begin{minipage}{3.0cm}
    \centering
            \includegraphics[width=3.0cm]{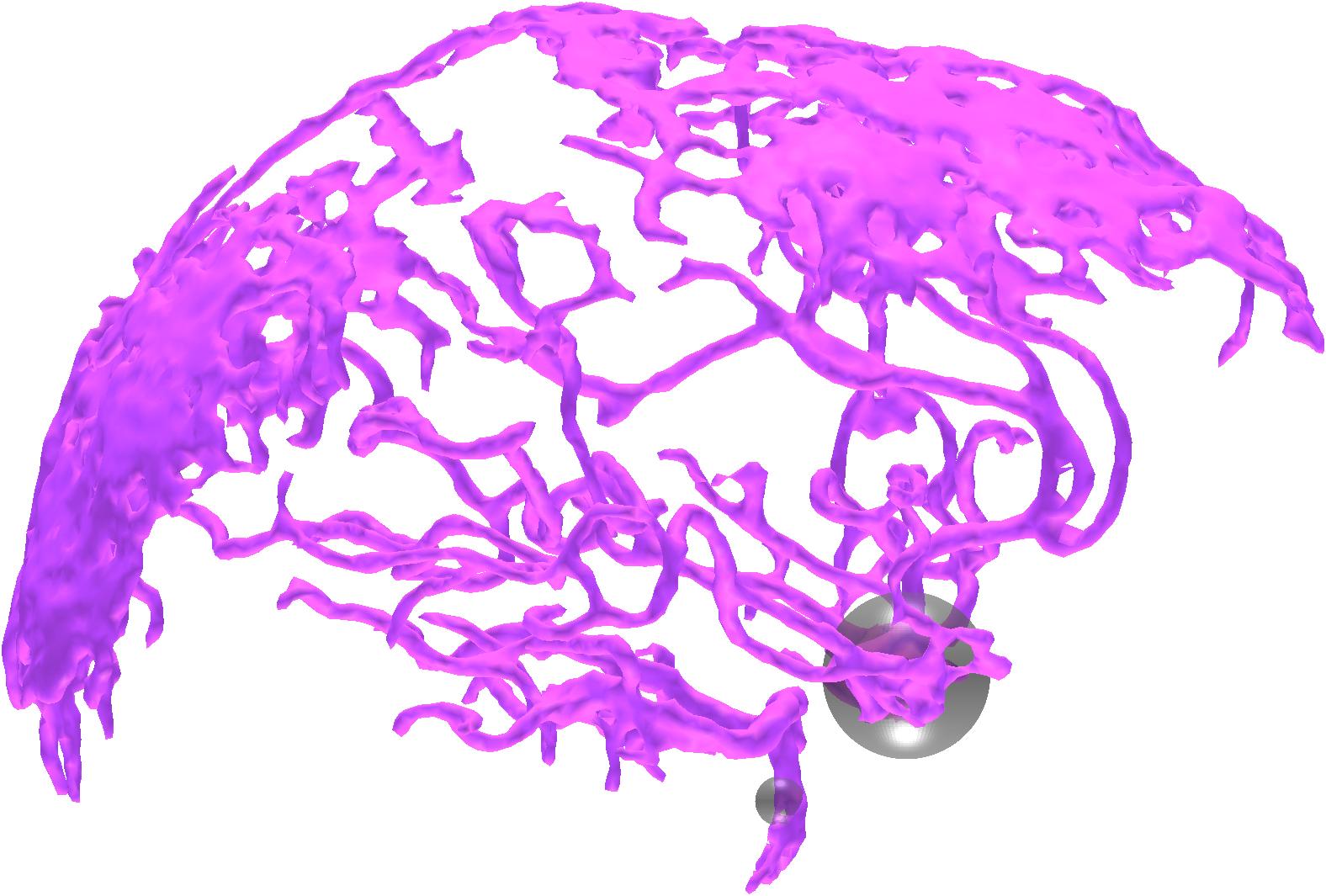} \\ Vessels, sagittal
            \end{minipage} 
                            \begin{minipage}{3.0cm}
    \centering
            \includegraphics[width=2.8cm]{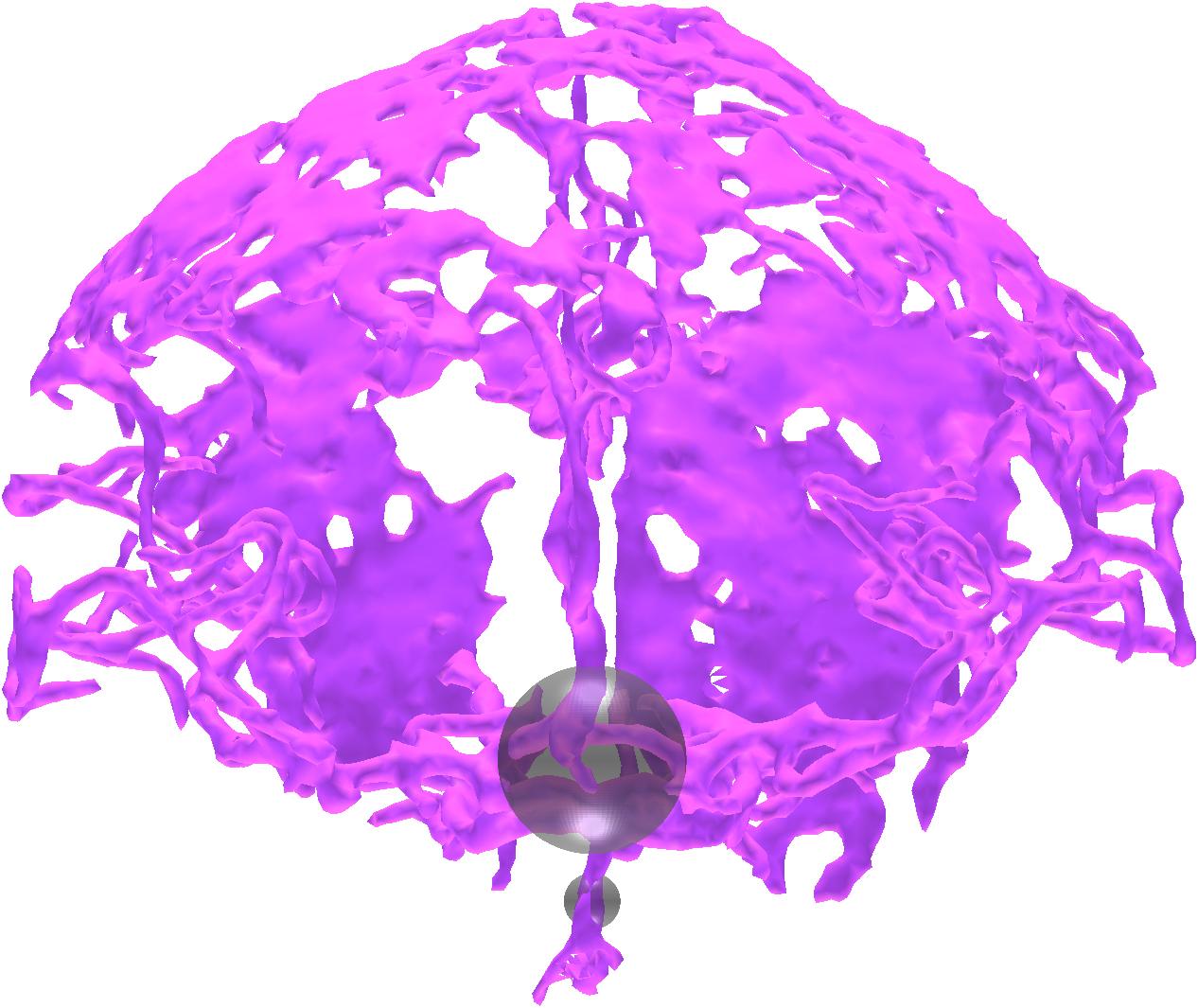} \\ Vessels, sagittal
            \end{minipage} \\
    \end{footnotesize}
    \caption{{\bf 1st to 3rd from left:} The computing domain consists of multiple brain compartments, of which the outermost one (green) depicts cerebrospinal fluid (CSF) enclosing vessels (violet), grey (grey), and white (white) matter compartments.  The other compartments correspond to subcortical nuclei. {\bf 4th and 5th from left:} A sagittal and coronal view of the arterial vessel compartment (violet) with two spherical sets depicting the spatial support $\mathcal{S}$ of the incoming pulse. } 
    \label{fig:domain}
\end{figure*}

The discretized arterial vessels are shown in Figure \ref{fig:domain}.Inspired by Choi {\em et al.} \cite{choi2020cerebral}, the Frangi filter was first applied to the MRI data slicewise. A final arterial model was created by aggregating the results. The process consisted of three steps:

\begin{enumerate}
\item  Frangi's algorithm was applied to both the preprocessed INV2 and T1w slices of the dataset separately. Both the preprocessed INV2 and T1w slices were subjected to Frangi's algorithm independently. The Frangi method was implemented using different parameters for each slice in Scikit-Image's \cite{van2014scikit}  package.
\item A mask was created by applying the filter to INV2 and T1w data and superposing the resulting two layers voxelwise.  Each voxel below a user-defined threshold level was set to zero, and those above it were set to one.
\item  The previous two steps were iterated for each MRI slice with a given orientation, and the results were aggregated.  Separate masks were obtained with sagittal, axial, and coronal sets of slices, labeling a voxel as a vessel if the sum of these masks was greater than or equal to two.
\end{enumerate}



\subsection{Numerical experiments}

\begin{table}[!h]
    \centering
        \caption{Gravitation has been set to be parallel to the z-axis. The conductivity of the blood $\sigma_f$ and the reference pressure were selected according to\cite{gabriel1996compilation} and \cite{blanco2017blood}, respectively. Blood density $\rho$,  $\mu$, total cerebral blood flow $Q$, pressure decay in arterioles $\vartheta$, diameters $D_a$, $D_c$ and $D_v$ of arterioles, capillaries, and venules (subtracting the total wall thicknesses, 2.0E-5, 2.0E-06, and 2E-6, respectively), and the relative total area fractions $\gamma_a$, $\gamma_c$ and $\gamma_v$, respectively, are based on  \cite{tu2015human,caro2012mechanics}. Microvessel densities have been suggested in \cite{kubikova2018numerical},  the cementation factor in \cite{j2001estimation}, and the viscosity parameters in \cite{weddell2015hemodynamic}. The smoothing parameters were chosen to be the smallest possible, which allowed obtaining a numerically stable NSE-solution.}
    \begin{footnotesize}
    \begin{tabular}{lllr}
    \hline
        Property & Param.   & Unit & Value\\
        \hline
        Gravitation (z-component) & $f^{(3)}$ &  m s\textsuperscript{-2} & -9.81\\
         Blood conductivity & $\sigma_f$ & S m\textsuperscript{-1} & 0.70  \\
        Average pressure difference & $\overline{p}$ & mmHg & 87  \\
        Arteriole length & $L$ & mm & 0.4 \\
         Blood density & $\rho$ & kg m\textsuperscript{-3} & 1050  \\
         Viscosity & $\mu$ & m\textsuperscript{2} Pa s & 4.0E-03 \\ 
         Total cerebral blood flow & $Q$ & ml min\textsuperscript{-1} & 750 \\
         Pressure decay in arterioles & $\vartheta$ & \% & 70 \\
         Arteriole diameter  & $D_a$ & m & 1.0E-05 \\
          Capillary diameter & $D_c$ & m & 7.0E-06 \\
           Venule diameter & $D_a$ & m & 1.8E-05 \\
        Arteriole total area fraction & $\gamma_a$ & \% & 25 \\
      Capillary area fraction & $\gamma_c$ & \% & 50 \\
    Venule area fraction & $\gamma_v$ & \% & 25 \\
    Microvessels in cerebral GM & $\xi$ & m\textsuperscript{-2} & 2.4E08 \\
    Microvessels in cerebral WM & $\xi$ & m\textsuperscript{-2} & 1.4E08 \\
    Microvessels in cerebellar GM & $\xi$ & m\textsuperscript{-2} & 3.0E08 \\
    Microvessels in cerebellar WM & $\xi$ & m\textsuperscript{-2} & 1.0E08 \\
    Microvessels in subcortical WM & $\xi$ & m\textsuperscript{-2} & 1.5E08 \\
     Microvessels in brainstem & $\xi$ & m\textsuperscript{-2} & 2.9E08 \\
       Cementation factor & $\beta$ & None & 5/3 \\
       Time step (s) & 2E-03 \\
       Leray regularization & $\epsilon$ & None & 0.03 \\ 
       Viscosity smoothing & $\epsilon$ & None & 0.01 \\
      Viscosity relaxation time & $\uplambda$ & s & 1.902\\
      Viscosity power law index & $\mathsf{n}$ & None & 0.22 \\
       Viscosity transition parameter & $\alpha$ & None & 1.25 \\ 
      \hline                        
    \end{tabular}
    \end{footnotesize}
    \label{tab:parameters}
\end{table}

We performed numerical experiments on modelling the dynamical blood conductivity via Archie's model (Section \ref{sec:archies_law}) applied to the volumetric blood concentration estimates given by Fick's law (Section \ref{sec:ficks_law}) together with the dynamic NSE solution for the cerebral blood circulation (Section \ref{sec:dynamic_form}). We used the present head model (Section \ref{sec:head_model}),  vessel segmentation (Section \ref{sec:vessels}), pulse form (Section \ref{sec:pulse}), and the parameter values described in Table  \ref{tab:parameters}. The smoothing and regularization parameters were set to be approximately at the minimal level required for numerical stability. In order to parallelize the spatio-temporal NSE system effectively, it was solved numerically using a graphics processing unit (GPU) Quadro RTX 8000 with 48 GB of random-access memory. Each simulation run took less than 30 minutes of total processing time and covered a simulated time interval of 6 s of which the first 3 s in the beginning were filtered out as a burn-in phase in which the amplitudes of the blood flow and viscosity per cardiac cycle adapt towards their long-term tendencies. The spatio-temporal data was stored for 300 equispaced temporal points between 3 and 6 s, i.e., for a step length of 0.01 s.  Our numerical solver implementation has been included as a part of the open Zeffiro Interface code package \cite{he2019zeffiro}.

\section{Results} 

The results have been included in Figures \ref{fig:arteries_60bpm}-- \ref{fig:conductivity_80bpm} which show the mean, peak, and standard deviation (STD)  for the simulated time-lapse of distributions. The results for pressure, velocity, and viscosity corresponding to the heart rates of 60 and 80 bpm can be found in Figures \ref{fig:arteries_60bpm} and \ref{fig:arteries_80bpm},   for concentration in Figures \ref{fig:concentration_60bpm} and \ref{fig:concentration_80bpm}, and for conductivity in Figures \ref{fig:conductivity_60bpm} and \ref{fig:conductivity_80bpm}. To filter out outliers, the peak values of pressure, velocity, and viscosity were set to be the 90 \% quantile of the distribution obtained.

For 60 bpm heart rate, the mean pressure, velocity, and viscosity over the simulated time interval vary spatially in the ranges 90--98 mmHg, 0--0.26 m/s, and 0.0035--0.0059 Pa s. For the peak, these ranges are 96--98 mmHg, 0--0.82 m/s, and 0.0040--0.0125 Pa s, and for STD, these ranges are 0--6.3 mmHg, 0--0.20 m/s, and 0--0.0042 Pa s, respectively.

For 80 bpm heart rate, the mean ranges are 90--98 mmHg, 0--0.53 m/s, and 0.0035--0.0052 Pa s. Compared to the lower heart rate, the pressure range is virtually the same, the velocity is elevated due to more blood flow, and the viscosity is somewhat lowered as a consequence of the greater velocity. The peak ranges are 96--127 mmHg, 0--1.4 m/s, and 0--0.0084  Pa s, and STD ranges 0--15.2 mmHg, 0--0.35 m/s, and 0--0.0032 Pa s. Compared to the lower heart rate, these show similar tendencies as in the case of the mean, but with a stronger elevation in pressure.

The mean, peak, and STD of volumetric blood concentration predicted by Fick's law can be found in Figures \ref{fig:concentration_60bpm} and \ref{fig:arteries_80bpm} for the heart rates of 60 and 80 bpm, respectively. In both cases, the concentration can be observed to decay to zero at a distance of 1–2 cm from the arterial blood vessels. The concentration is one or close to one in the vicinity of the vessels. STD differs from zero outside this region, varying between 0-0.05. The peak concentration is slightly higher than the mean. A heart rate of 80 bpm yields a slightly greater overall concentration than 60 bpm.

The mean, peak, and STD of conductivity following from the concentration distributions via Archie's law are shown in Figures \ref{fig:conductivity_60bpm} and \ref{fig:conductivity_80bpm} for the heart rates of 60 and 80 bpm, respectively. The effect of elevated concentrations is clearly visible in these distributions. The effect of the concentration is slightly more spread at the peak as compared to the mean. STD varies in the range 0-0.55 S/m corresponding to the areas of concentration variation. For the 80 bpm heart rate, the conductivity close to the arterial blood vessels is, overall, slightly lifted as compared to the case of 60 bpm.

\label{sec:Results}

\begin{figure}[h!]
\begin{footnotesize}
    \centering
    \begin{minipage}{2.5cm}
    \centering
    \includegraphics[width=2.5cm]{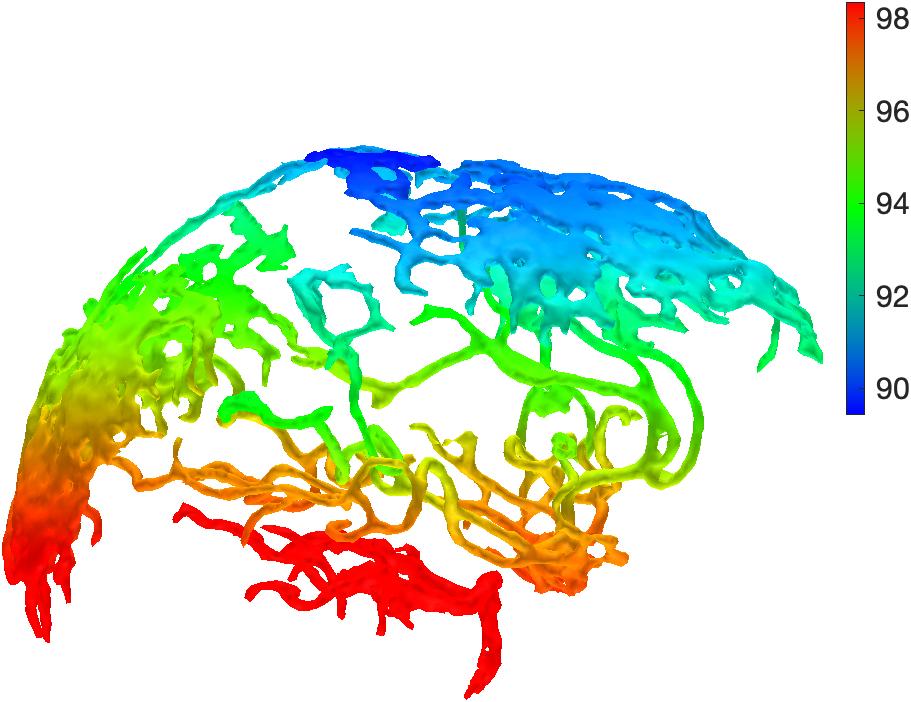} \\
    Mean pressure
\end{minipage}
 \begin{minipage}{2.5cm}
    \centering
        \includegraphics[width=2.5cm]{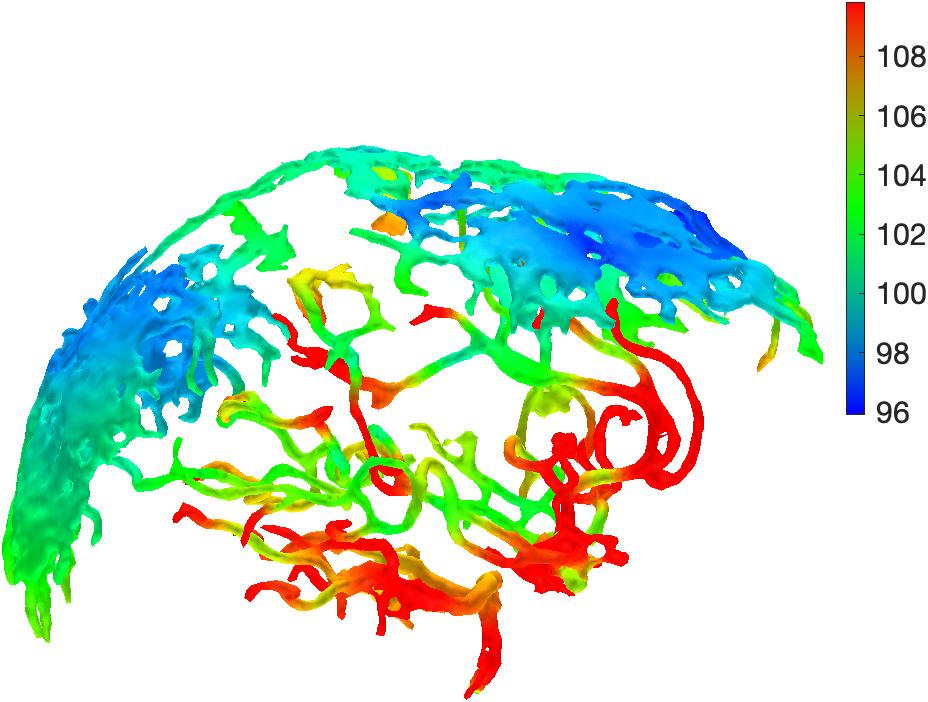} \\
        Peak pressure
        \end{minipage}
        \begin{minipage}{2.5cm}
    \centering
        \includegraphics[width=2.5cm]{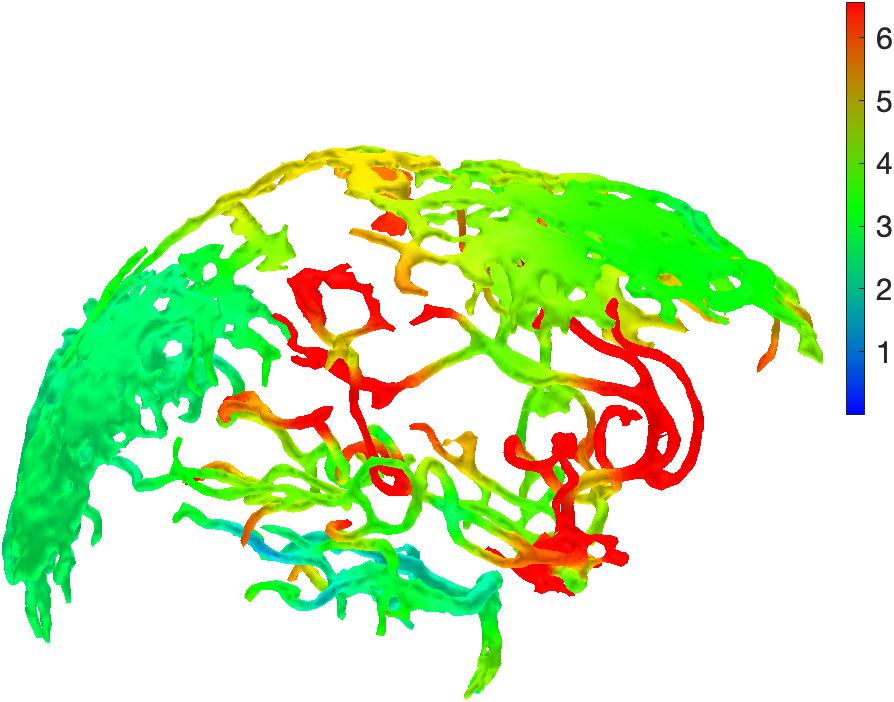} \\
        STD pressure
        \end{minipage} \\ \vskip0.2cm 
                \begin{minipage}{2.5cm}
    \centering
            \includegraphics[width=2.5cm]{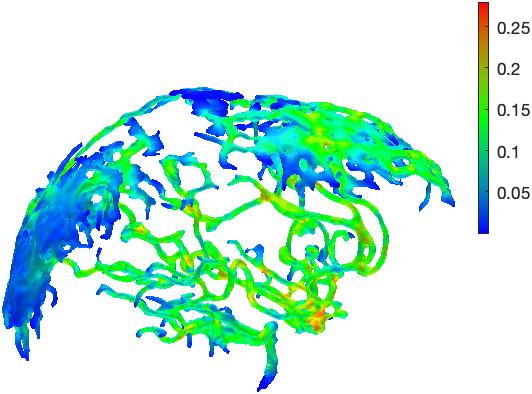} \\ Mean velocity
            \end{minipage} 
             \begin{minipage}{2.5cm}
    \centering
        \includegraphics[width=2.5cm]{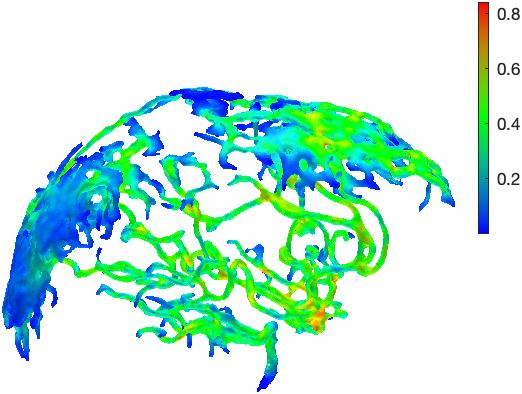} \\
        Peak velocity
        \end{minipage}
\begin{minipage}{2.5cm}
    \centering
        \includegraphics[width=2.5cm]{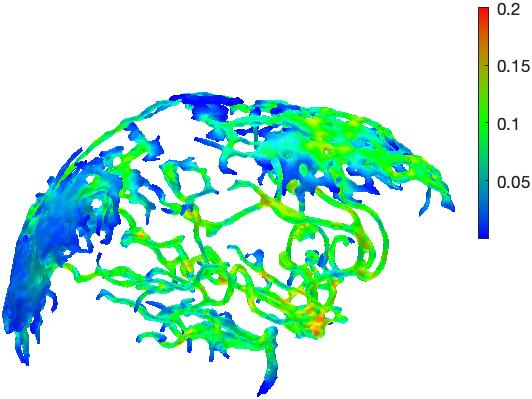} \\
        STD velocity
        \end{minipage} \\ \vskip0.2cm
            \begin{minipage}{2.5cm}
    \centering
            \includegraphics[width=2.5cm]{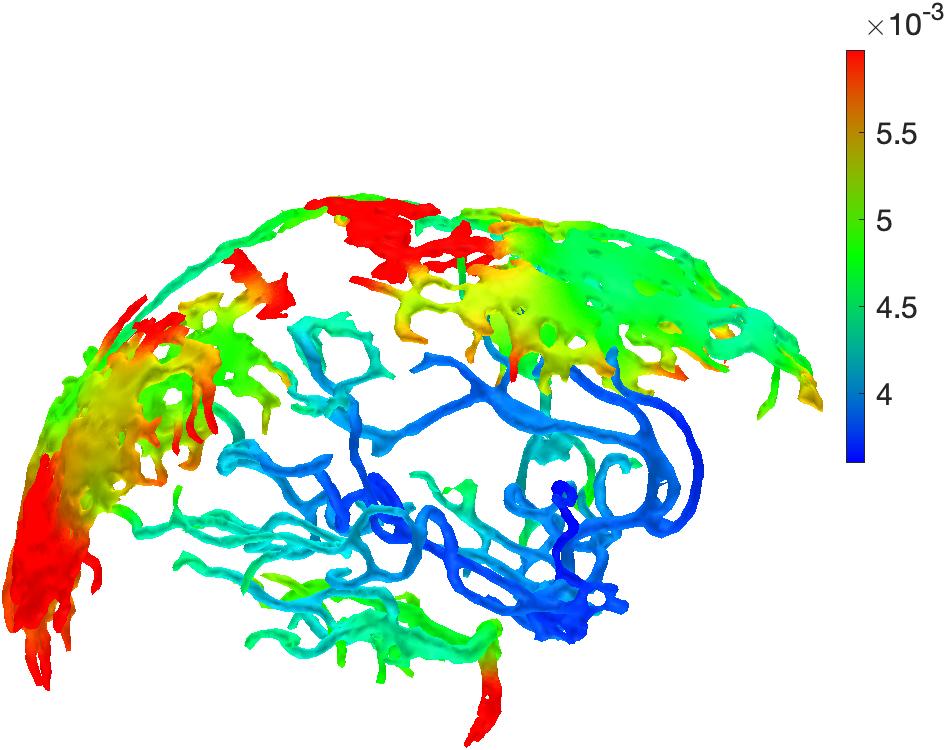} \\ Mean viscosity
            \end{minipage} 
             \begin{minipage}{2.5cm}
    \centering
        \includegraphics[width=2.5cm]{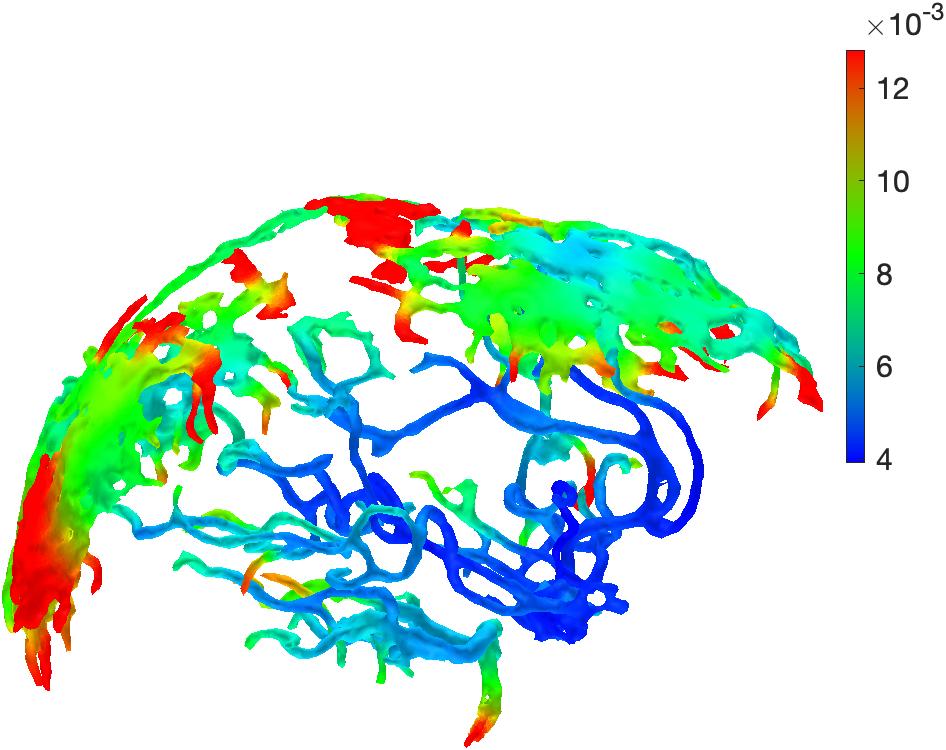} \\
        Peak viscosity
        \end{minipage}
\begin{minipage}{2.5cm}
    \centering
        \includegraphics[width=2.5cm]{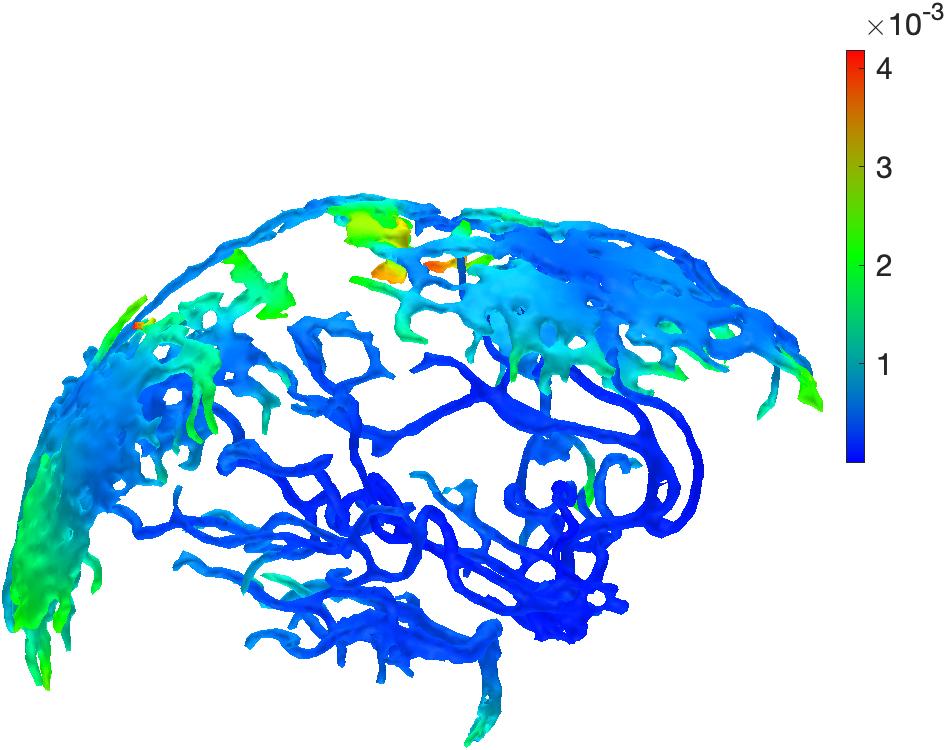} \\
        STD viscosity
        \end{minipage} \\ 
    \end{footnotesize}
    \caption{Mean, peak and standard deviation  (STD) estimates of  blood pressure, velocity, and viscosity in arterial circulation for 60 bpm heart rate.}
    \label{fig:arteries_60bpm}
\end{figure}

\begin{figure}[h!]
\begin{footnotesize}
    \centering
    \begin{minipage}{2.5cm}
    \centering
    \includegraphics[width=2.5cm]{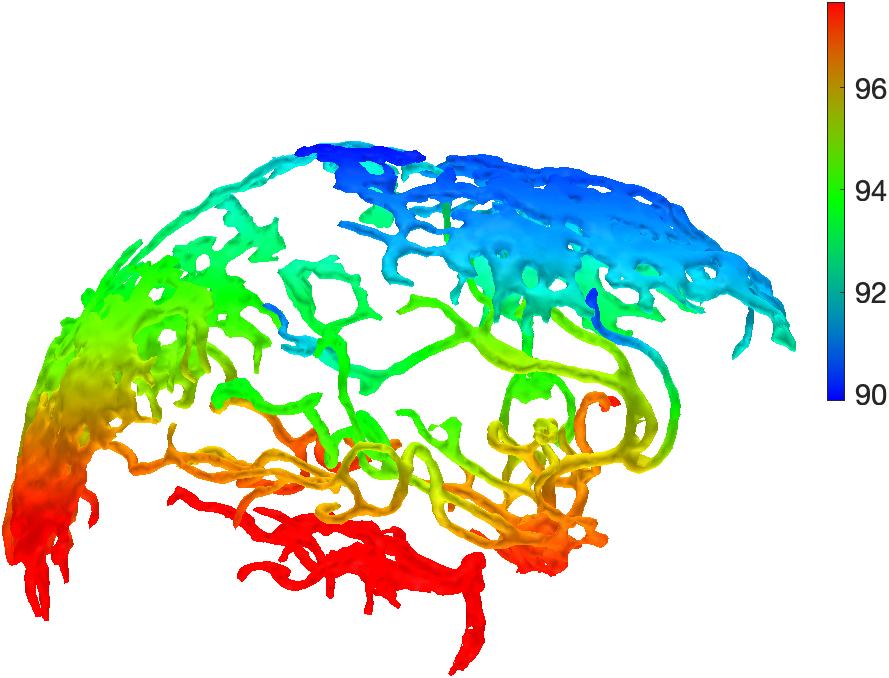} \\
    Mean pressure
\end{minipage}
 \begin{minipage}{2.5cm}
    \centering
        \includegraphics[width=2.5cm]{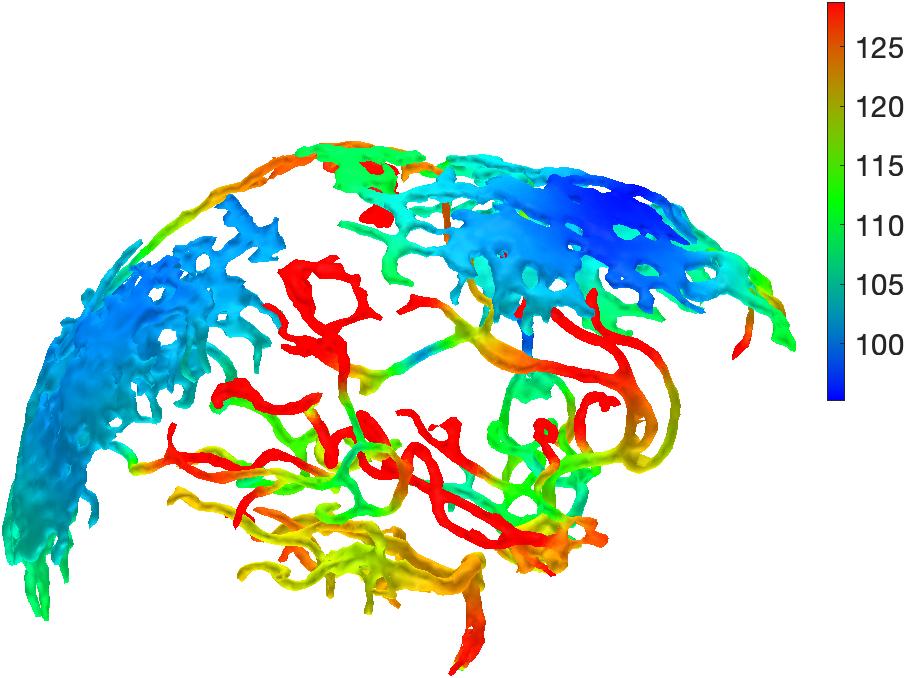} \\
        Peak pressure
        \end{minipage}
        \begin{minipage}{2.5cm}
    \centering
        \includegraphics[width=2.5cm]{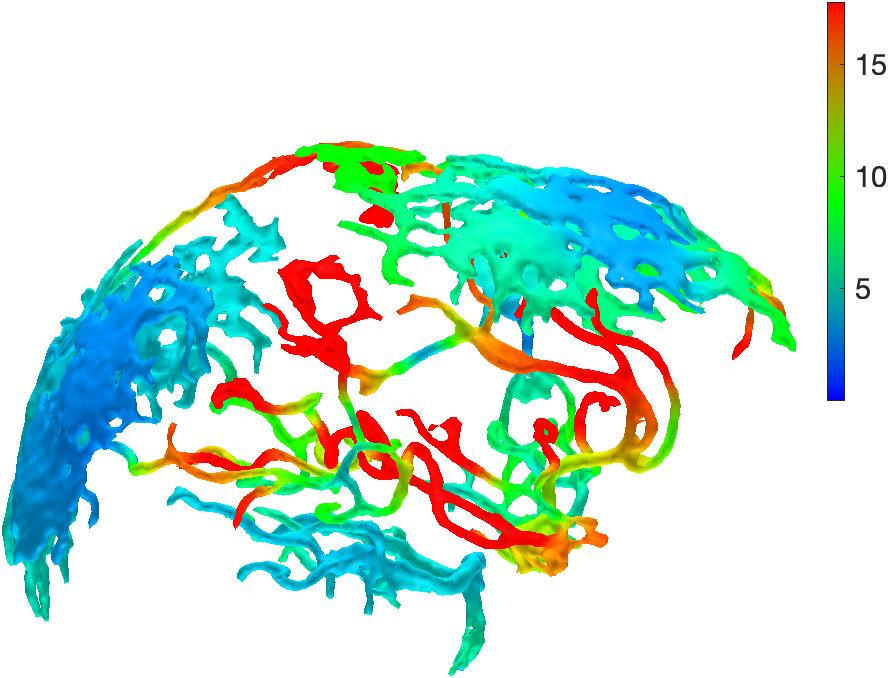} \\
        STD pressure
        \end{minipage} \\ \vskip0.2cm 
                \begin{minipage}{2.5cm}
    \centering
            \includegraphics[width=2.5cm]{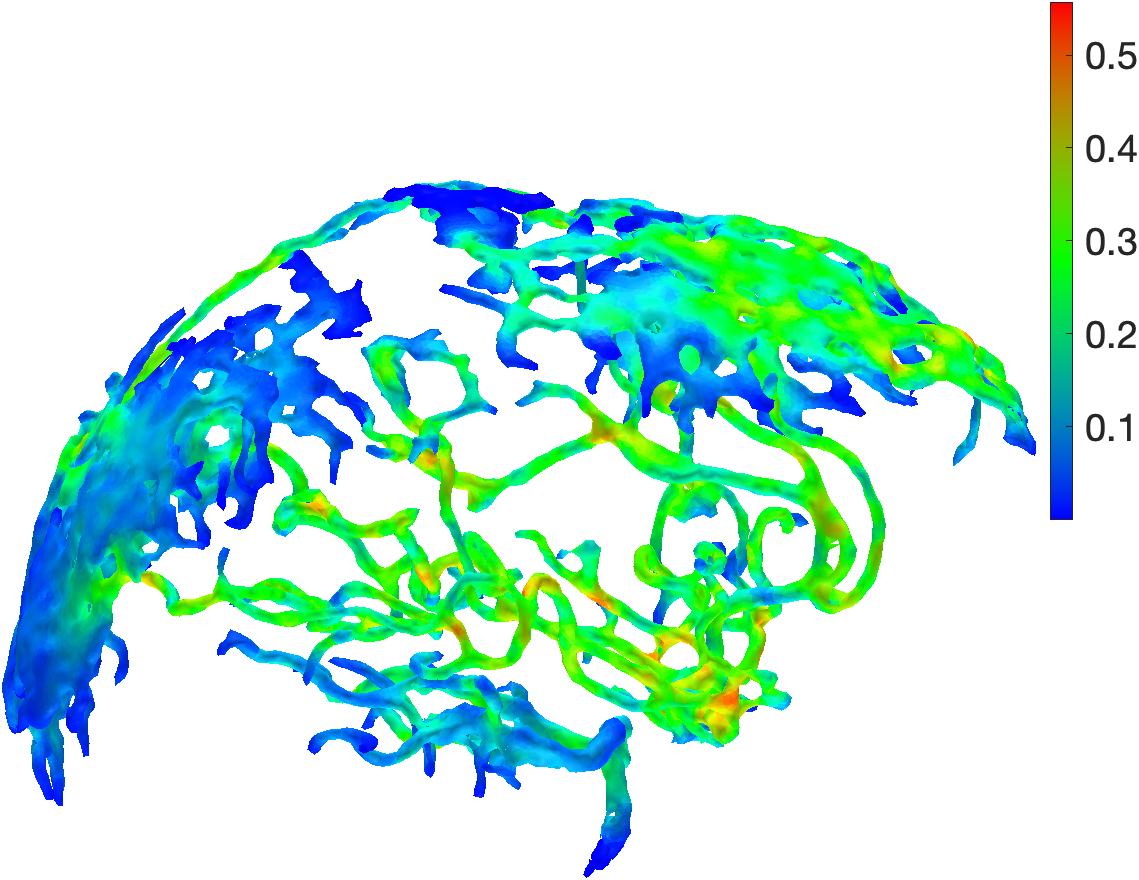} \\ Mean velocity
            \end{minipage} 
             \begin{minipage}{2.5cm}
    \centering
        \includegraphics[width=2.5cm]{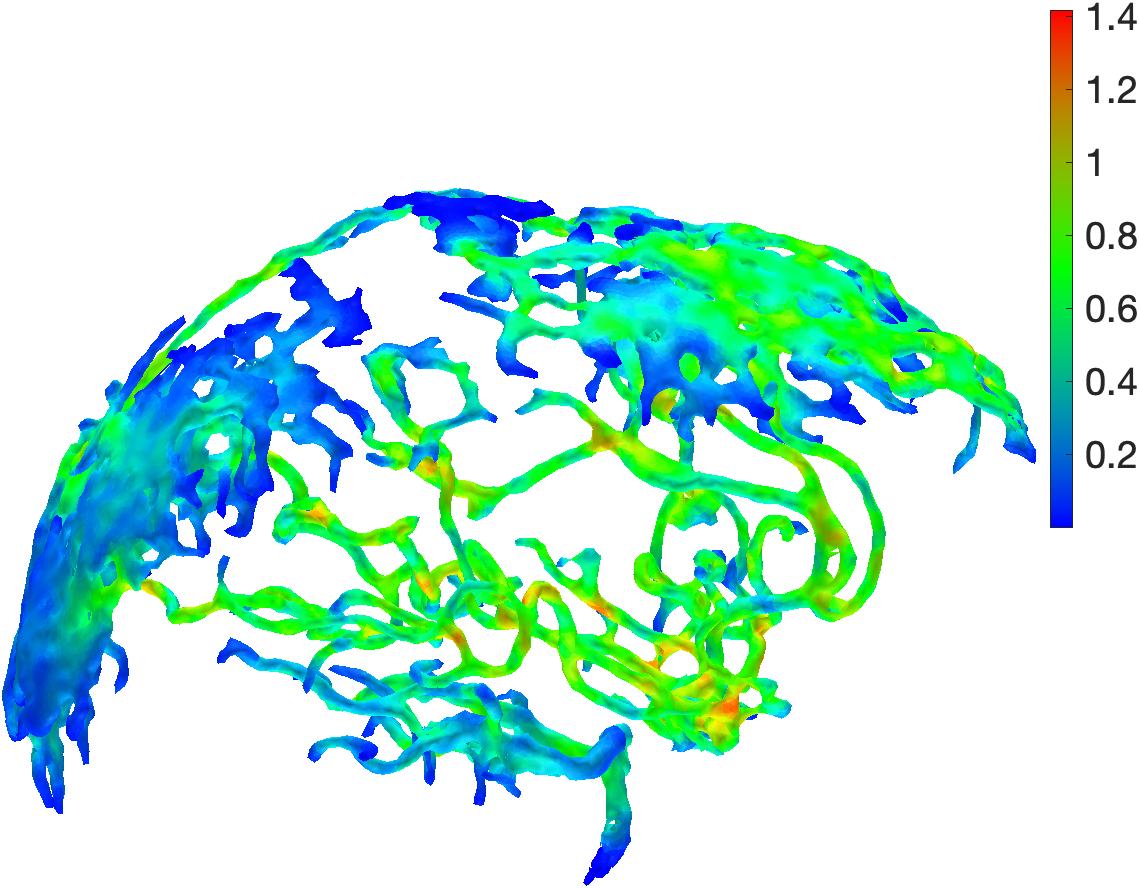} \\
        Peak velocity
        \end{minipage}
\begin{minipage}{2.5cm}
    \centering
        \includegraphics[width=2.5cm]{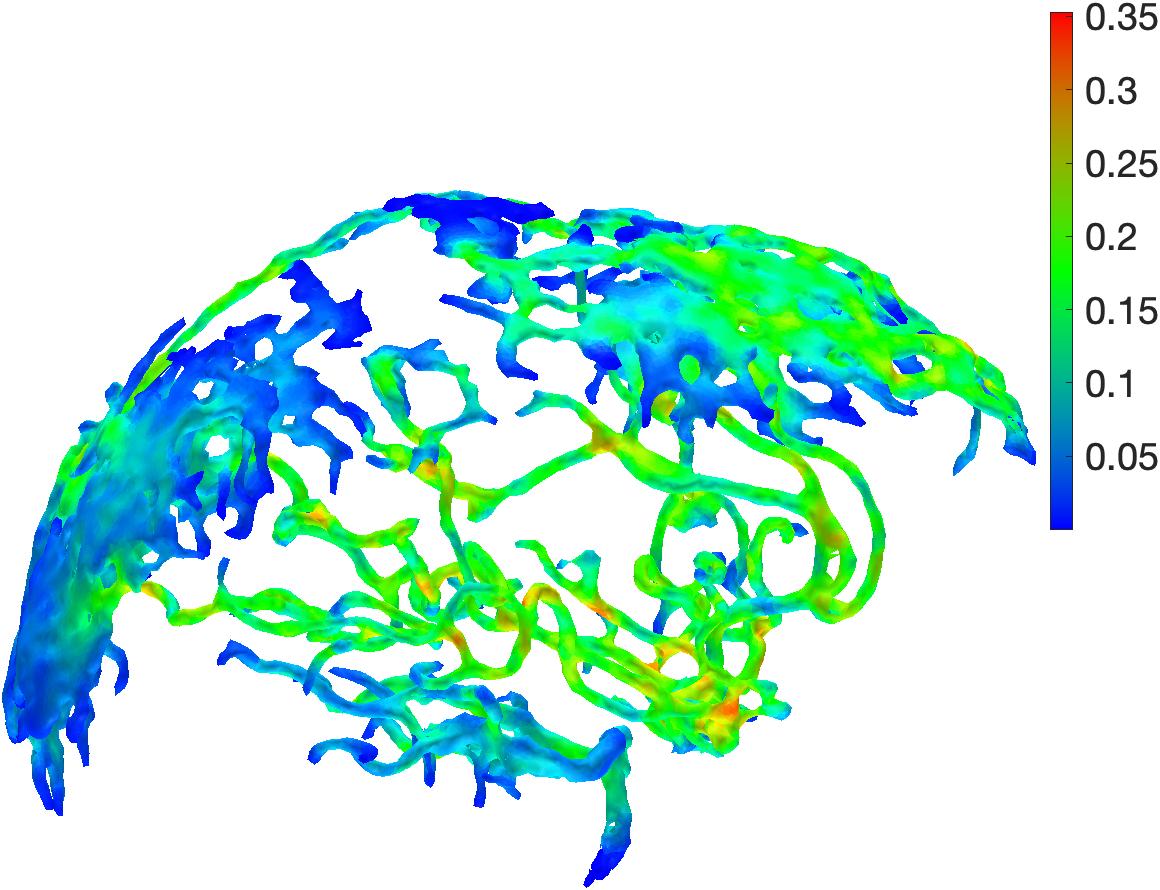} \\
        STD velocity
        \end{minipage} \\ \vskip0.2cm
            \begin{minipage}{2.5cm}
    \centering
            \includegraphics[width=2.5cm]{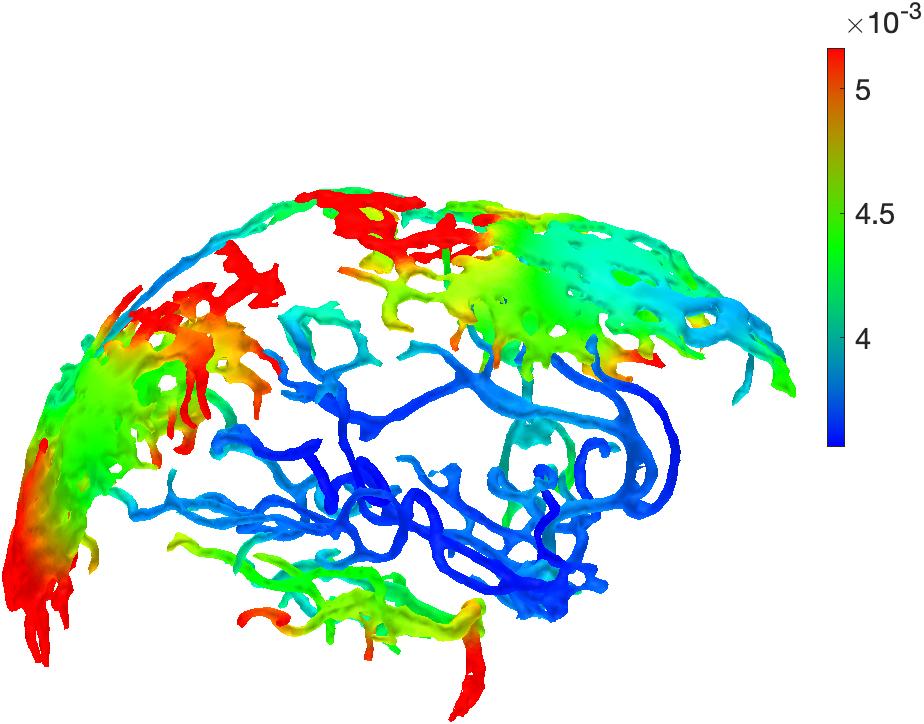} \\ Mean viscosity
            \end{minipage} 
             \begin{minipage}{2.5cm}
    \centering
        \includegraphics[width=2.5cm]{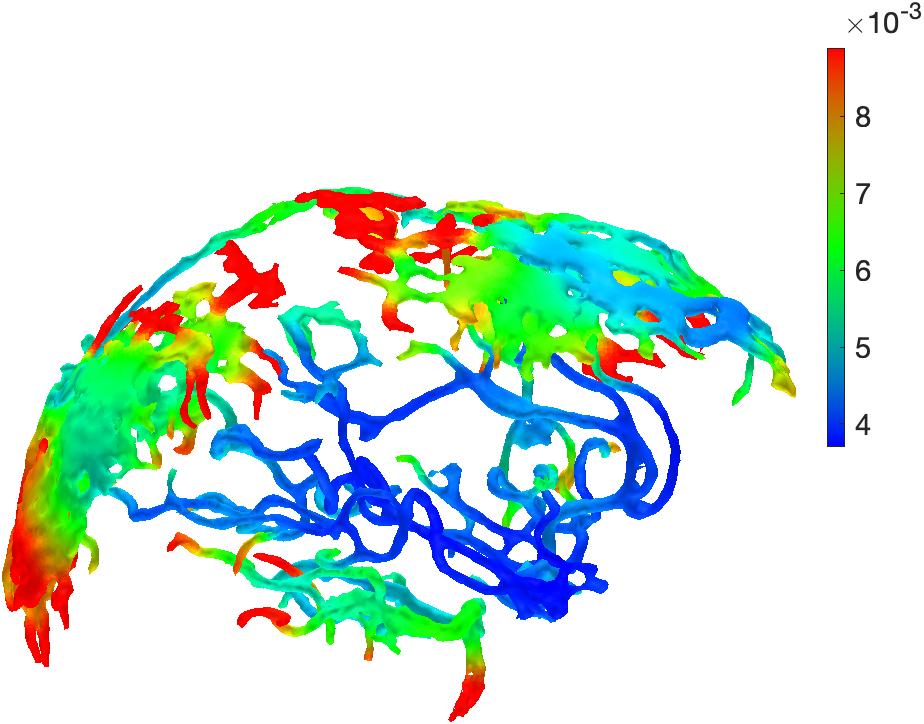} \\
        Peak viscosity
        \end{minipage}
\begin{minipage}{2.5cm}
    \centering
        \includegraphics[width=2.5cm]{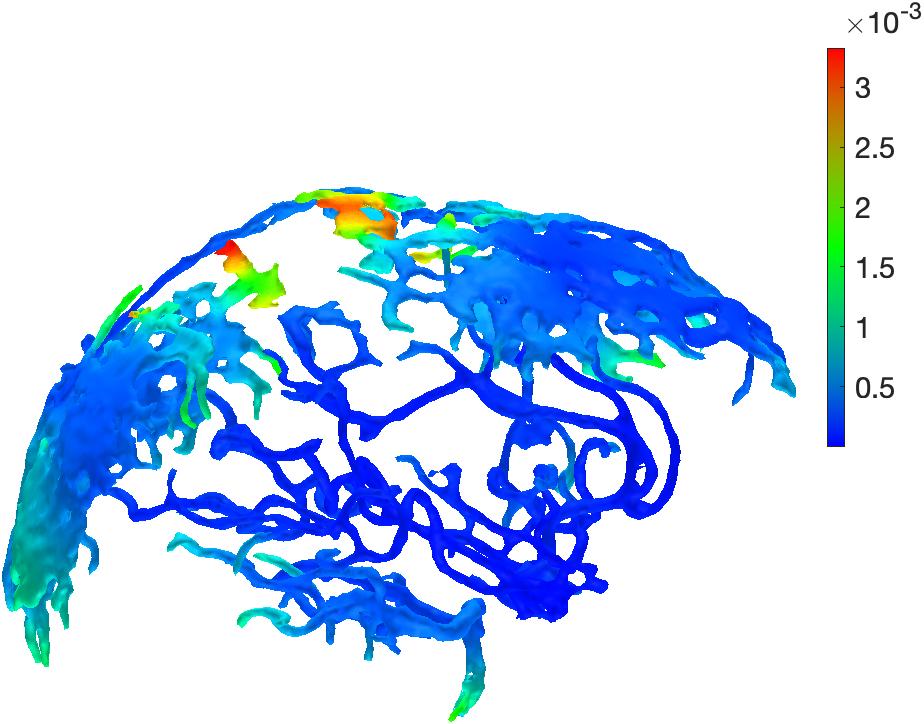} \\
        STD viscosity
        \end{minipage} \\ 
    \end{footnotesize}
    \caption{
    Mean, peak and standard deviation (STD) estimates of  blood pressure, velocity, and viscosity  for  80 bpm heart rate.}
    \label{fig:arteries_80bpm}
\end{figure}

\begin{figure}[h!]
\begin{footnotesize}
    \centering
    \begin{minipage}{2.5cm}
    \centering
    \includegraphics[width=2.5cm]{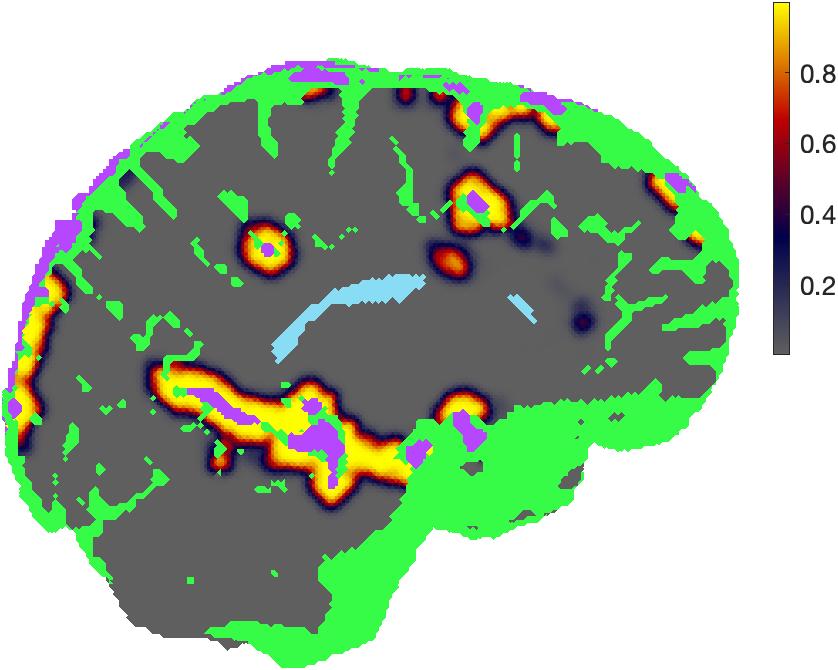} \\
    Mean, sagittal
\end{minipage}
 \begin{minipage}{2.5cm}
    \centering
            \includegraphics[width=2.5cm]{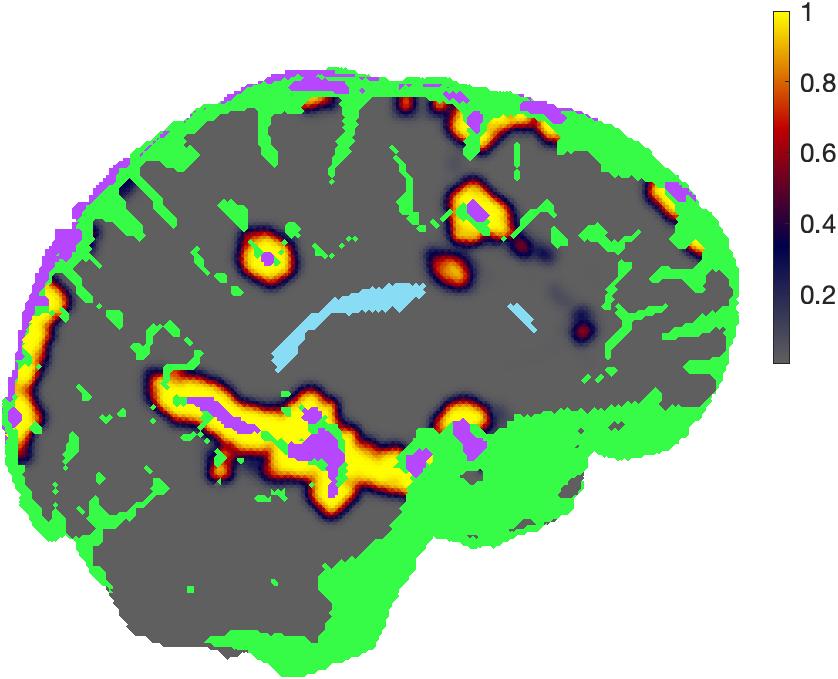} \\ peak, sagittal
            \end{minipage} 
 \begin{minipage}{2.5cm}
    \centering
            \includegraphics[width=2.5cm]{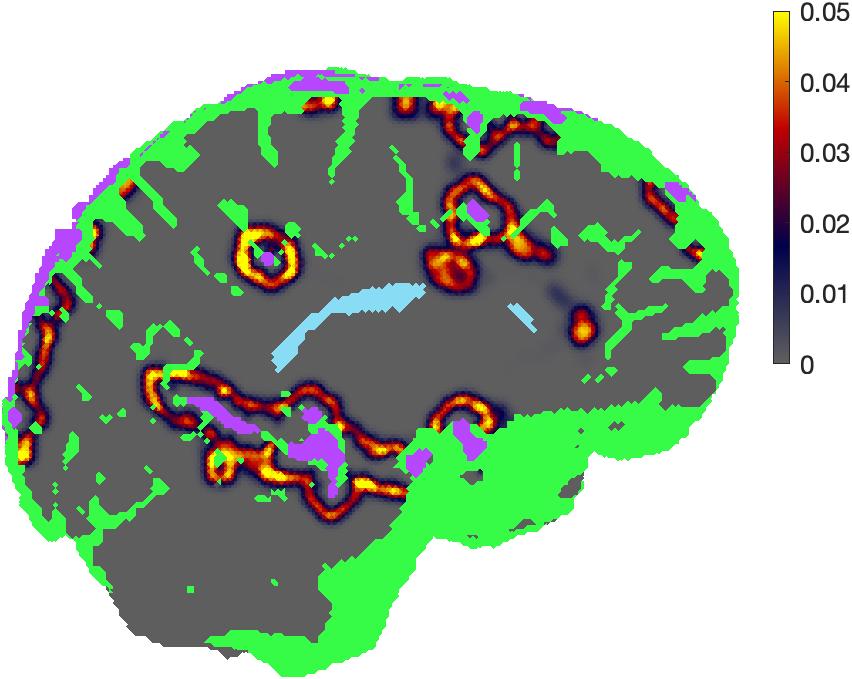} \\ STD, sagittal
            \end{minipage} \\ \vskip0.2cm 
 \begin{minipage}{2.5cm}
    \centering
        \includegraphics[width=2.4cm]{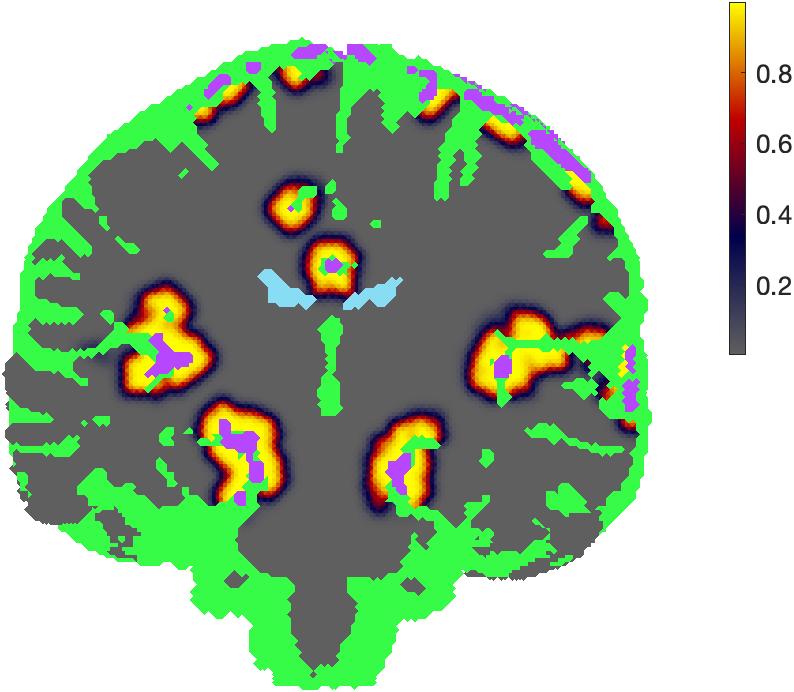} \\
        Mean, coronal 
        \end{minipage}
        \begin{minipage}{2.5cm}
    \centering
        \includegraphics[width=2.4cm]{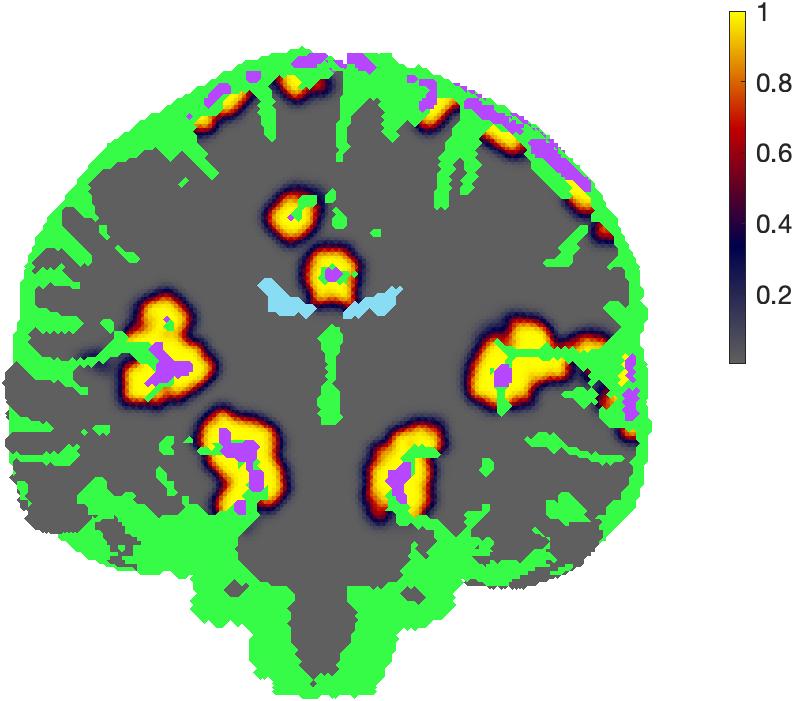} \\
        Peak, coronal
        \end{minipage}
             \begin{minipage}{2.5cm}
    \centering
        \includegraphics[width=2.4cm]{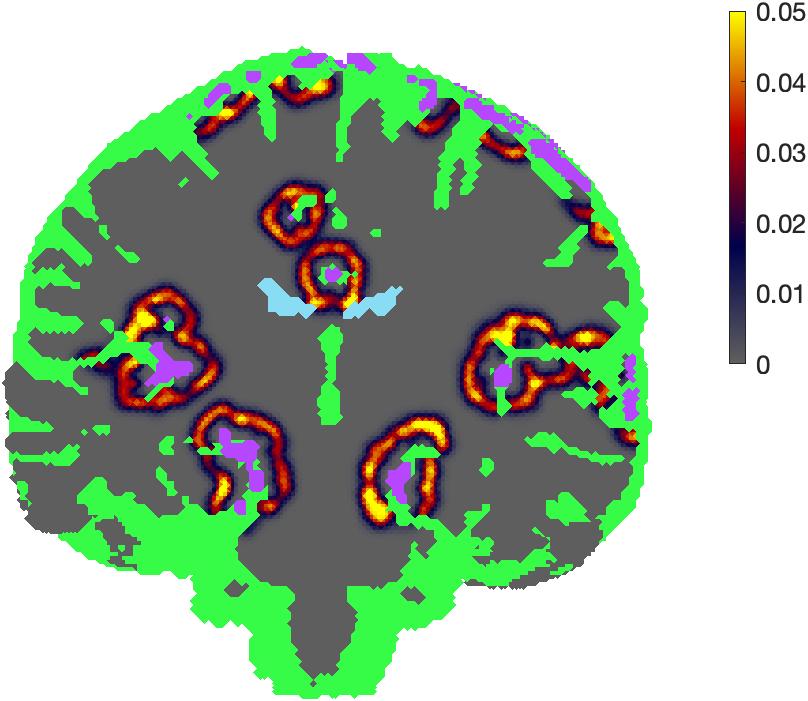} \\
        STD, coronal
        \end{minipage} \\ \vskip0.2cm
  \begin{minipage}{2.5cm}
    \centering
        \includegraphics[width=2.1cm]{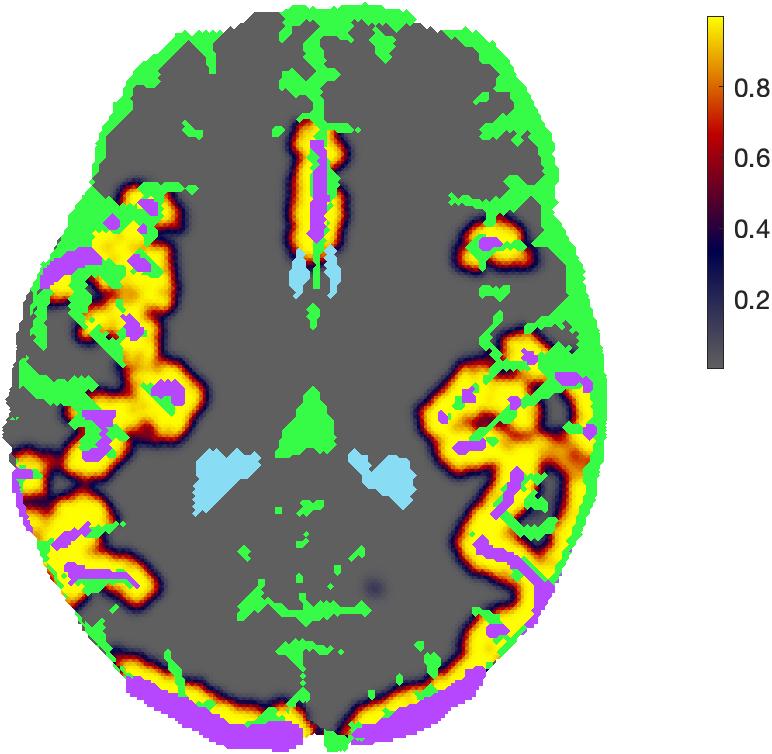} \\
      Mean, axial
        \end{minipage} 
\begin{minipage}{2.5cm}
    \centering
        \includegraphics[width=2.1cm]{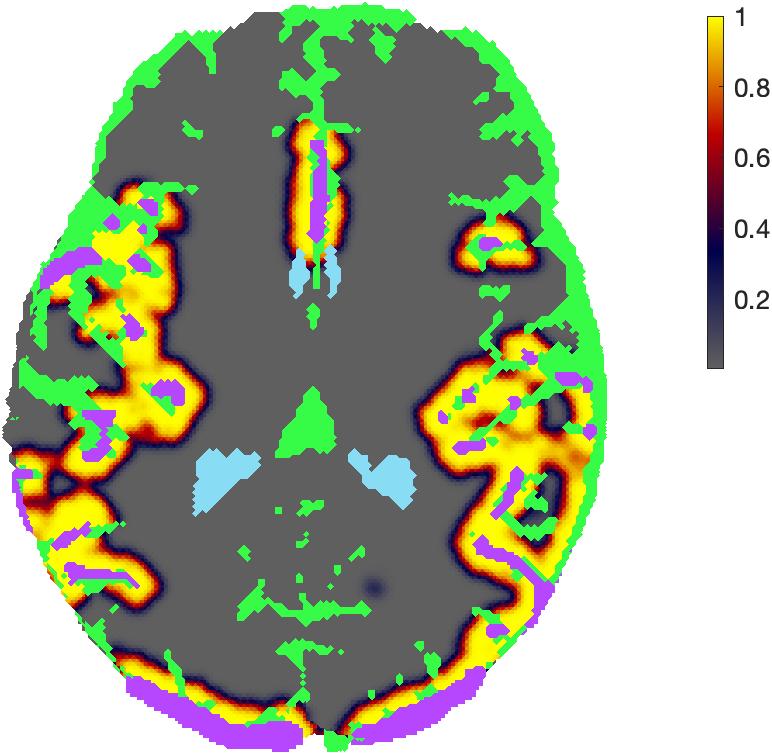} \\
        Peak, axial
        \end{minipage} 
\begin{minipage}{2.5cm}
    \centering
        \includegraphics[width=2.1cm]{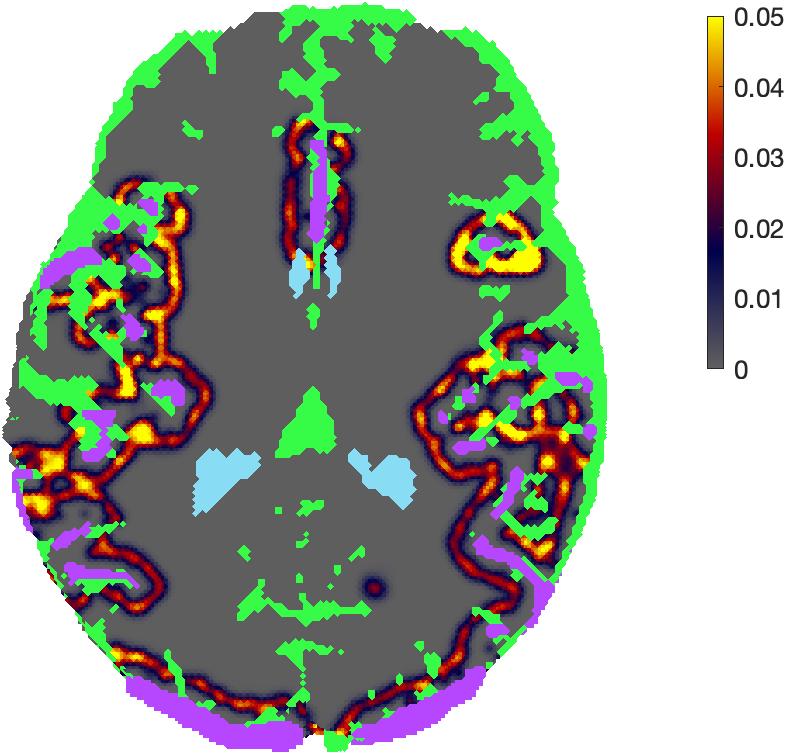} \\
        STD, axial
        \end{minipage} \\    
    \end{footnotesize}
    \caption{Estimates of  volumetric blood concentration in microcirculation for  80 bpm heart rate.}
    \label{fig:concentration_60bpm}
\end{figure}

\begin{figure}[h!]
\begin{footnotesize}
    \centering
    \begin{minipage}{2.5cm}
    \centering
    \includegraphics[width=2.5cm]{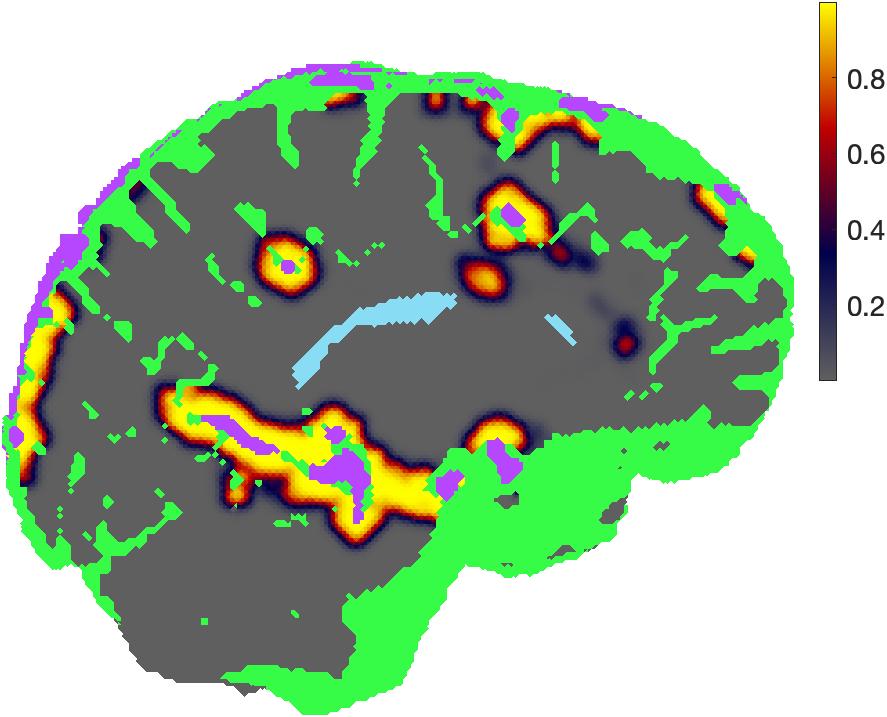} \\
    Mean, sagittal
\end{minipage}
 \begin{minipage}{2.5cm}
    \centering
            \includegraphics[width=2.5cm]{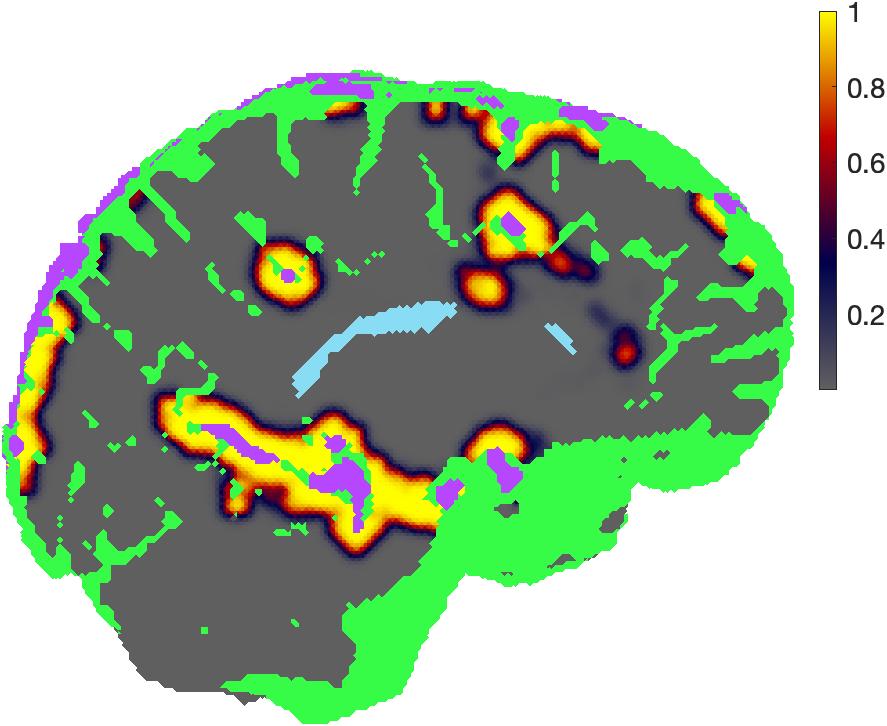} \\ Peak, sagittal
            \end{minipage} 
 \begin{minipage}{2.5cm}
    \centering
            \includegraphics[width=2.5cm]{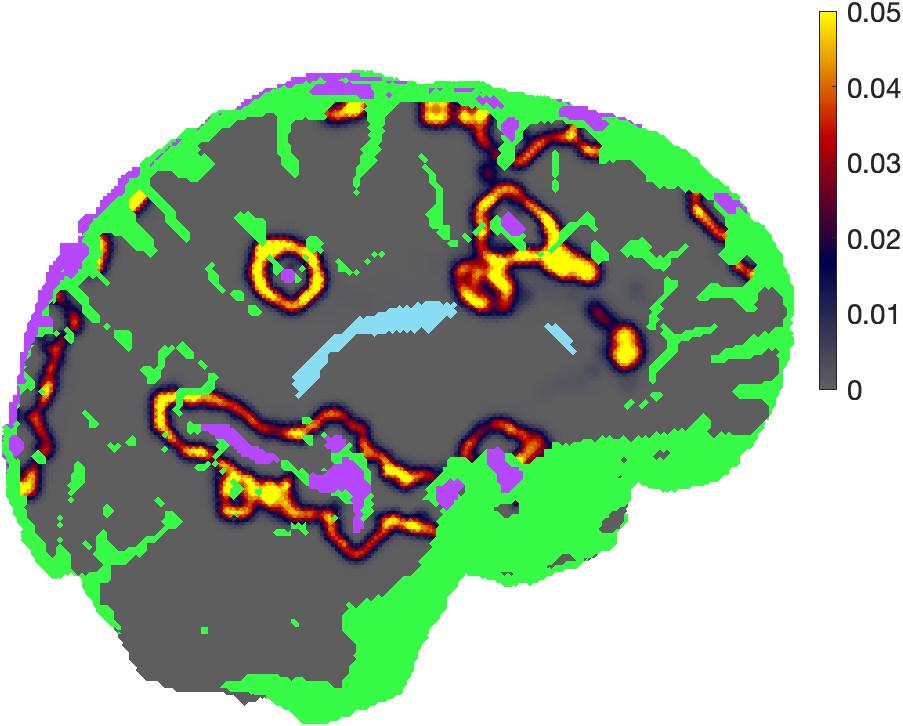} \\ STD, sagittal
            \end{minipage} \\ \vskip0.2cm 
 \begin{minipage}{2.5cm}
    \centering
        \includegraphics[width=2.4cm]{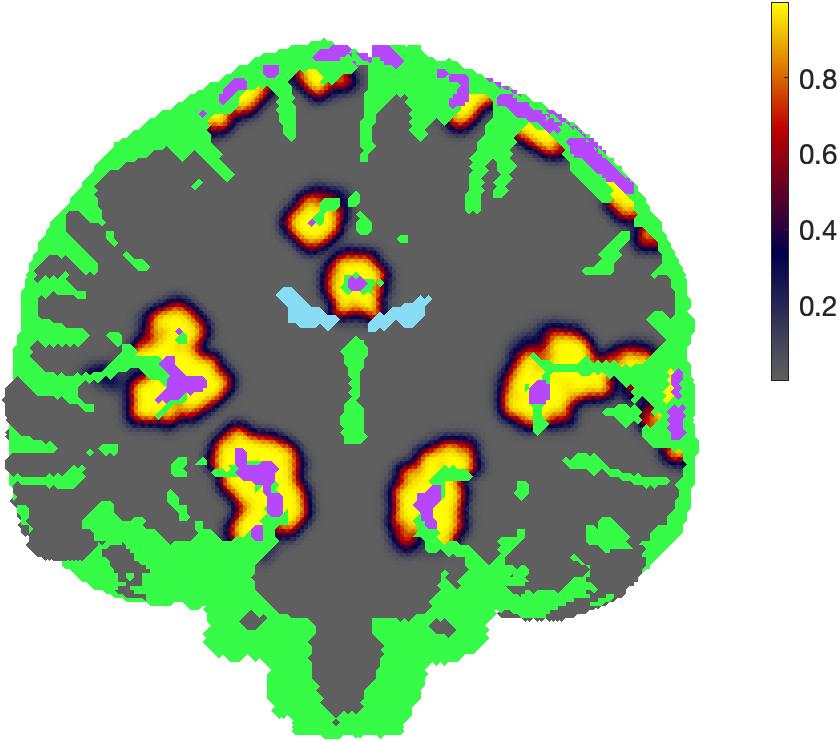} \\
        Mean, coronal 
        \end{minipage}
        \begin{minipage}{2.5cm}
    \centering
        \includegraphics[width=2.4cm]{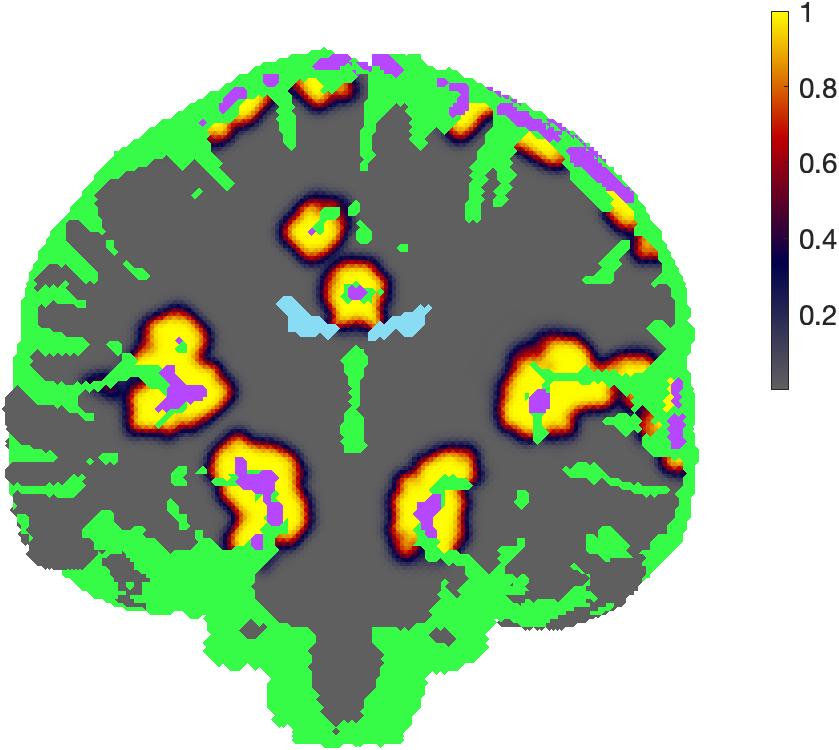} \\
        Peak, coronal
        \end{minipage}
             \begin{minipage}{2.5cm}
    \centering
        \includegraphics[width=2.4cm]{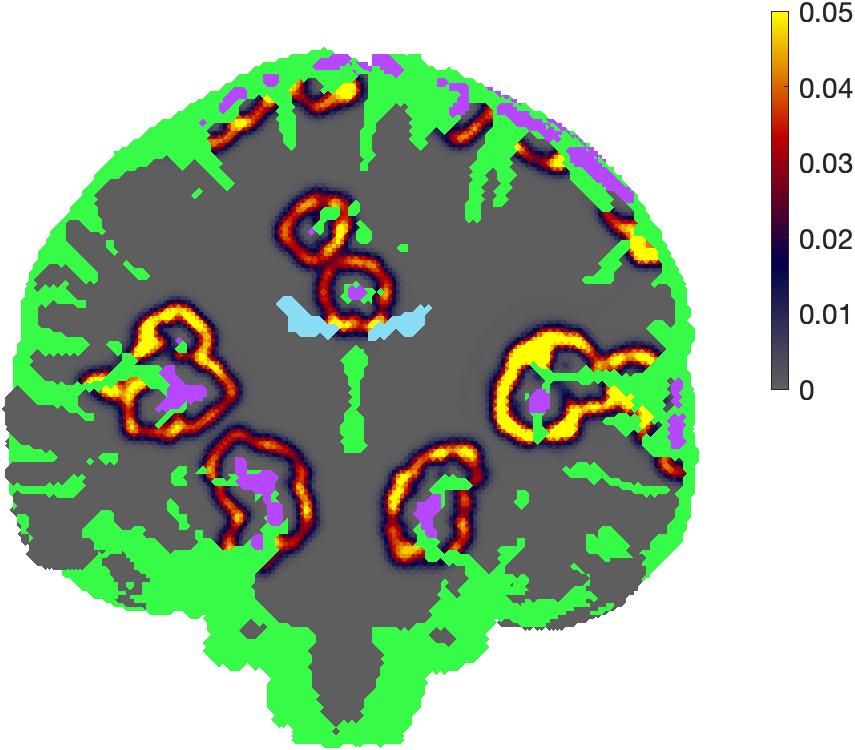} \\
        STD, coronal
        \end{minipage} \\ \vskip0.2cm
  \begin{minipage}{2.5cm}
    \centering
        \includegraphics[width=2.1cm]{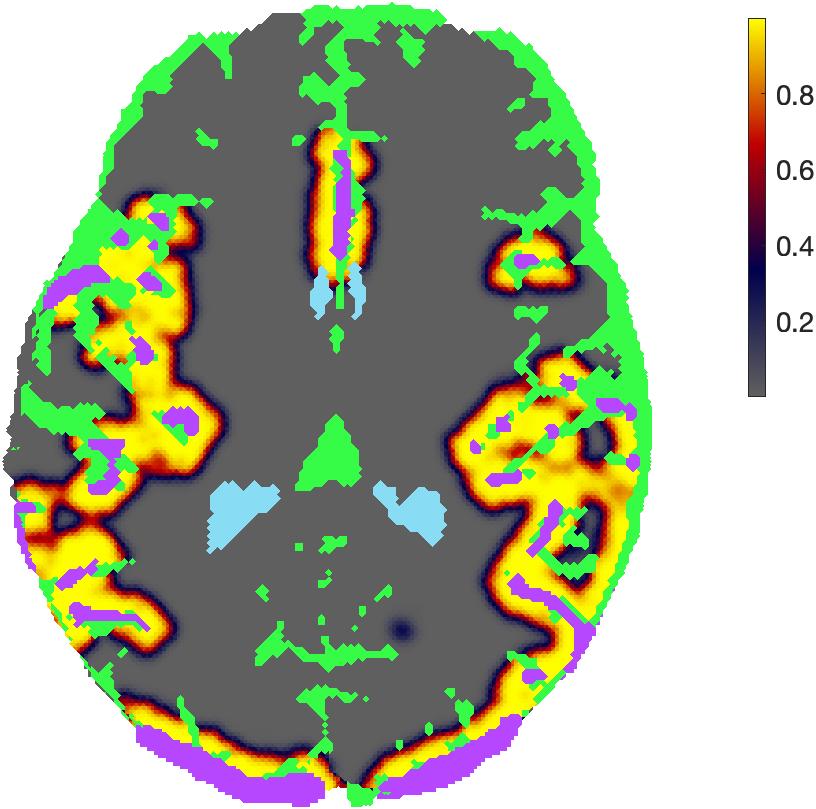} \\
      Mean, axial
        \end{minipage} 
\begin{minipage}{2.5cm}
    \centering
        \includegraphics[width=2.1cm]{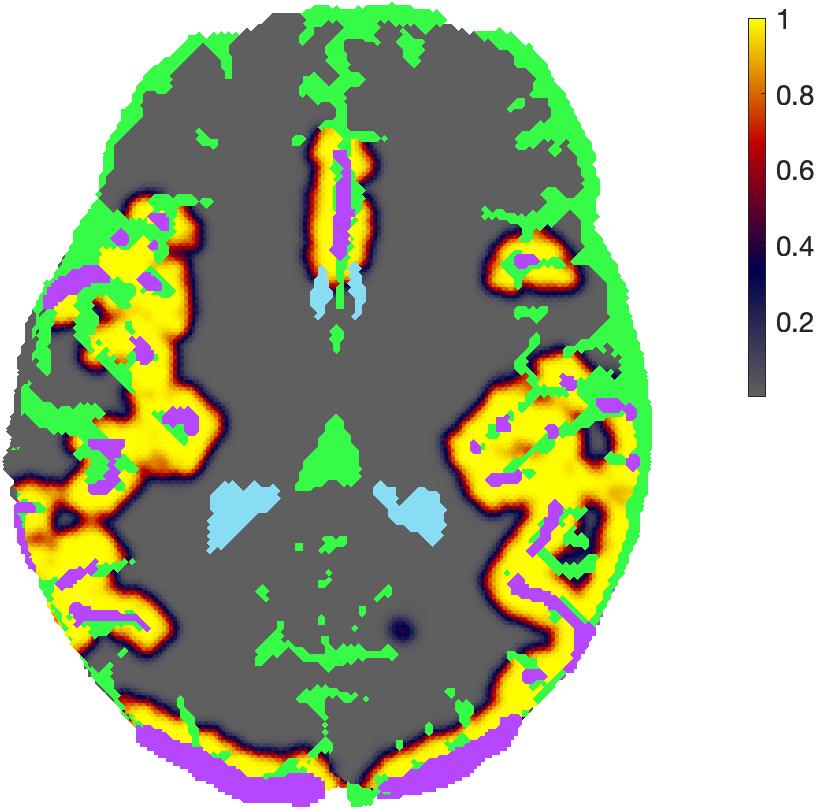} \\
        Peak, axial
        \end{minipage} 
\begin{minipage}{2.5cm}
    \centering
        \includegraphics[width=2.1cm]{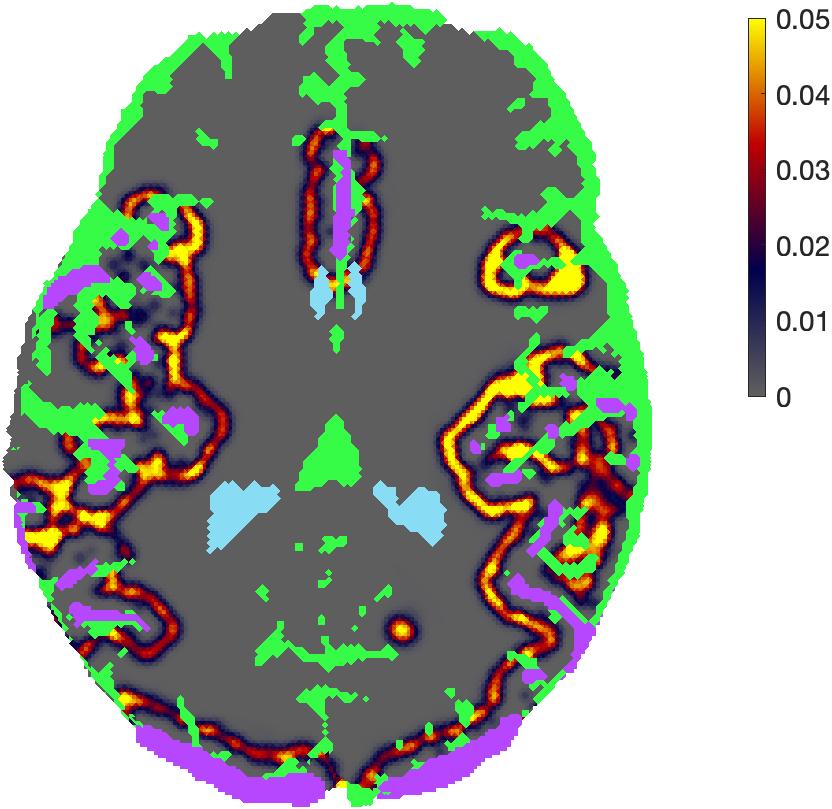} \\
        STD, axial
        \end{minipage} \\    
    \end{footnotesize}
    \caption{Estimates of  volumetric blood concentration for 80 bpm heart rate.}
    \label{fig:concentration_80bpm}
\end{figure}

\begin{figure}[h!]
\begin{footnotesize}
    \centering
    \begin{minipage}{2.2cm}
    \centering
    \includegraphics[width=2.1cm]{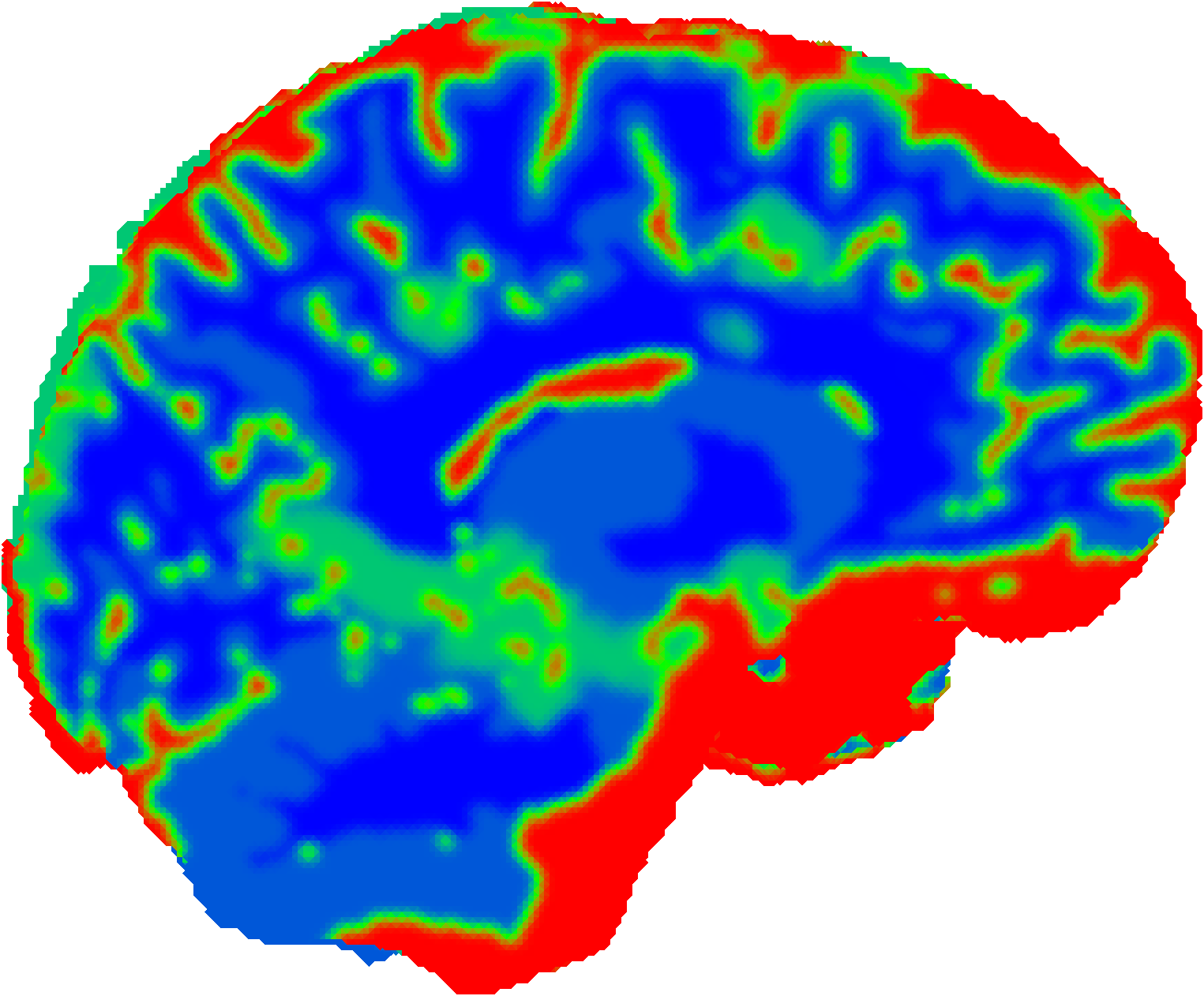} \\
    Mean, sagittal
\\ \vskip0.2cm
    \centering
        \includegraphics[width=2.0cm]{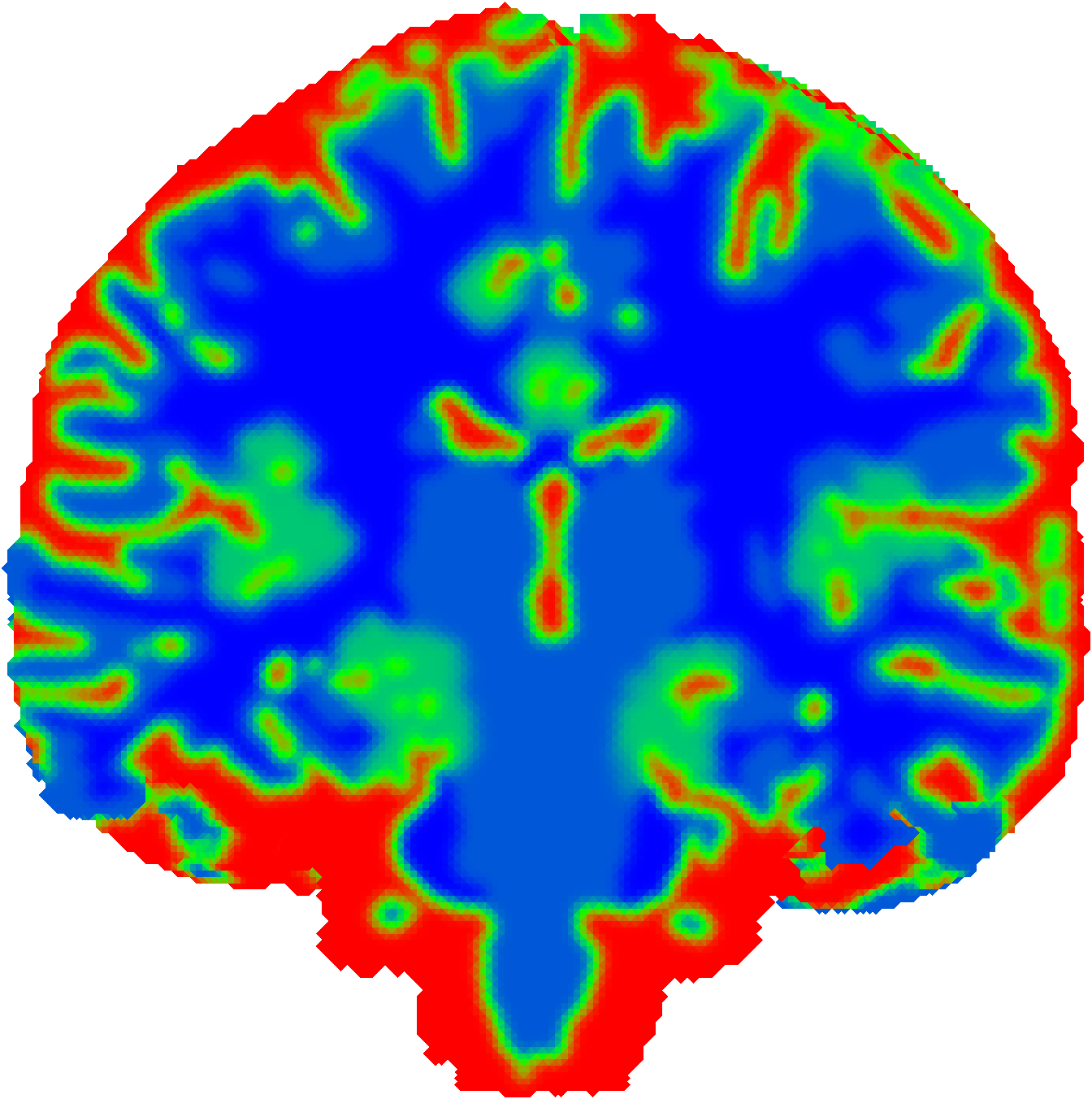} \\
        Mean, coronal 
       \\ \vskip0.2cm
    \centering
        \includegraphics[width=1.7cm]{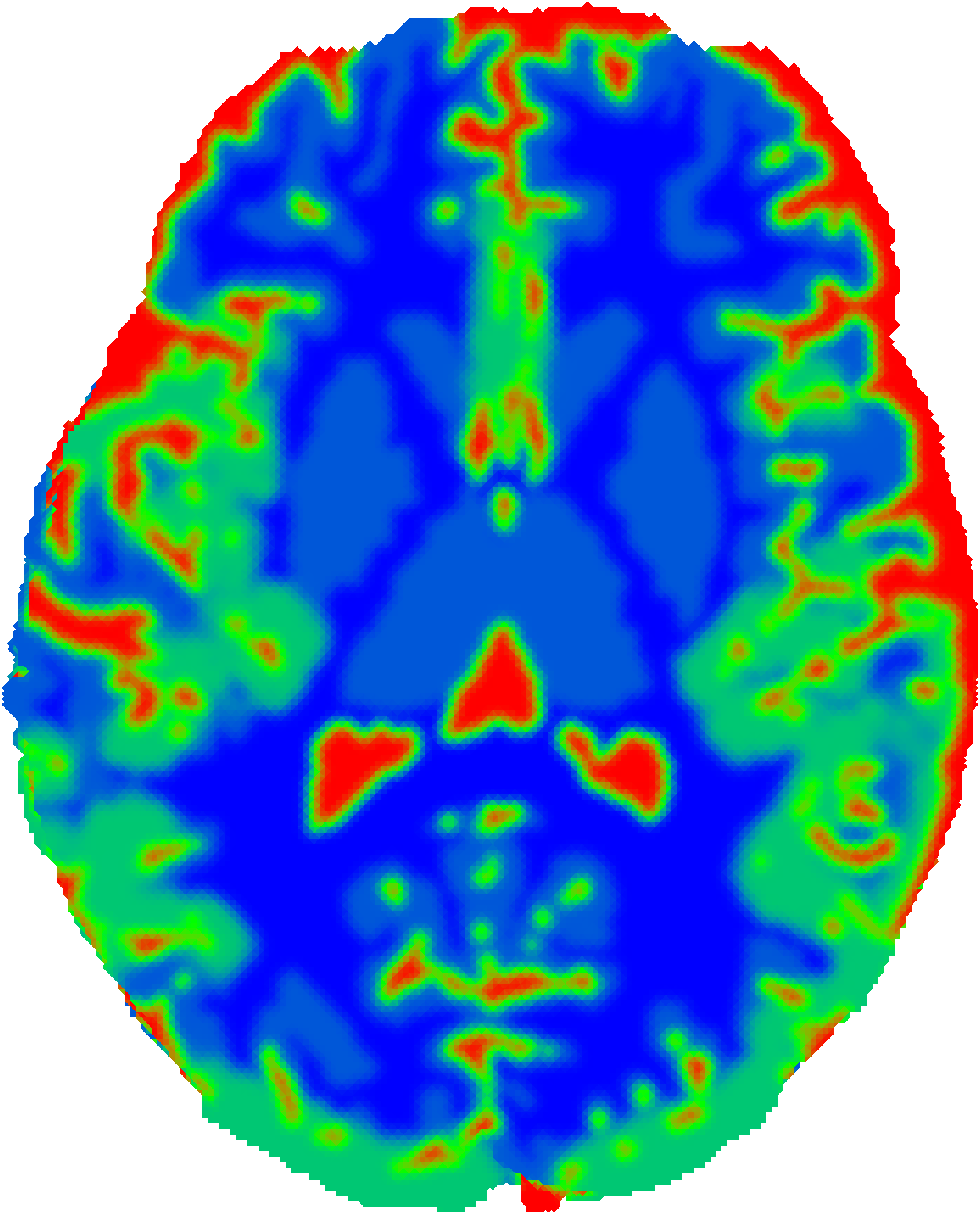} \\
      Mean, axial
        \end{minipage} 
        \begin{minipage}{0.5cm}
              \includegraphics[height=4.5cm]{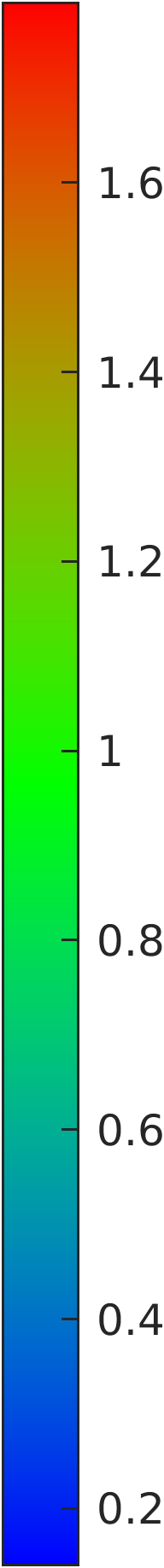}
              \end{minipage}
 \begin{minipage}{2.2cm}
    \centering
            \includegraphics[width=2.1cm]{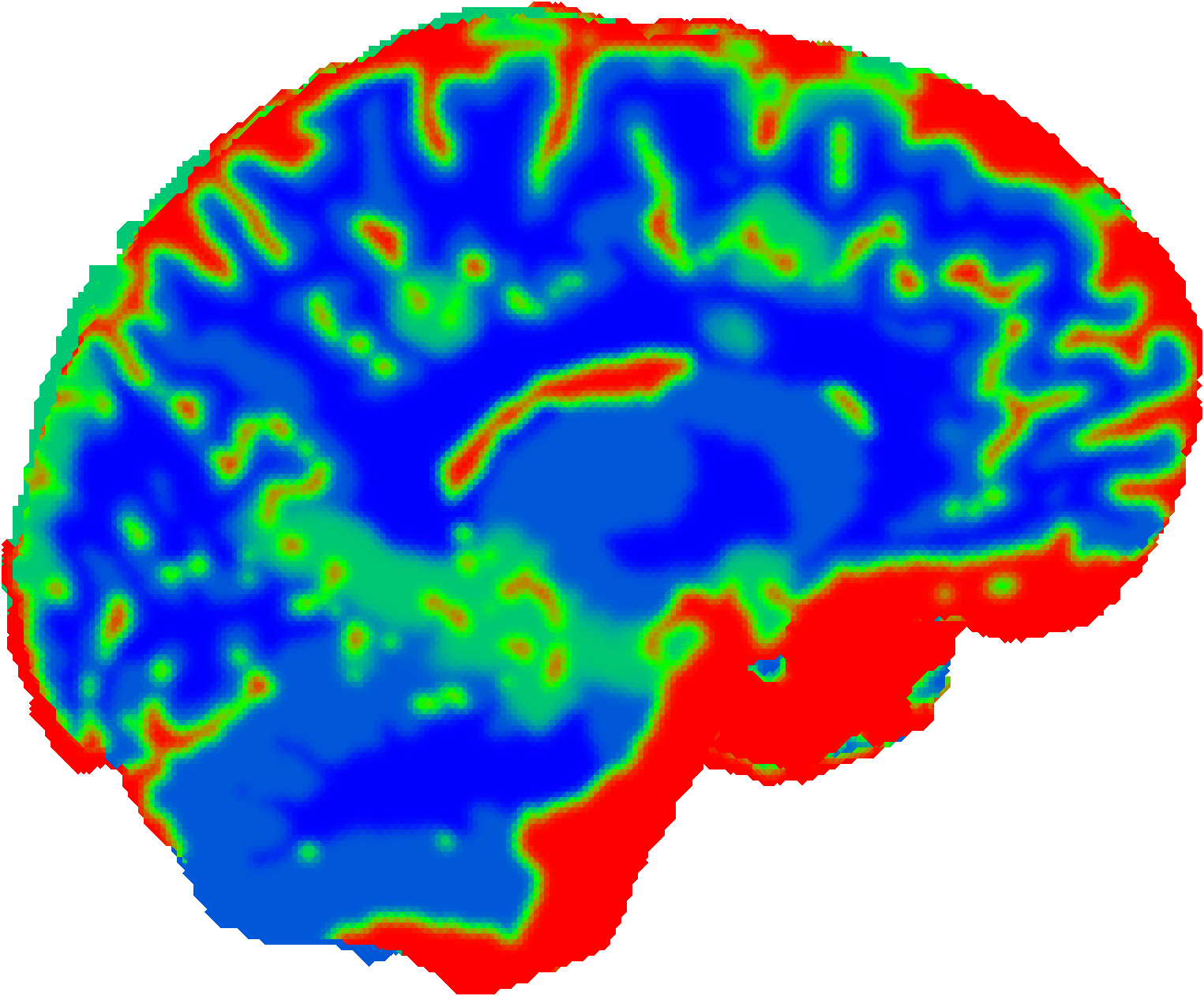} \\ Peak, sagittal
          \\ \vskip0.2cm
    \centering
        \includegraphics[width=2.0cm]{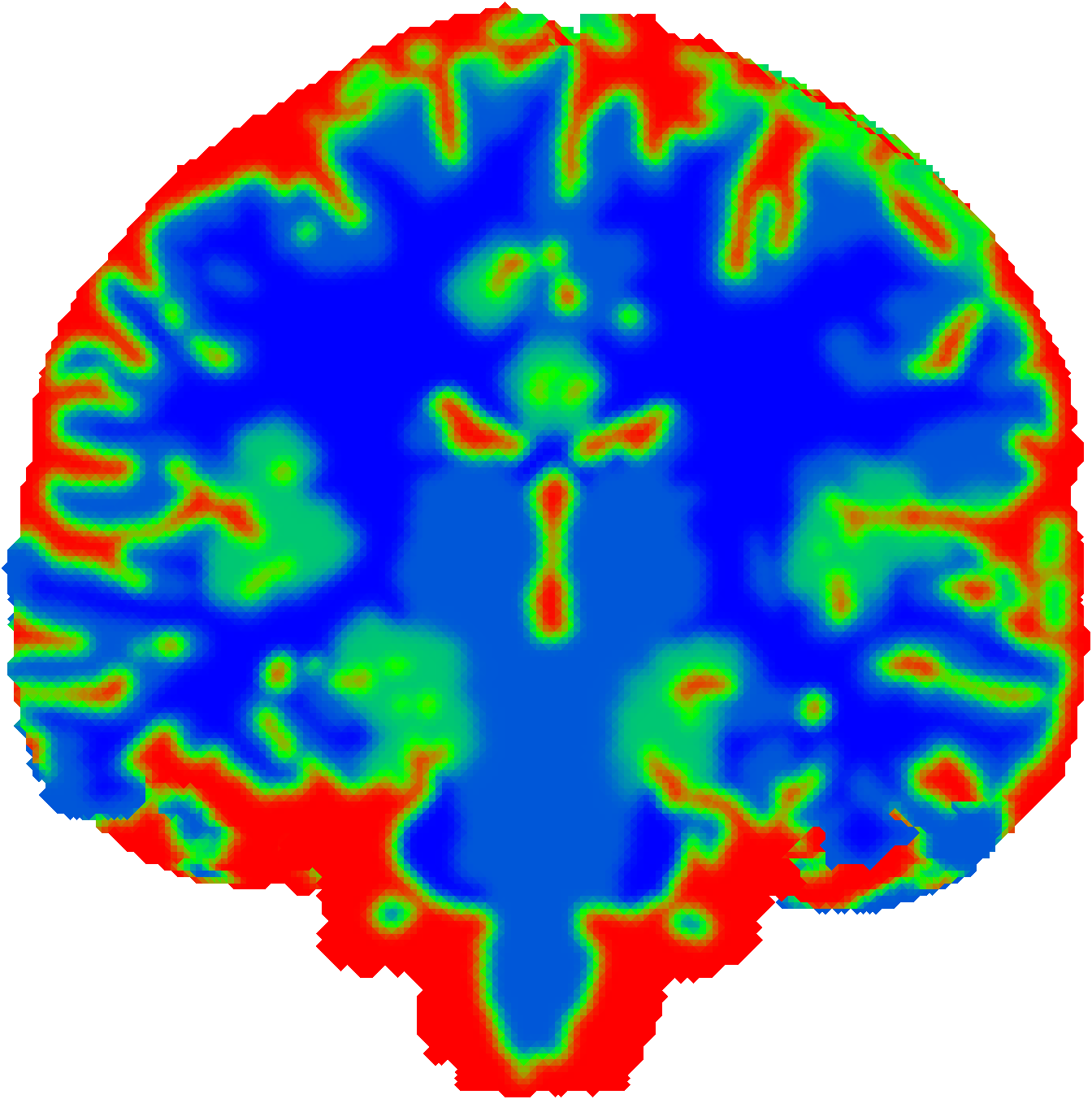} \\
        Peak, coronal
   \\ \vskip0.2cm
    \centering
        \includegraphics[width=1.7cm]{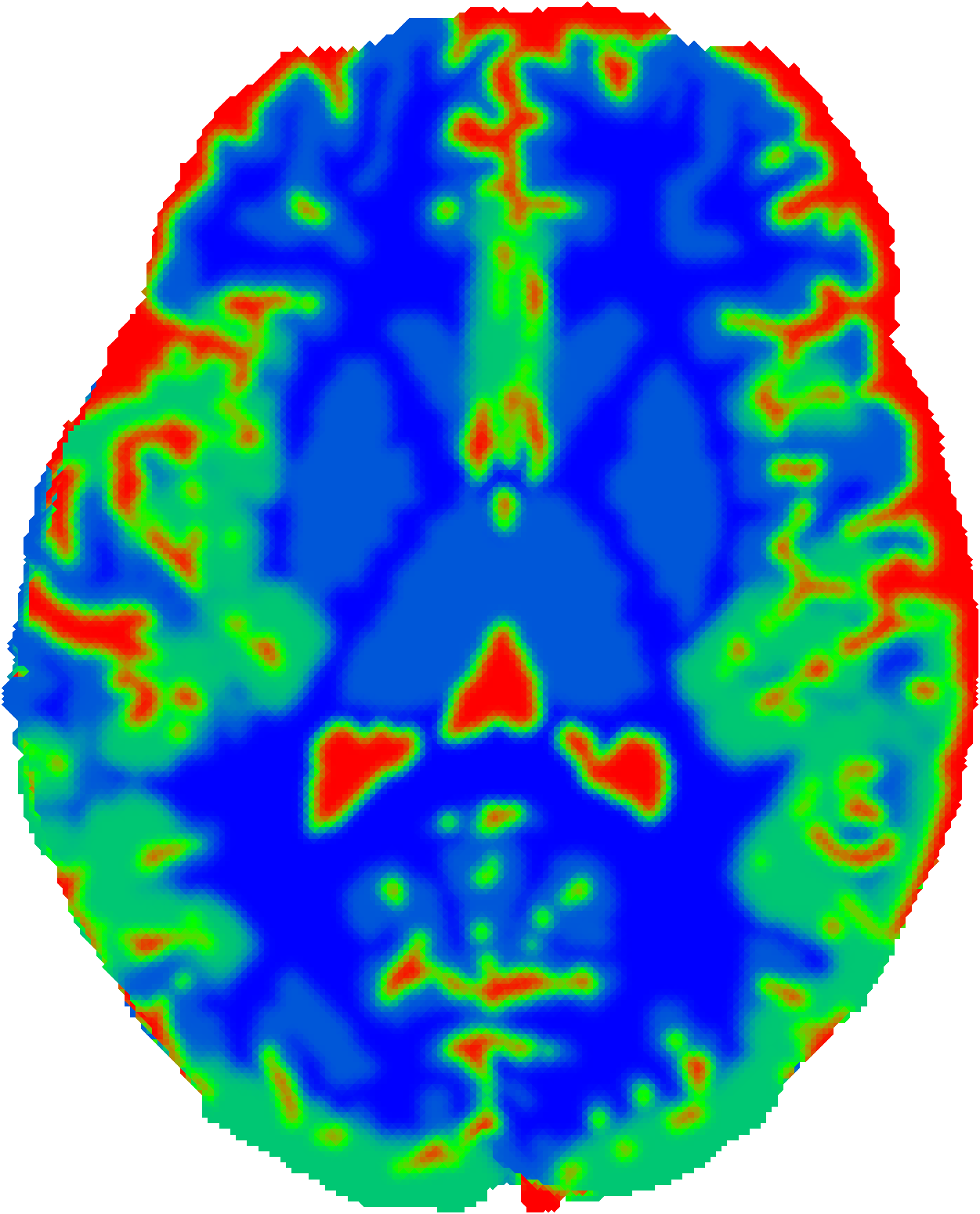} \\
        Peak, axial
        \end{minipage} 
       \begin{minipage}{0.5cm}
              \includegraphics[height=4.5cm]{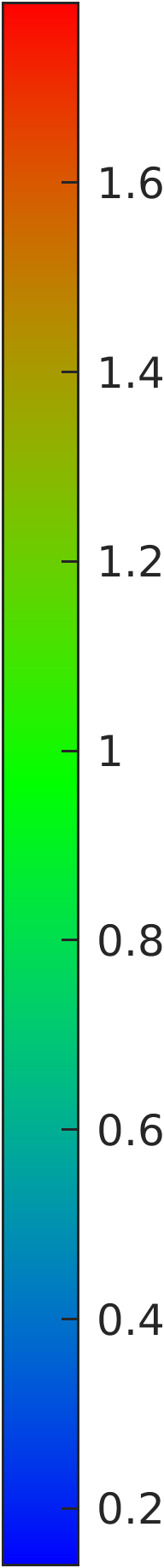}
              \end{minipage}
        \begin{minipage}{2.2cm}
    \centering
            \includegraphics[width=2.1cm]{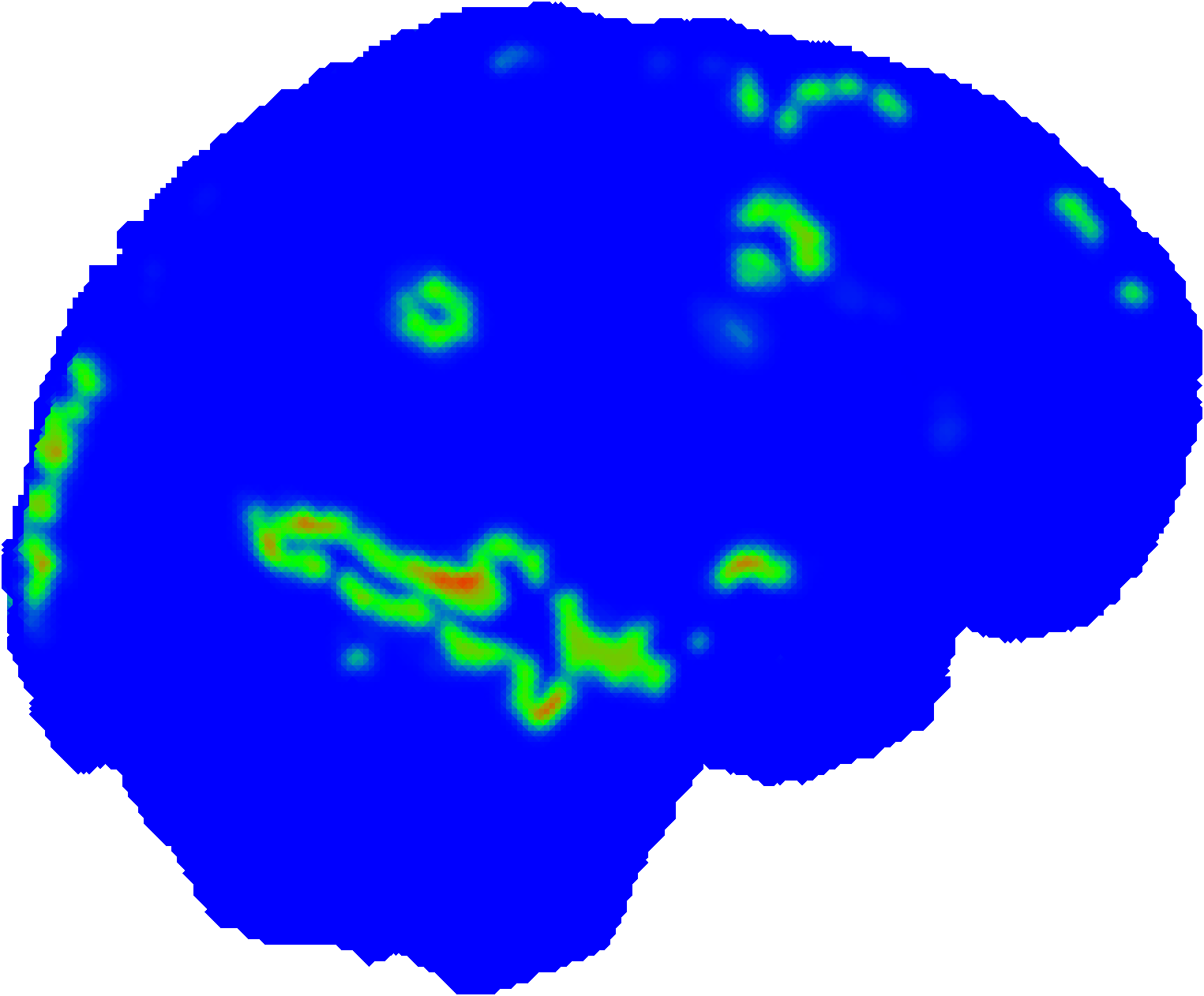} \\ STD, sagittal
         \\ \vskip0.2cm 
    \centering
        \includegraphics[width=2.0cm]{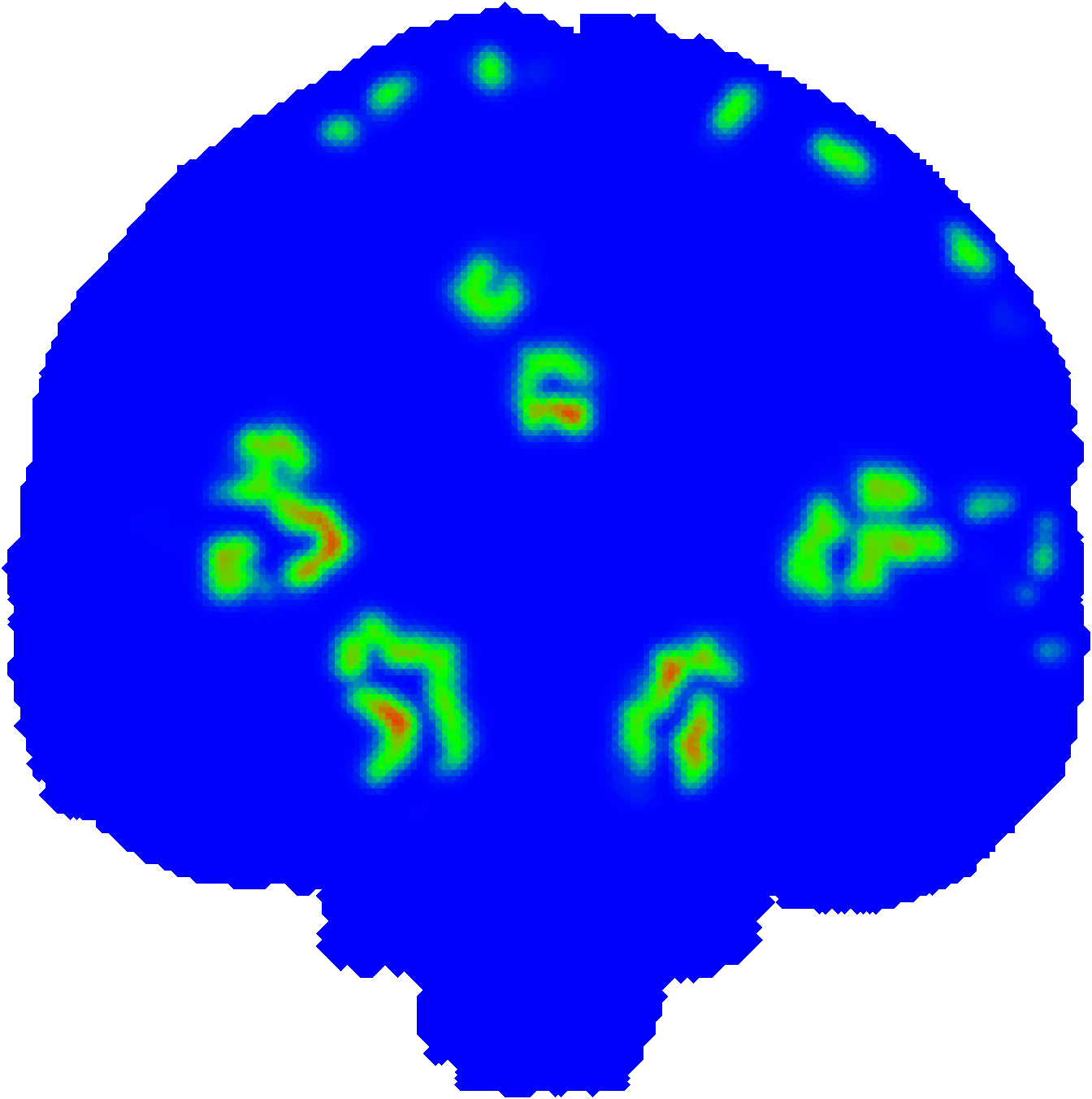} \\
        STD, coronal
      \\ \vskip0.2cm
    \centering
        \includegraphics[width=1.7cm]{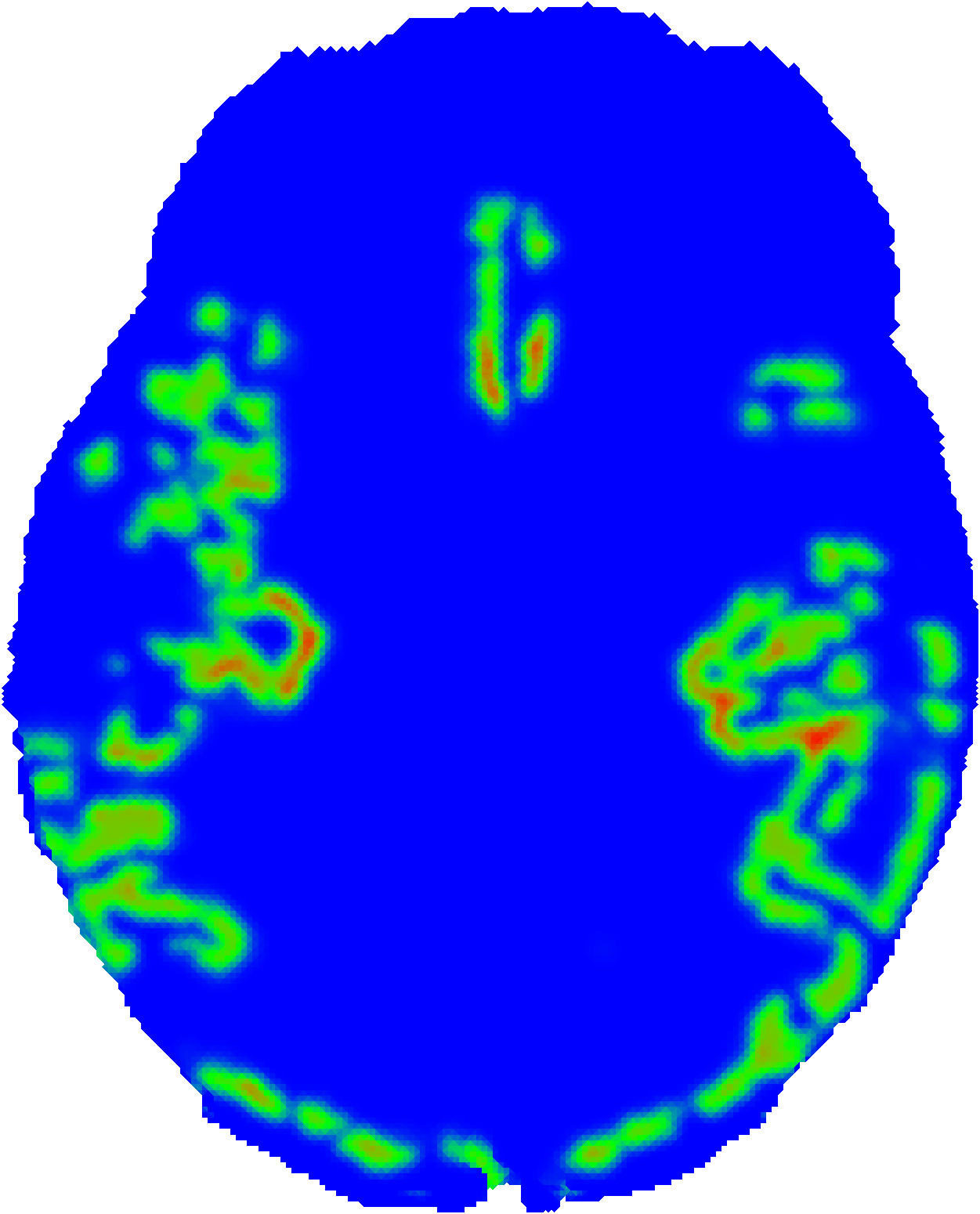} \\
        STD, axial
        \end{minipage} 
       \begin{minipage}{0.5cm}
              \includegraphics[height=4.5cm]{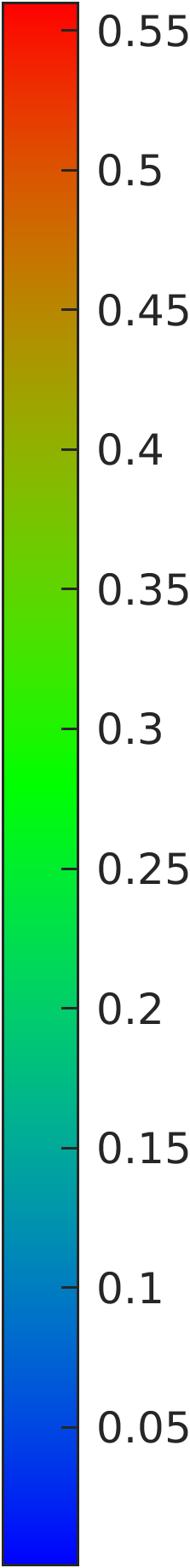}
              \end{minipage}
        \\    
    \end{footnotesize}
    \caption{Estimates of  tissue conductivity based on a piecewise constant background conductivity and estimated volumetric blood concentration in microcirculation for   60 bpm heart rate.}
    \label{fig:conductivity_60bpm}
\end{figure}

\begin{figure}[h!]
\begin{footnotesize}
    \centering
    \begin{minipage}{2.2cm}
    \centering
    \includegraphics[width=2.1cm]{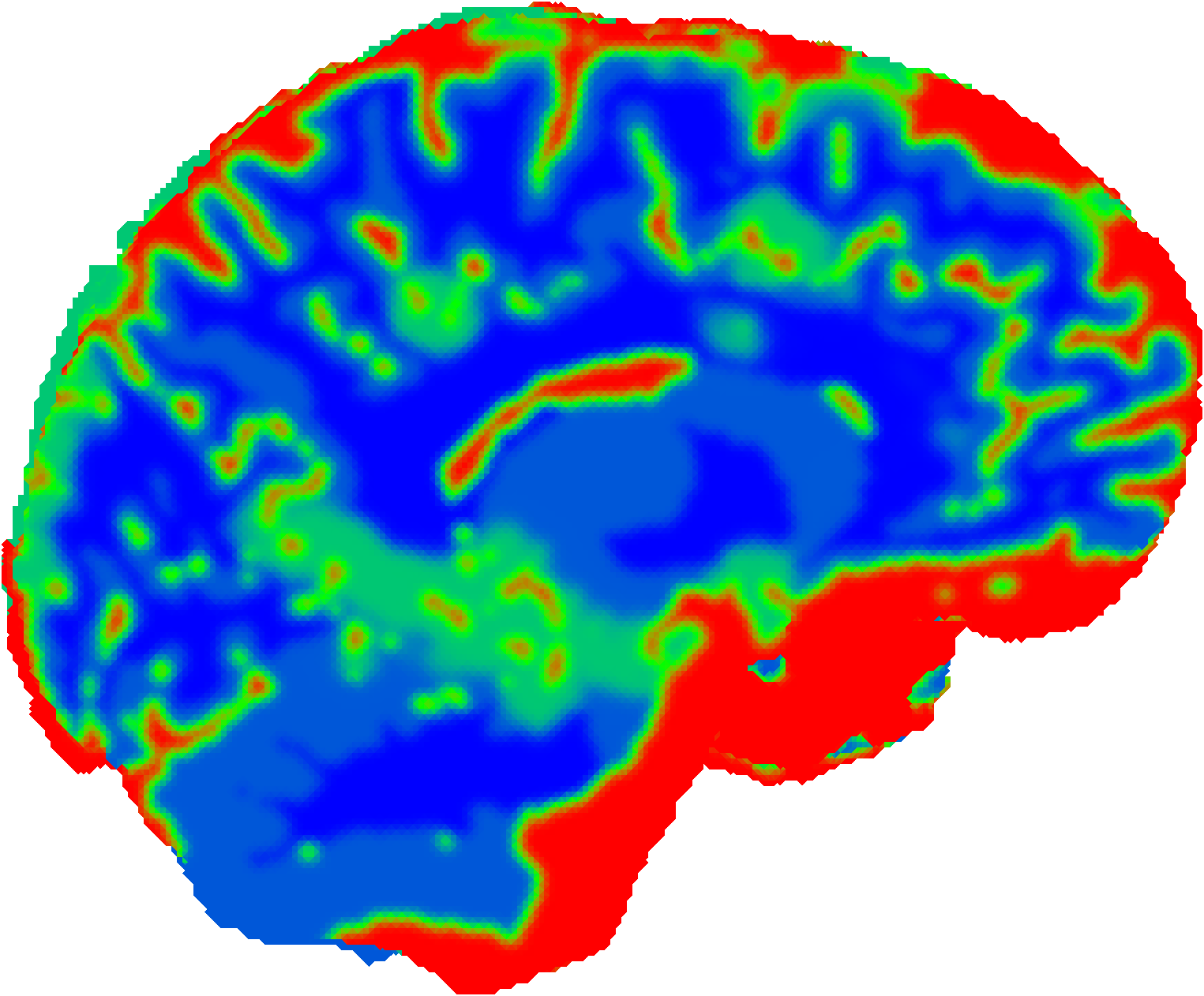} \\
    Mean, sagittal
\\ \vskip0.2cm
    \centering
        \includegraphics[width=2.0cm]{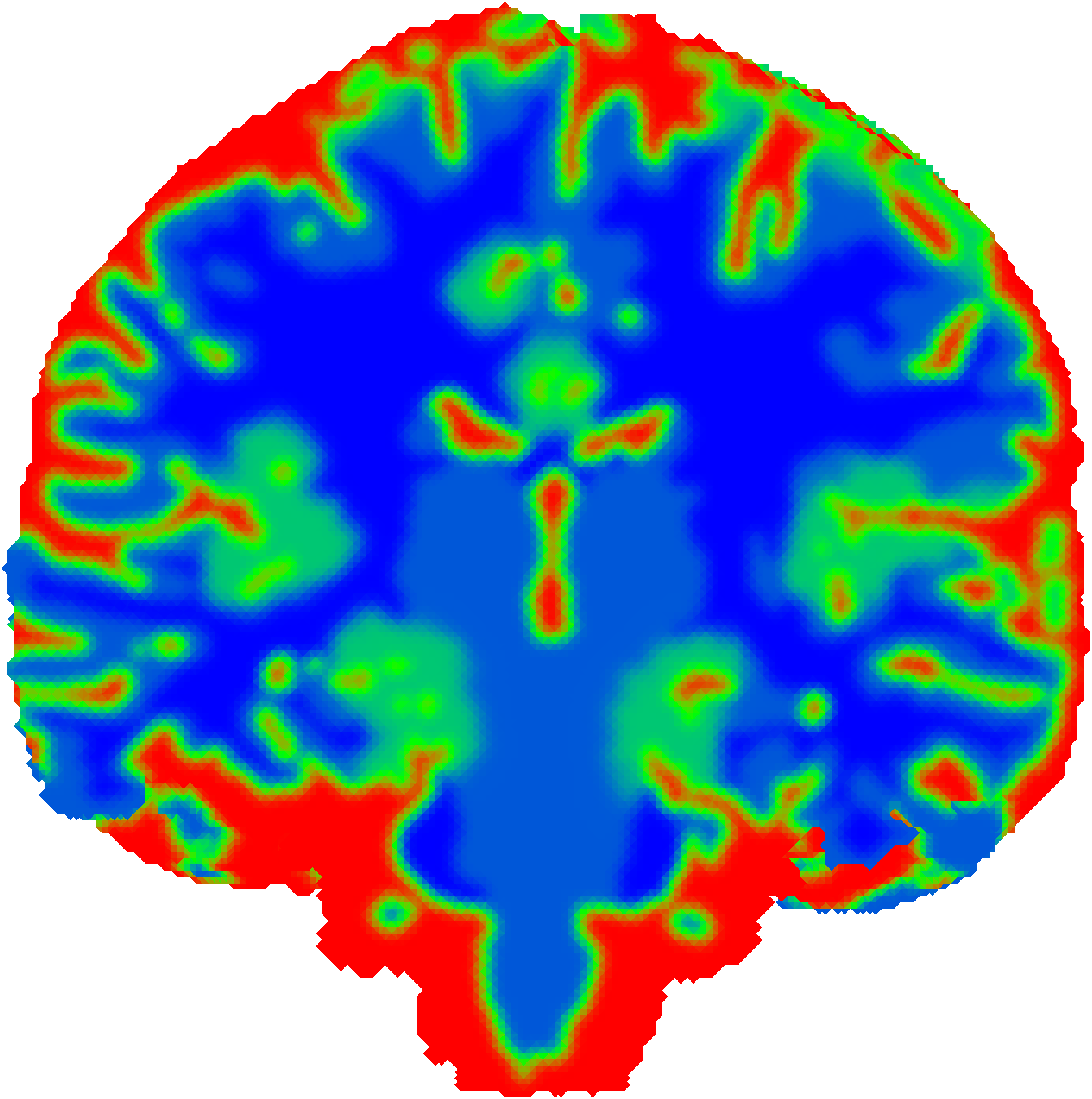} \\
        Mean, coronal 
       \\ \vskip0.2cm
    \centering
        \includegraphics[width=1.7cm]{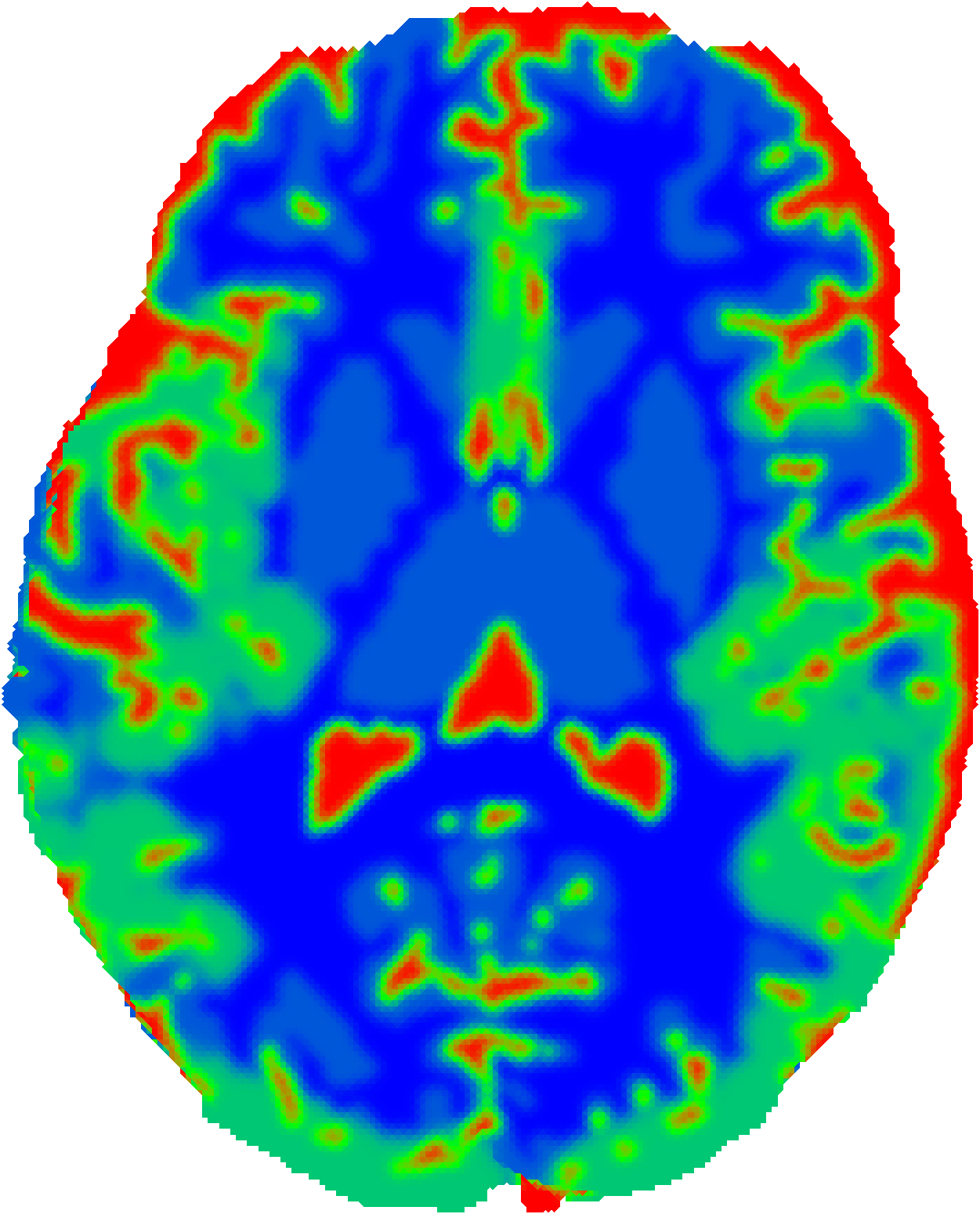} \\
      Mean, axial
        \end{minipage} 
        \begin{minipage}{0.5cm}
              \includegraphics[height=4.5cm]{mean_bar.png}
              \end{minipage}
 \begin{minipage}{2.2cm}
    \centering
            \includegraphics[width=2.1cm]{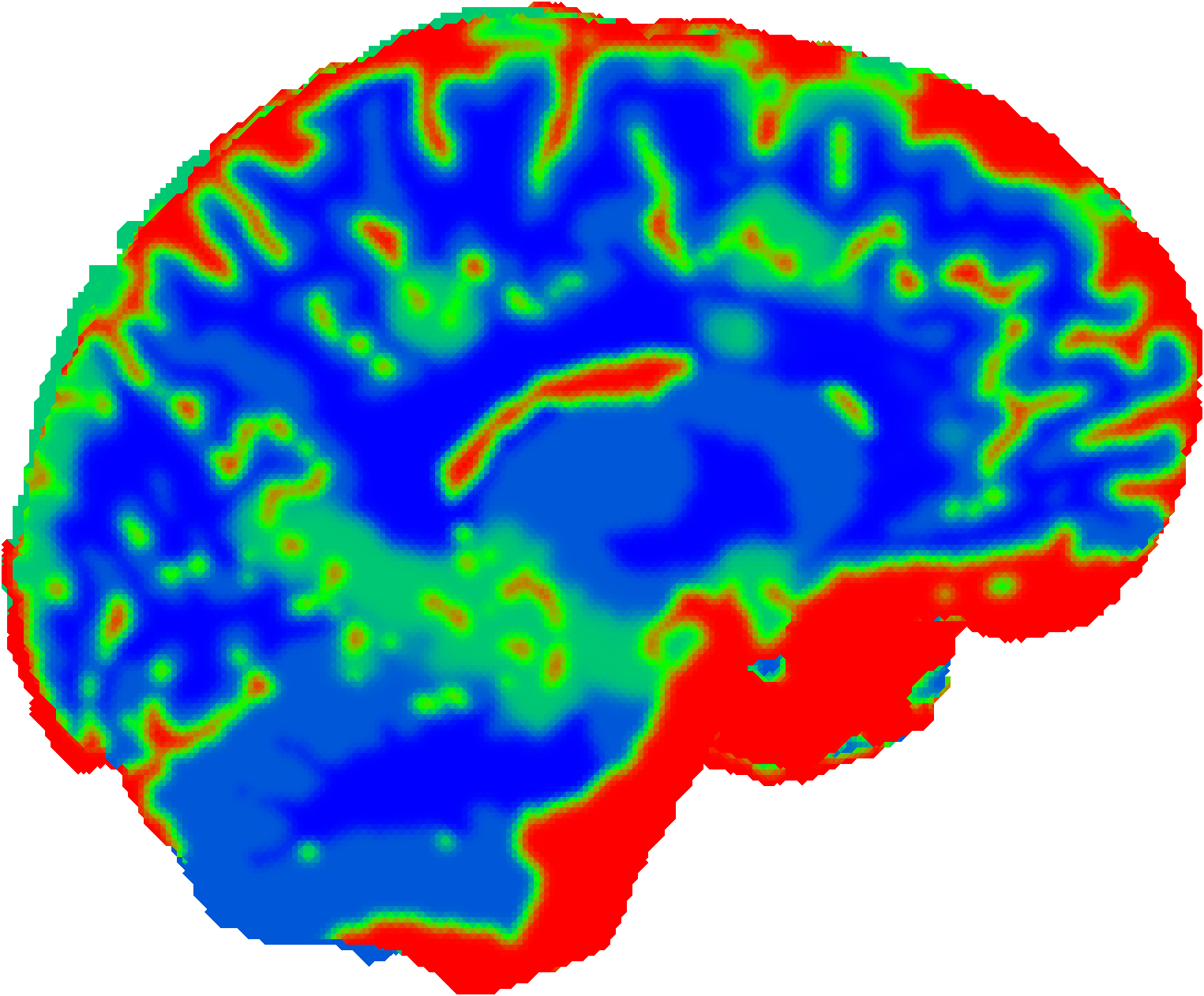} \\ Peak, sagittal
          \\ \vskip0.2cm
    \centering
        \includegraphics[width=2.0cm]{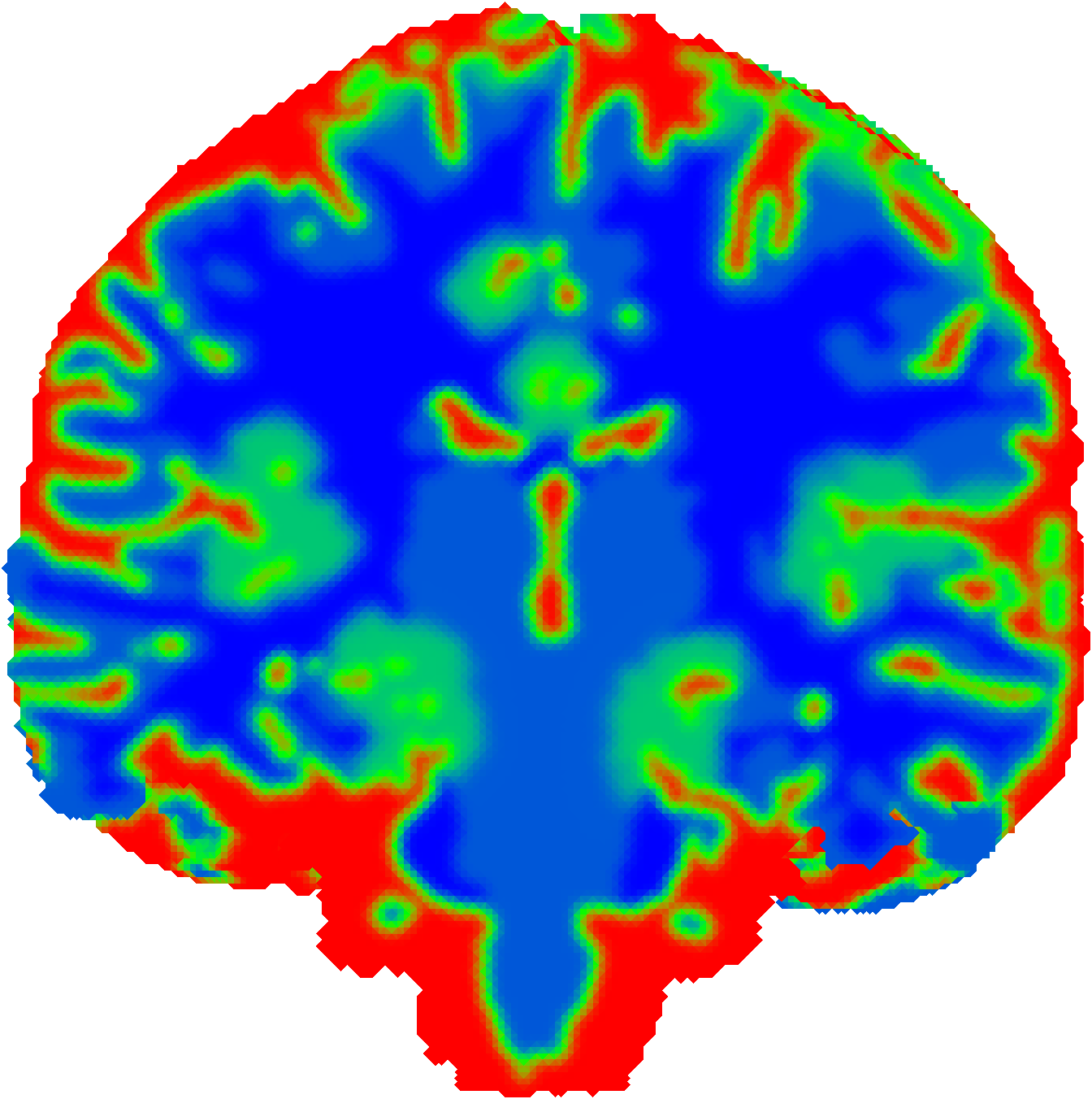} \\
        Peak, coronal
   \\ \vskip0.2cm
    \centering
        \includegraphics[width=1.7cm]{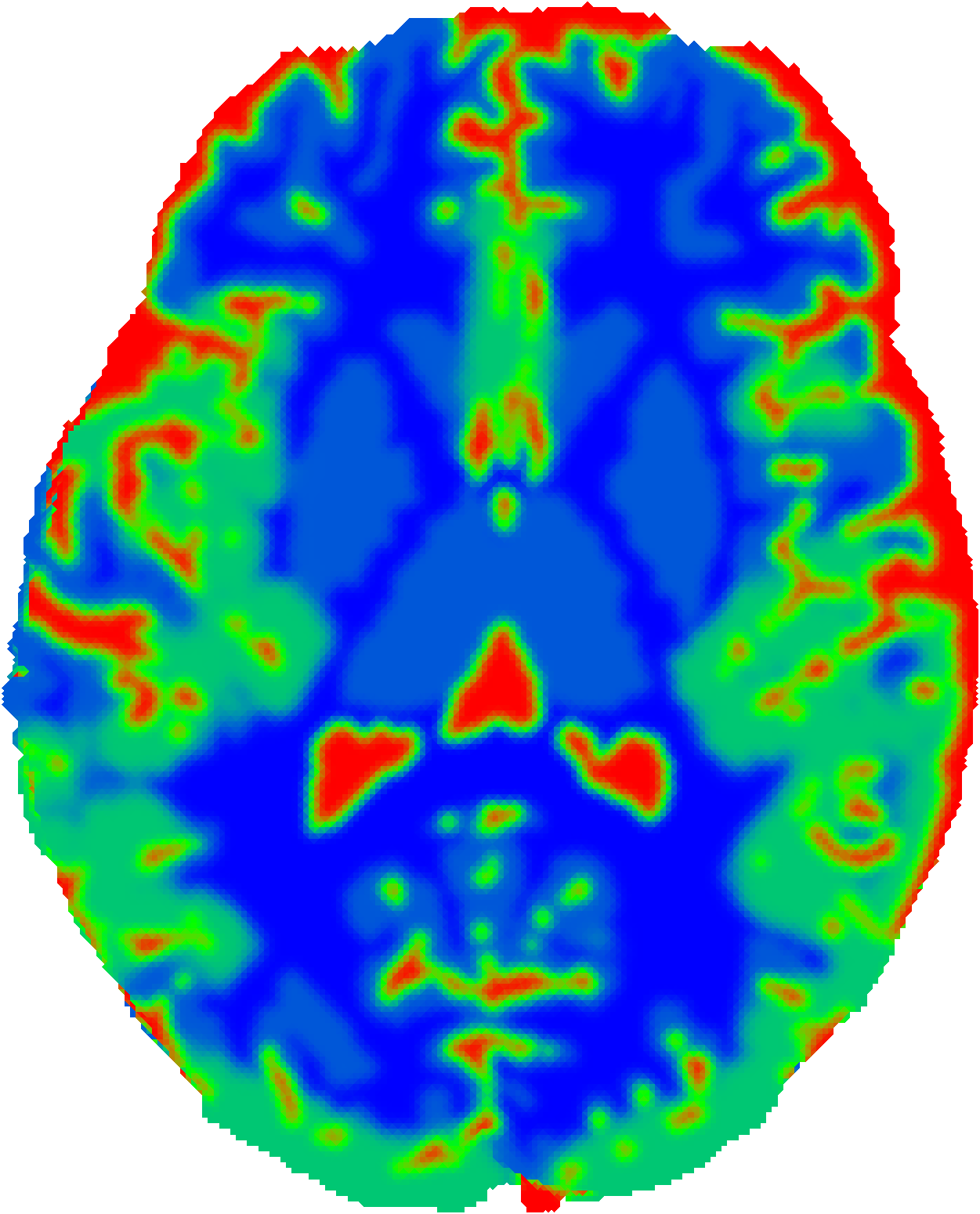} \\
        Peak, axial
        \end{minipage} 
       \begin{minipage}{0.5cm}
              \includegraphics[height=4.5cm]{max_bar.png}
              \end{minipage}
        \begin{minipage}{2.2cm}
    \centering
            \includegraphics[width=2.1cm]{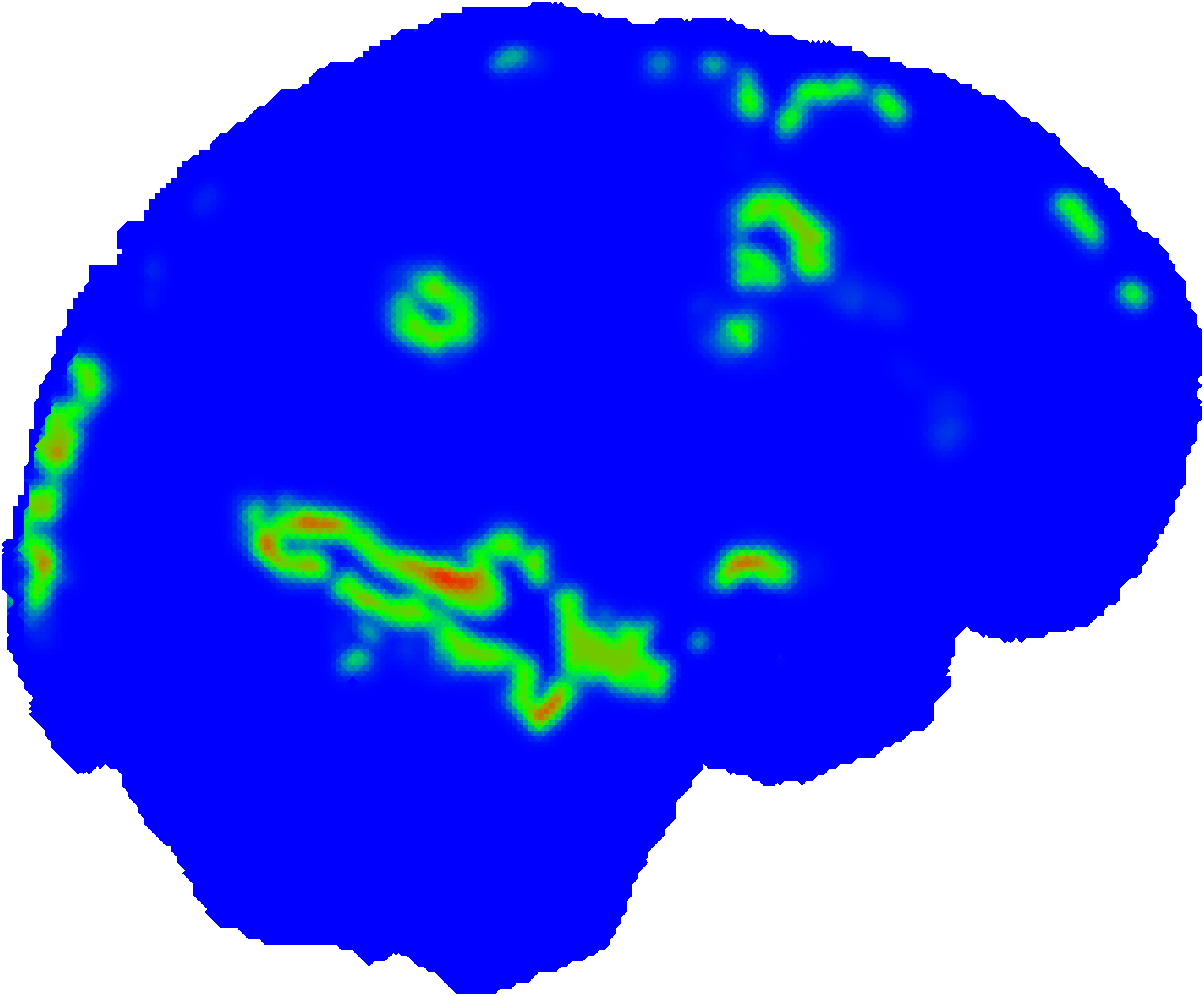} \\ STD, sagittal
         \\ \vskip0.2cm 
    \centering
        \includegraphics[width=2.0cm]{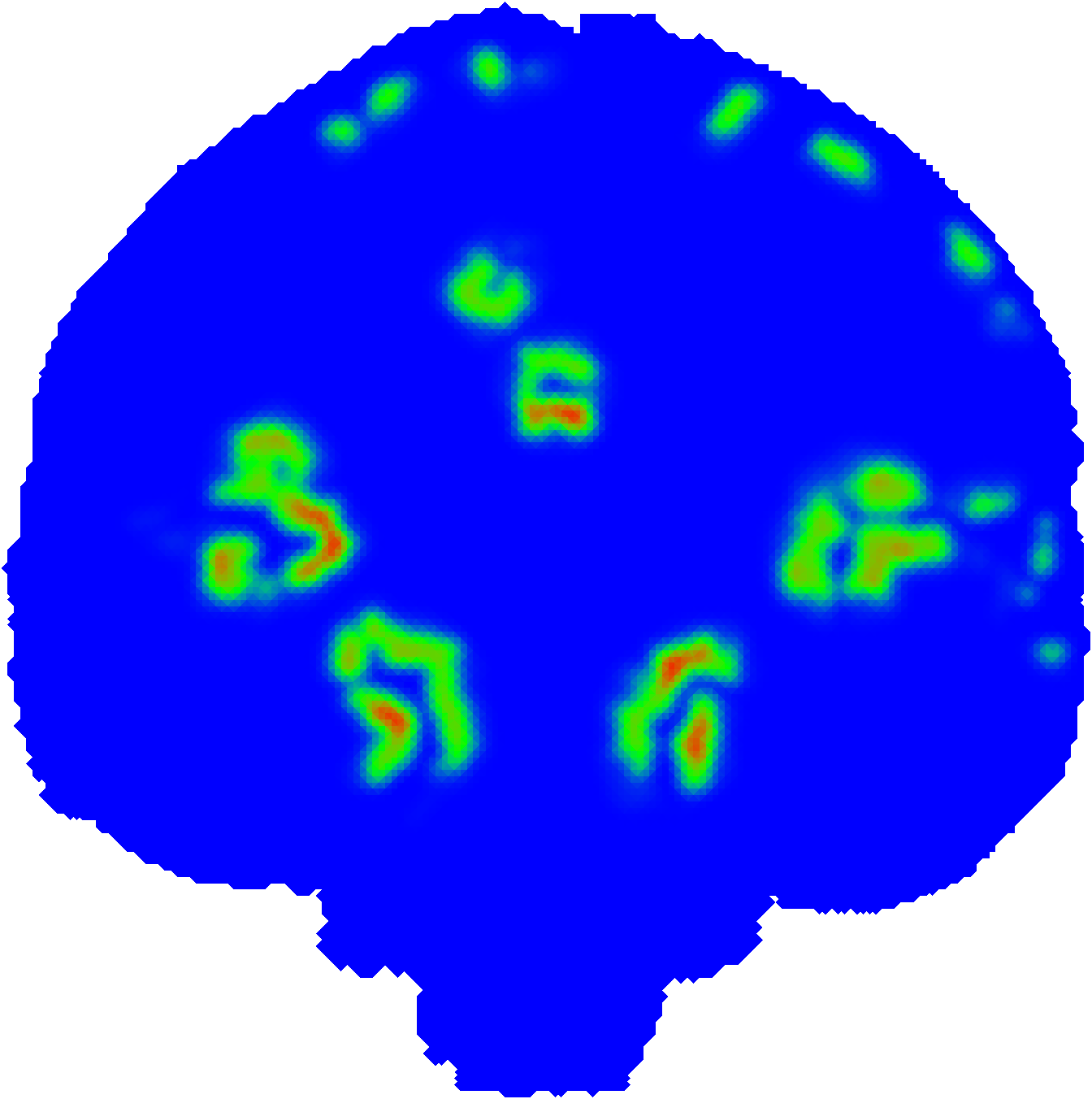} \\
        STD, coronal
      \\ \vskip0.2cm
    \centering
        \includegraphics[width=1.7cm]{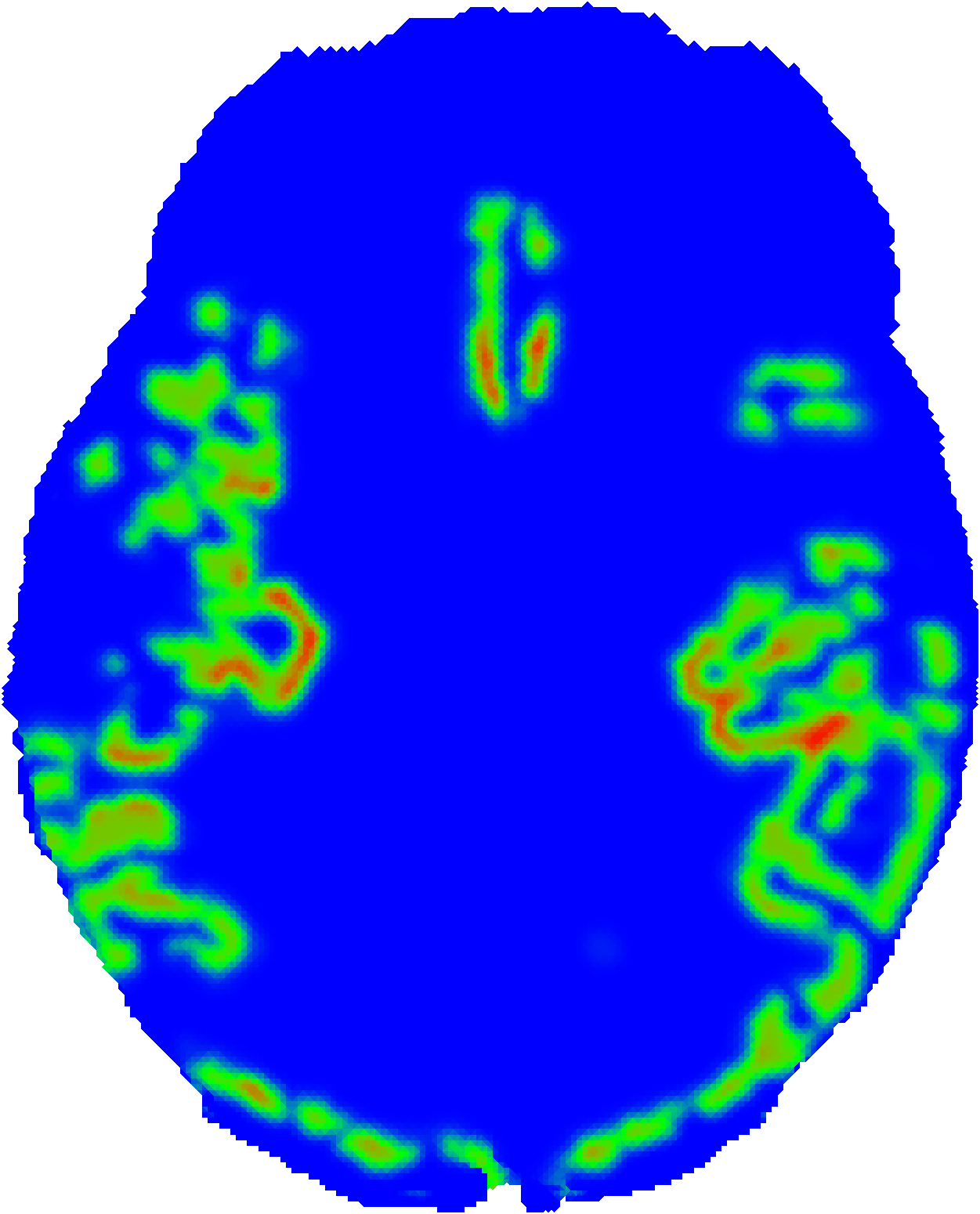} \\
        STD, axial
        \end{minipage} 
       \begin{minipage}{0.5cm}
              \includegraphics[height=4.5cm]{std_bar.png}
              \end{minipage}
        \\    
    \end{footnotesize}
    \caption{Estimates of  tissue conductivity for 80 bpm heart rate.}
    \label{fig:conductivity_80bpm}
\end{figure}

\section{Discussion}
\label{sec: Discussion}

The present study aimed at obtaining a dynamic estimate for brain conductivity given a piecewise constant background estimate \cite{dannhauer2010}, a rough arterial blood vessel segmentation $\Omega$, and a dynamical solution of NSEs together with Fick's law for microcirculation in $\hat{\Omega} $\cite{berg2020modelling, arciero2017mathematical}, and Archie's law of two-phase mixture conductivity \cite{j2001estimation,peters2005electrical,glover2000modified,cai2017electrical}. The motivation for the present study was that conductivity modelling has an important role in various electrophysiological modalities, including EEG  \cite{niedermeyer2004}, tES \cite{herrmann2013transcranial}, EIT \cite{cheney1999electrical}. In state-of-the-art applications, the dynamical features can be considered especially important in EIT applications monitoring dynamic conductivity distribution of the brain \cite{moura2021anatomical}. Recently, a dynamic conductivity model has been suggested to have an effect on EEG source localization \cite{lahtinen2023silico}, and dynamic electric field modelling in tES has been discussed in   \cite{alekseichuk2019electric}.

In order to make our approach applicable with limited {\em a priori} information about the blood flow as a boundary condition, we used a given pressure distribution since it can be easily approximated based on textbook non-invasive blood pressure measurements. An alternative approach would be to set the blood velocity in the incoming vessels, which might be approximated non-invasively, e.g., based on transcranial Doppler ultrasound \cite{lindegaard1987variations,wang2016peak,purkayastha2012transcranial} or whole-brain cerebral blood flow scans built upon MRI, positron emission tomography (PET), or single-photon emission computed tomography (SPECT) \cite{chen2018evaluation, liu2018resting, taber2005blood}. The two investigated boundary pressure distributions with a 60 and 80 bpm heart rate and a pulse pressure of 50 mmHg correspond to light physical activity, e.g., walking.  

A discretized non-Newtonian NSE system was solved by iterating the following two stages for each time step: (1) evaluation of a pressure estimate as the solution of a dynamic PPE \cite{samavaki2023ppe,pacheco2021continuous}, i.e., an equation following as the divergence of the original NSE; and (2) velocity field update by substituting the pressure field of the first stage into a Leray-regularized version of the original equation  \cite{leray1934mouvement}. It is well-known that regularization of NSEs is necessary to obtain a stable solution   \cite{pietarila2008three, guermond2004definition}. In addition, the viscosity field discretized using the linear Lagrangian basis was smoothed using a smoothing strategy similar to Leray regularization. That is, the resolution of the viscosity was made slightly lower than that of the pressure or the velocity field to establish the numerical stability of the solution.  The amount of smoothing was kept as low as possible; larger smoothing parameters led to numerical instability.

The results match well to the existing knowledge of blood pressure and velocity in cerebral blood circulation. For example, \cite{blanco2017blood} demonstrated a numerically simulated systolic normotensive blood pressure distribution that is close to the current peak estimates: 113 mmHg for the basilar artery, 110 mmHg for the distal medial striate artery, and 85 mmHg for the posterior parietal branch of the middle cerebral artery. The mean and peak velocity ranges of 0--0.26 m/s, 0--0.82 m/s for 60 bpm, and 0--0.53 m/s, 0--1.4 m/s for 80 bpm heart rate, respectively,  are in an appropriate agreement with the findings of experimental studies on cerebral blood velocity \cite{bouvy2016assessment,lindegaard1987variations,wang2016peak}. 

The mean viscosity is largely maintained in the interval 0.0035--0.0055 Pa s, which can be considered normal \cite{nader2019blood}. The peak viscosities are above the normal range, which can be expected to be the case for actual cerebral blood circulation too. Here the peak values can be interpreted to be at least partially due to the geometrical artifacts in the model, as the peak values are observed in regions corresponding to the grey matter where the individual arterial vessels are too thin to be reconstructed with the current MRI geometry. Due to this reason, those vessels were reconstructed as a clustered mass distribution.

Our distributional dynamic estimates for the volumetric blood concentration roughly match the experimental blood flow distributions obtained through MRI, PET, and SPECT \cite{chen2018evaluation, liu2018resting, taber2005blood}. The quickly decaying blood flow variation when moving away from the blood vessels is also a well-documented phenomenon \cite{caro2012mechanics}. The conductivity atlases obtained based on the concentration clearly show the estimated effect of microcirculation, including the time-variation. Comparing the present results with the recent dynamic conductivity estimates obtained with a cylindrical NSE model \cite{lahtinen2023silico,moura2021anatomical}, there is a match in the distribution and value-range of the conductivity estimates and their time-variation. Hence, we can conclude that our goal of modelling the conductivity distribution through dynamic NSEs was achieved, showing the feasibility of the task.

Compared to the approach of solving cylindrical NSEs  \cite{lahtinen2023silico,moura2021anatomical}, the importance of our present technique is its applicability to an arbitrary head segmentation with arterial vessels; whereas a solution for cylindrical NSEs needs to be extended and interpolated to cover the full computing domain, our method is, in principle, capable of finding the actual NSE-based blood flow for a given head segmentation. An obvious limitation of our approach is, however, its dependence on the 7-Tesla MRI data, which in this study was obtained from the open dataset of CEREBRUM-7T \cite{svanera2021cerebrum}. Applying our present method using datasets corresponding to a lower magnetic flux density might necessitate the embedding of an external arterial vessel structure relying on an intersubject average approximation akin to the interpolation strategy adopted in \cite{lahtinen2023silico,moura2021anatomical}.

As an obvious future work direction, we consider investigating the theoretical grounds of our present two-stage approach to solving NSEs, which combines a dynamical PPE with Leray-regularized NSEs. For example, analytical solutions and alternative regularization strategies may be involved. Another important direction is to examine the effect of the computing geometry on the results, for example, by performing a multi-subject study using our present approach or evaluating the two-stage NSE solution for a specific geometry, such as an individual vessel branch, to find its performance at different scales. Other alternatives might be to consider different boundary conditions in addition to the present pressure condition and to investigate the performance of alternative viscosity models in addition to the Carreau-Yasuda model.


\appendix
\section{Background Theory}
\label{B-H}
We define the Bochner and Hodge Laplacians for any vector field ${\bf u}$ as follows
\begin{align*}    
     \Delta_B{\bf u}&=(\delta d+d\delta){\bf u}\,,
      \\
       \Delta_H{\bf u}&=-\sharp (\delta d+d\delta)\flat {\bf u}\, .
\end{align*}
Furthermore, we define the (2,0)-tensors $ {\bf Su}$ and $ {\bf Au}$ as follows
\begin{equation}
\begin{aligned}
 {\bf Su}&:=\nabla {\bf u}+(\nabla {\bf u})^T 
 \\ 
  {\bf Au}&:=\nabla {\bf u}-(\nabla {\bf u})^T 
\end{aligned}   
\label{harmo}
\end{equation}
The following definitions are given
\begin{equation}
\begin{aligned}
{\bf L}{\bf u}&:=\mathsf{div}({\bf S}{\bf u})=\Delta_B {\bf u}+\mathsf{grad}(\mathsf{div}({\bf u}))+{\bf Ri}({\bf u})
\\
{\bf L_{A}}{\bf u}&:=\mathsf{div}({\bf A}{\bf u})=\Delta_H {\bf u}-\mathsf{grad}(\mathsf{div}({\bf u}))
\end{aligned}   
\label{Lu-Lu}
\end{equation}
where $\Delta_H {\bf u}=\Delta_B {\bf u}-{\bf Ri}({\bf u})$. 

Note that $\nabla\times {\bf u}$ and $\Box\times {\bf u}$ are curl and cross products of the vector field ${\bf u}$, respectively, and that the following equations give their representations in coordinates due to the components of the Levi-Civita symbol:
\begin{align*}
(\nabla\times {\bf u})^k&=\varepsilon^{ijk}g_{jh}u^h_{;i}
\\
(\Box\times {\bf u})^{k}&=\varepsilon_{ijh}\,g^{kh} u^i \Box^j  
\end{align*}

In 3D, for any smooth function $f\in C^{\infty}(\Omega)$ and two vector fields ${\bf u}$ and ${\bf v}$ in $\tau(\Omega)$, the following lemmas hold.
\setcounter{lemma}{0}
\renewcommand{\thelemma}{\Alph{section}\arabic{lemma}}
\begin{lemma}
\begin{align*}
  {\bf g}\big(\nabla\times {\bf u}, \nabla\times(f\,{\bf v})\big)= {\bf g}\big({\bf Au}(\nabla f), {\bf v}\big)+f({\bf Au} : \nabla{\bf v})\,,
\end{align*}
where 
\[
  {\bf Au}:\nabla{\bf v}={\bf g}(\nabla{\bf u},\nabla{\bf v})-(\nabla{\bf u}:\nabla{\bf v})\,.
\]
If the vector field $\bf v$ can be represented as a potential, such that there exists a smooth function $q$ satisfying $\nabla q={\bf v}$, then the given integral takes the specific form:
\begin{align*}
  {\bf g}\big(\nabla\times {\bf u}, \nabla\times(f\,\nabla q)\big)= {\bf g}\big({\bf Au}(\nabla f), \nabla q\big)\,.
\end{align*}
\label{lemma1}
\end{lemma}
\begin{proof}
First, it is straightforward to verify that
\[
{\bf g}\big(\nabla\times {\bf u}, \nabla\times(f\,{\bf v})\big)={\bf g}(\nabla\times {\bf u}, {\bf v}\times\nabla f)+ f\,{\bf g}(\nabla\times {\bf u}, \nabla\times {\bf v})
\]
Expressed in coordinates, the equation is as follows:
\begin{align*}
    {\bf g}(\nabla\times {\bf u}, {\bf v}\times\nabla f)&=\varepsilon^h_{k\ell}\varepsilon^{ij\ell}f_{;h}g^{\alpha k} v_{\alpha}g_{j\beta}u^{\beta}_{;i}
    =f_{;i}v_{j}(g^{ih}u^{j}_{;h}-g^{jh}u^{i}_{;h})
    \\
    &={\bf g}\big({\bf Au}(\nabla f), {\bf v}\big)\,.
\\
   {\bf g}(\nabla\times {\bf u}, \nabla\times {\bf v})&= \varepsilon^h_{k\ell}\varepsilon^{ij\ell}g^{\alpha k} v_{\alpha;h}g_{j\beta}u^{\beta}_{;i}=v_{j;i}(g^{ih}u^{j}_{;h}-g^{jh}u^{i}_{;h})
   \\
   &=
   {\bf Au} : \nabla{\bf v}\,,
\end{align*}
where
\[
  {\bf Au}:\nabla{\bf v}={\bf g}(\nabla{\bf u},\nabla{\bf v})-(\nabla{\bf u}:\nabla{\bf v})=g^{hj}u^{i}_{;h}g_{ik}v^{k}_{;j}-u^{j}_{;i}v^{i}_{;j}\,.
\]
\end{proof}
\setcounter{lemma}{1}
\renewcommand{\thelemma}{\Alph{section}\arabic{lemma}}
\begin{lemma}
\begin{align*}
\int_{\Omega}
 \mathsf{div}\big(f\,{\bf Au}({\bf v})\big)\,\mathrm{d} \omega_\Omega &= \int_{\partial\Omega}
 {\bf g}( f\,\nabla\times {\bf u}, {\vec{{\bf n}}} \times {\bf v}) \,\mathrm{d}\omega_{\partial \Omega}\,.
\\
\int_{\Omega} \mathsf{div}\big(f\,{\bf Su}({\bf v})\big)\,\mathrm{d} \omega_\Omega &= \int_{\Omega} \mathsf{div}\big(f\,{\bf Au}({\bf v})\big)\,\mathrm{d} \omega_\Omega
\\
&+2\int_{\partial\Omega}{\bf g}\big(f \,(\nabla {\bf u})^T({\bf v}), {\vec{{\bf n}}}\big ) \,\mathrm{d}\omega_{\partial \Omega}\,.
\end{align*}
\label{lemma2}
\end{lemma}
\begin{proof}
The proof of the first integral follows directly from the definition. The second integral can be determined by applying the formula \eqref{harmo}: 
\[
 {\bf Su}({\bf v})={\bf Au}({\bf v})+2(\nabla {\bf u})^T({\bf v})\,.
\]  
\end{proof}
\\
Here, ${\vec{{\bf n}}}$ denotes the unit normal vector that is orthogonal to the surface. 
\setcounter{lemma}{2}
\renewcommand{\thelemma}{\Alph{section}\arabic{lemma}}
\begin{lemma}
The following integral is valid for the vector field ${\bf u}$ with the constraints $\mathsf{div}({\bf u})=0$ and ${\bf Ri}({\bf u})=0$
\begin{align*}
&\int_{\Omega} {\bf g}(\nabla\times(\nabla\times{\bf u}), f{\bf v}\big)\,\mathrm{d} \omega_\Omega=\int_{\Omega} {\bf g}\big(\nabla\times{\bf u}, \nabla\times(f{\bf v}))\,\mathrm{d} \omega_\Omega
\\
&+\int_{\Omega}f\,({\bf Au}:\nabla{\bf v})  \,\mathrm{d} \omega_\Omega-\int_{\partial\Omega}{\bf g}(f\,\nabla\times{\bf u}, {\vec{{\bf n}}}\times{\bf v})\,\mathrm{d}\omega_{\partial \Omega}\,.
\end{align*}
The integral takes as follows in the specific form:
\begin{align*}
\int_{\Omega} {\bf g}(\nabla\times(\nabla\times{\bf u}), f\,\nabla q\big)\,\mathrm{d} \omega_\Omega &=\int_{\Omega} {\bf g}\big(\nabla\times{\bf u}, \nabla\times(f\,\nabla q))\,\mathrm{d} \omega_\Omega
\\
&-\int_{\partial\Omega}{\bf g}(f\,\nabla\times{\bf u}, {\vec{{\bf n}}}\times\nabla q)\,\mathrm{d}\omega_{\partial \Omega}\,.
\end{align*}
\label{curl-au}
\end{lemma}
\begin{proof}
By applying the \emph{Ricci Identity} and the \emph{Bianchi Identity} \cite{petersen}, we obtain the following expression:
\[
  u^h_{;ih}-u^h_{;hi}=u^{h} R_{ih}
\]
which, in conjunction with the assumptions provided by the lemma and applying lemma \ref{lemma1}, completes the proof.

\end{proof}



\section{Pulse wave equation}
\label{pulse_wave_equation}

Denoting the boundary normal derivative of $p$ by 
\begin{equation} 
  \frac{\partial p}{\partial \vec{\bf n}}  = {\bf g}(\nabla p, {\vec{{\bf n}}})\,,  
\end{equation}
equation (\ref{eqn:line-2.4}) set on the boundary $\partial \Omega$ can be expressed as 
\begin{equation}
  \frac{\partial p}{\partial \vec{\bf n}} = - \frac{1}{\nu}  \,  \left( \frac{\partial p}{\partial t} - \frac{\partial p^{\mathcal{(B)}}}{\partial t} \right) \quad \hbox{with} \quad \nu = \frac{(1+\beta^2) \,  \overline{Q}}{\zeta \, \lambda \,  \beta^2  \, \overline{V}}
\end{equation}
It follows that 
\[
 \frac{\partial }{\partial t} \frac{\partial p}{\partial \vec{\bf n}}  = - \frac{1}{\nu} \,  \left(\frac{\partial^2 p}{\partial t^2} - \frac{\partial^2  p^{\mathcal{(B)}}}{\partial t^2} \right)\,,
\]
and 
\[
\frac{\partial }{\partial \vec{\bf n}} \frac{\partial p}{\partial t} = \frac{\partial }{\partial \vec{\bf n}}   \left( -\nu \frac{\partial p}{\partial \vec{\bf n} } + \frac{\partial  p^{\mathcal{(B)}}}{\partial t} \right) \,.
\]
Further the pressure $p$ satisfies the following wave equation on the boundary $\partial \Omega$:
\begin{equation}
    \frac{1}{\nu^2} \,  \frac{\partial^2 p}{\partial t^2}   -  \frac{\partial^2 p}{\partial \vec{\bf n}^2 } = \frac{1}{\nu^2} \,  \frac{\partial^2  p^{\mathcal{(B)}}}{\partial t^2} + \frac{1}{\nu} \, \frac{\partial}{\partial \vec{\bf n}} \frac{\partial  p^{\mathcal{(B)}}}{\partial t}\, . 
\end{equation}
\begin{equation}
    \frac{1}{\nu^2} \,  \frac{\partial^2 p}{\partial t^2}   -  \frac{\partial^2 p}{\partial \vec{\bf n}^2 } = \frac{1}{\nu^2} \,  \frac{\partial^2  p^{\mathcal{(B)}}}{\partial t^2} - \frac{1}{\nu} \, \frac{\partial}{\partial \vec{\bf n}} \frac{\partial  p^{\mathcal{(B)}}}{\partial t}\, . 
\end{equation}
The velocity of the waveform solution in the above equation is denoted by $\nu$, with the right-hand side terms serving as sources.

\section*{Acknowledgements}
The work of Maryam Samavaki and Sampsa Pursiainen is supported by the Academy of Finland Centre of Excellence (CoE) in Inverse Modeling and Imaging 2018–2025 (decision 336792) and project 336151; Santtu Söderholm's work has been funded by the ERA PerMED (PerEpi) project 344712; and Arash Zarrin Nia has been funded by a scholarship from the K. N. Toosi University of Technology.

\bibliographystyle{elsarticle-num} 
\bibliography{cas-refs}





\end{document}

%% file: preamble.tex
\usepackage{xparse}

\NewDocumentCommand\microdomain{}{\hat\Omega}